\documentclass[10pt,leqno]{amsart}
\usepackage{amssymb}
\usepackage{graphicx}
\usepackage[colorlinks=true]{hyperref}
\hypersetup{urlcolor=blue, citecolor=red}
\usepackage{mathrsfs} 
\usepackage{tikz-cd} 
\usetikzlibrary{decorations.pathmorphing} 
\usepackage{blkarray} 

\newtheorem{theorem}{Theorem}[section]

\newtheorem{lemma}[theorem]{Lemma}
\newtheorem{proposition}[theorem]{Proposition}
\theoremstyle{definition}
\newtheorem{remark}[theorem]{Remark}
\newtheorem{assumption}{Assumption}

\newcommand{\IN}{\,\mathrm{in}}
\newcommand{\OUT}{\,\mathrm{out}}
\newcommand{\zmin}{z_{\min}}
\newcommand{\zmax}{z_{\max\phantom{i}\!\!\!}}
\newcommand{\Ahat}{\;\widehat{\phantom{A}}\hspace{-.9em}\mathscr{A}}
\newcommand{\Bhat}{\;\widehat{\phantom{B}}\hspace{-.8em}\mathscr{B}}
\newcommand{\Atilde}{\;\widetilde{\phantom{A}}\hspace{-.9em}\mathscr{A}}

\newcommand{\A}{\mathscr{A}}
\newcommand{\B}{\mathscr{B}}

\graphicspath{{./graphics/},{./images/},{./graphics/CoEvolution/},{./graphics/TradeOff/},{./graphics/PreySwitching/},{./graphics/ChainQ/},{./graphics/Transition/},{./graphics/RO_2D/}}

\title[\uppercase{Relaxation Oscillation and Entry-Exit Function}]{%
Relaxation Oscillations and the Entry-Exit Function
in Multi-Dimensional Slow-Fast Systems}

\thanks{This research was
partially supported by the Fuqua Research Assistant Professorship and
the National Science Foundation.}

\author[Ting-Hao Hsu and Shigui Ruan]{}

\date{\today}

\begin{document}


\begingroup
\def\uppercasenonmath#1{\Large} 
\maketitle
\endgroup

\centerline{Ting-Hao Hsu$^{\dag}$ and Shigui Ruan$^{\dag}$}
\medskip
{\footnotesize
\centerline{$^\dag$Department of Mathematics, University of Miami, Coral Gables, FL, USA}
}

\begin{abstract}
For a slow-fast system of the form
$\dot{p}=\epsilon f(p,z,\epsilon)+h(p,z,\epsilon)$, $\dot{z}=g(p,z,\epsilon)$
for $(p,z)\in \mathbb R^n\times \mathbb R^m$,
we consider the scenario that 
the system has invariant sets $M_i=\{(p,z): z=z_i\}$, $1\le i\le N$,
linked by a singular closed orbit
formed by trajectories of the limiting slow and fast systems.
Assuming that the stability of $M_i$ changes
along the slow trajectories
at certain turning points,
we derive criteria for the existence and stability of 
relaxation oscillations
for the slow-fast system.
Our approach is based on a generalization
of the entry-exit relation
to systems with multi-dimensional fast variables.
We then apply our criteria to several predator-prey systems with
rapid ecological evolutionary dynamics
to show the existence of relaxation oscillations
in these models.
\end{abstract}

\section{Introduction}
\label{sec_intro}
We consider a system
of ordinary differential equations
for $(p,z)\in \mathbb R^n\times \mathbb R^m$
of the form \begin{equation}\label{deq_pz}\begin{aligned}
  &\dot{p}= \epsilon f(p,z,\epsilon)+ h(p,z,\epsilon),
  \\
  &\dot{z}= g(p,z,\epsilon),
\end{aligned}\end{equation}
where $\cdot$ denotes $\frac{d}{d t}$,
the functions $f$, $g$ and $h$ are smooth,
and $\epsilon>0$ is a parameter.
This system is a generalization of
the classical slow-fast systems
in Fenichel~\cite{Fenichel:1979},
where the term $h$ was absent.
In the scenario
that $g$ and $h$
both vanish
on some level sets
${M}_i= \{(p,z):z=z_i\}$,
$i=1,2,\dots,N$,
where $z_i\in \mathbb R^m$ are constants,
each $M_i$ is invariant under \eqref{deq_pz} since $\dot{z}=0$.
System \eqref{deq_pz} restricted on ${M}_i$ is
\begin{equation}\label{deq_pz_M}
  p'= f(p,z_i,\epsilon),
  \quad
  z=z_i,
\end{equation}
where $\prime$ denotes $\frac{d}{d\tau}$
with $\tau=\epsilon t$.
Hence system \eqref{deq_pz} has two {\em distinguished limits}:
The {\em limiting fast system} \begin{equation}\label{fast_pz}
  \dot{p}= h(p,z,0),
  \quad
  \dot{z}= g(p,z,0),
\end{equation}
obtained by setting $\epsilon=0$ in system \eqref{deq_pz},
and the {\em limiting slow system} \begin{equation}\label{slow_pz}
  p'= f(p,z_i,0),
  \quad
  z=z_i,
\end{equation}
obtained by setting $\epsilon=0$ in \eqref{deq_pz_M}.
When there are trajectories $\gamma_i$ of \eqref{fast_pz}
and trajectories $\sigma_i\subset {M}_i$ of \eqref{slow_pz}
such that \begin{equation}\label{config_N}
  \gamma_1\cup \sigma_1
  \cup \gamma_2\cup \sigma_2
  \cup \cdots
  \cup \gamma_N\cup \sigma_N
\end{equation}
forms a closed configuration,
in the spirit of {\em Geometric Singular Perturbation Theory}
(GSPT)
(see e.g.\ 
Fenichel \cite{Fenichel:1979}, Jones~\cite{Jones:1995}
and Kuehn\cite{Kuehn:2016}),
there is potentially a periodic orbit of \eqref{deq_pz}
near configuration \eqref{config_N}
for all small $\epsilon>0$.
However,
in the case that $\sigma_i$
contains {\em turnning points},
at which the stability of $M_i$ changes,
the so-called {\em entry-exit function}
is needed to determine whether
there are trajectories of \eqref{deq_pz}
near the singular orbit.
The classical entry-exit function
was defined for system~\eqref{deq_pz}
with $p$ being a one-dimensional variable
(see
De Maesschalck \cite{DeMaesschalck:2008JDE},
De Maesschalck and Schecter \cite{DeMaesschalck:2016},
Hsu~\cite{Hsu:2017},
Wang and Zhang~\cite{ChengWang:2018}
and references therein).
In the present paper 
we generalize the entry-exit function (see Section~\ref{sec_hyp})
for system~\eqref{deq_pz}
with a multi-dimensional variable $p$.
Using our generalized entry-exit function,
we provide criteria
under which periodic orbits near the singular orbit exist.
Note that if such periodic orbits exist,
they must form a {\em relaxation oscillation}
because the vector field \eqref{deq_pz}
has magnitude of order $O(\epsilon)$ near $\sigma_i$
and has magnitude of order $O(1)$ near $\gamma_i$.

Our objective is
to understand the mechanism
of {\em rapid regime shifts} in ecological systems.
One example is trait oscillations
exhibited in an eco-evolutionary system proposed by
Cortez and Weitz \cite{Cortez:2014-PNAS}.
The system takes the following form.
\begin{equation}\label{deq_CoEvol}\begin{aligned}
  &x'=F(x,\alpha)-G(x,y,\alpha,\beta),
  \\
  &y'=H(x,y,\alpha,\beta)- D(y,\beta),
  \\
  &\epsilon\, \alpha'=\alpha(1-\alpha)\frac{\partial}{\partial \alpha}\left(\frac{x'}{x}\right),
  \\
  &\epsilon\, \beta'=\beta(1-\beta)\frac{\partial}{\partial \beta}\left(\frac{y'}{y}\right),
\end{aligned}\end{equation}
where $x(t)$ and $y(t)$ are the prey and predator densities, respectively,
and $\alpha(t)$ and $\beta(t)$ are the average trait values
of the prey and predators, respectively,
at time $t$.
The functions $F$ and $H$ are related to
the growth rates of the prey and predators, respectively,
$G$ is related to the encounter rate,
and $D$ is related to the death rate of predators.
The equations of $\alpha$ and $\beta$
were derived from the assumption that
the adaptive change in the trait follows fitness-gradient dynamics
(see Abrams et al.\ \cite{Abrams:1993}),
i.e., 
the rate of change of the mean trait value is
proportional to the fitness gradient of an individual with this mean trait value.
In Cortez and Weitz \cite{Cortez:2014-PNAS},
numerical evidences
of periodic orbits oscillating between the level sets,
for $(\alpha,\beta)=(0,0), (0,1), (1,1)$ and $(1,0)$,
were provided
for certain functional responses.
A simulation of a periodic orbit
with data from that paper
is shown in Figure~\ref{fig_CoEvol}.
Applying one of our criteria (Theorem~\ref{thm_main2})
in Section~\ref{sec_coevolution},
besides confirming the existence of periodic orbits,
we determine the limiting configuration
(see Figure~\ref{fig_CoEvol_sing}) of the periodic orbit as $\epsilon\to 0$.
This singular orbit can be used to predict the location of periodic orbits.

\begin{figure}[htbp]
\begin{center}
{\includegraphics[trim =  0cm 0cm 0cm 0cm, clip, width=.46\textwidth]%
{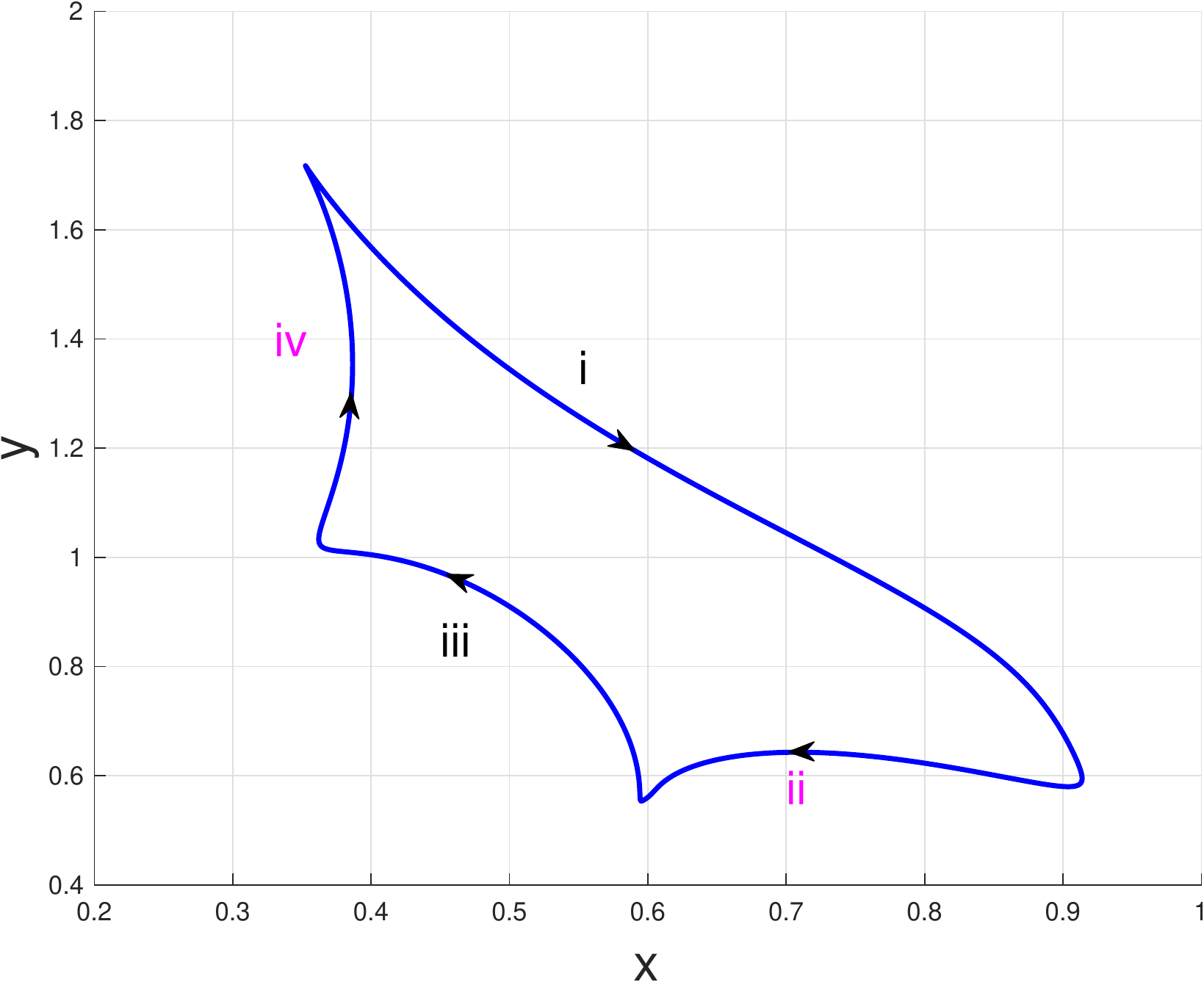}}
{\includegraphics[trim =  3em 1em 3em 1em, clip, width=.46\textwidth]%
{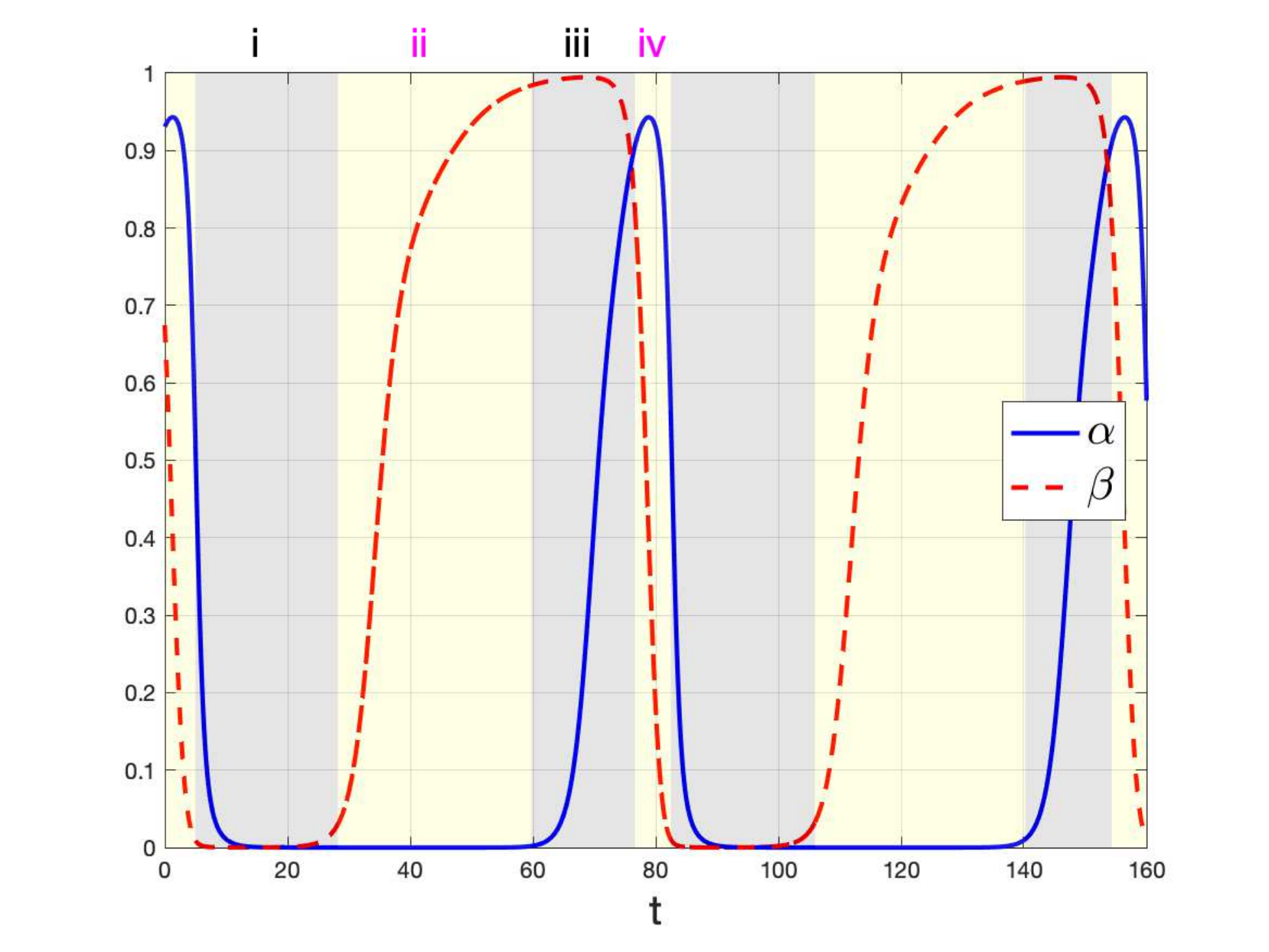}}
\\
(a)
\hspace{17em}
(b)
\end{center}
\caption{
  A periodic orbit for system \eqref{deq_CoEvol} with $\epsilon=0.25$.
  (a)
  On the $(x,y)$-plane the trajectory can roughly be split into four segments.
  (b)
  The value of $\alpha$ remains close to $0$ along segments {\sf i} and {\sf ii}
  and becomes close to $1$ in segments {\sf iii} and {\sf iv}.
  The value of $\beta$ is close to $0$
  in segments {\sf i} and {\sf iv}
  and is close to $1$ in segments {\sf ii} and {\sf iii}.
}
\label{fig_CoEvol}
\end{figure}

\begin{figure}[htbp]
\begin{center}
{\includegraphics[trim =  0cm 0cm 0cm 0cm, clip, width=.46\textwidth]%
{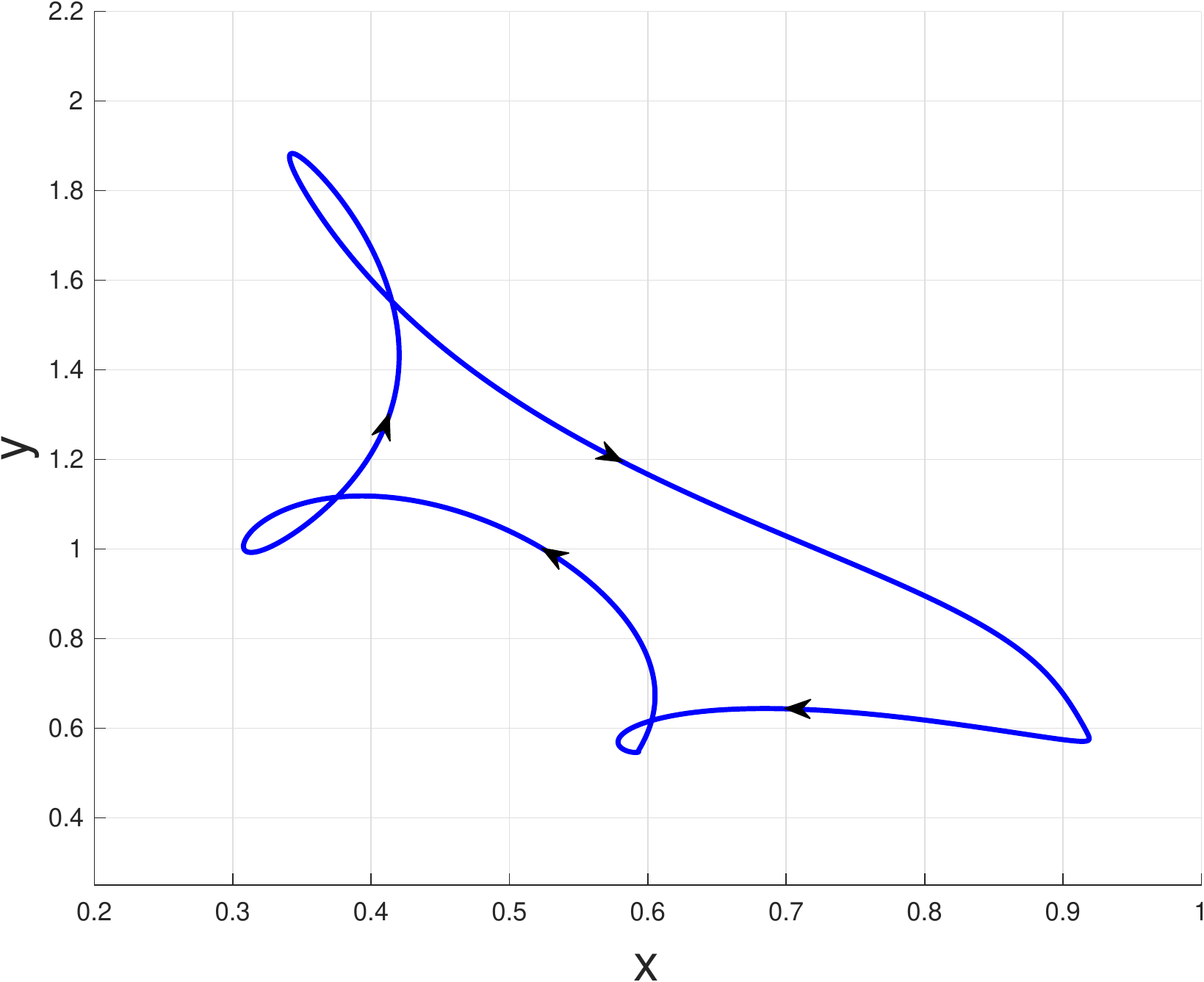}}
{\includegraphics[trim =  0cm 0cm 0cm 0cm, clip, width=.46\textwidth]%
{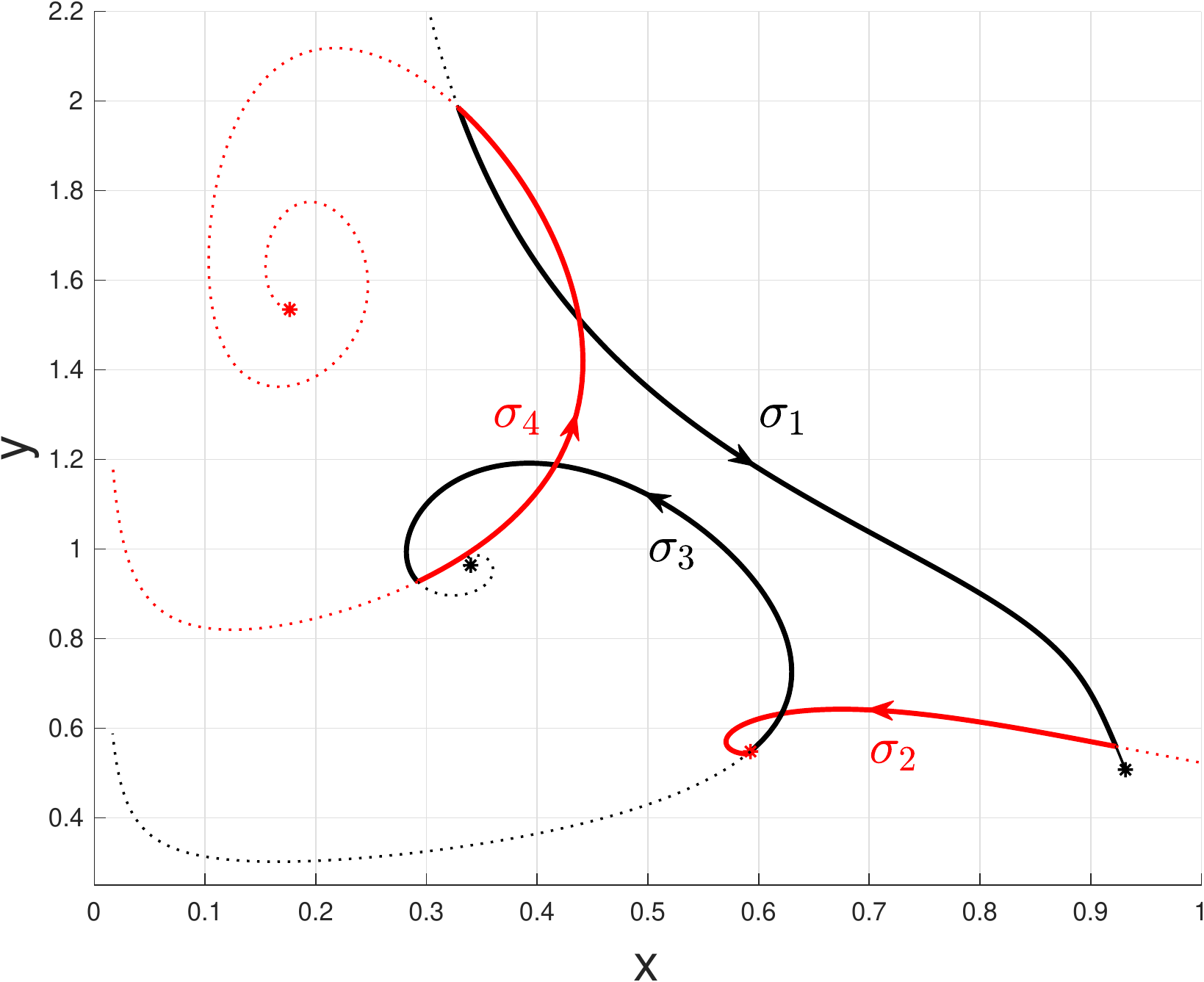}}
\\
(a)
\hspace{17em}
(b)
\end{center}
\caption{
(a)
A periodic orbit for system \eqref{deq_CoEvol} with
$\epsilon=0.10$.
(b) A singular closed orbit which consists of trajectories
of limiting subsystems.
}
\label{fig_CoEvol_sing}
\end{figure}

Another example,
proposed by Cortez and Ellner \cite{Cortez:2010},
is a predator-prey system with rapid prey evolution:
\begin{equation}\label{deq_tradeoff}\begin{aligned}
 &x'= x(\alpha+r-kx)-\frac{xy(a\alpha^2+b\alpha+c)}{1+x},
 \\
 &y'=\frac{xy(a\alpha^2+b\alpha+c)}{1+x}-dy,
 \\
 &\epsilon\,\alpha'
 =\alpha(1-\alpha)\left(1-\frac{y(2a\alpha+b)}{1+x}\right)
 \equiv \alpha(1-\alpha)E(x,y,\alpha),
\end{aligned}\end{equation}
which can be regarded as a special case of \eqref{deq_CoEvol}
with $\beta$ being constant.
Periodic orbits
that travel back and forth between the manifolds
${M}_0$
and
${M}_1$
corresponding to $\alpha=0$ and $\alpha=1$, respectively,
was discovered numerically
by Cortez and Ellner~\cite{Cortez:2010}
(see Figure~\ref{fig_tradeoff} for a simulation
with data from that paper).
Note that the sign of $E(x,y,\alpha)$, where $\alpha=0$ (resp.\ $\alpha=1$),
determines whether ${M}_{0}$ (resp.\ ${M}_{1}$)
is attracting or repelling at that point.
It was indicated by those authors that
if the trait oscillation occurs,
at the landing and jumping points on each ${M}_i$
the values of $E$ has opposite signs.
In Section~\ref{sec_tradeoff},
applying our criterion (Theorem~\ref{thm_main1})
we determine two pairs of the landing and jumping points,
$A_1,B_1\in {M}_0$ and $A_2,B_2\in {M}_1$,
by the equations \begin{equation}\label{int_E}
  \int_{\sigma_1}E(x,y,0)\;dt
  =\int_{\sigma_2}E(x,y,1)\;dt
  =0,
\end{equation}
where $\sigma_1$ is a trajectory on ${M}_0$ connecting $A_1$ and $B_1$,
and $\sigma_2$ is a trajectory on ${M}_1$ connecting $A_2$ and $B_2$
(see Figure~\ref{fig_tradeoff}).
The derivation of \eqref{int_E}
is based on the entry-exit functions on $M_i$.
Also we prove that the corresponding periodic orbits are
orbitally locally asymptotically stable.

\begin{figure}[htbp]
\begin{center}
{\,\includegraphics[trim =  8em 1em 7em 1em, clip, width=.48\textwidth]%
{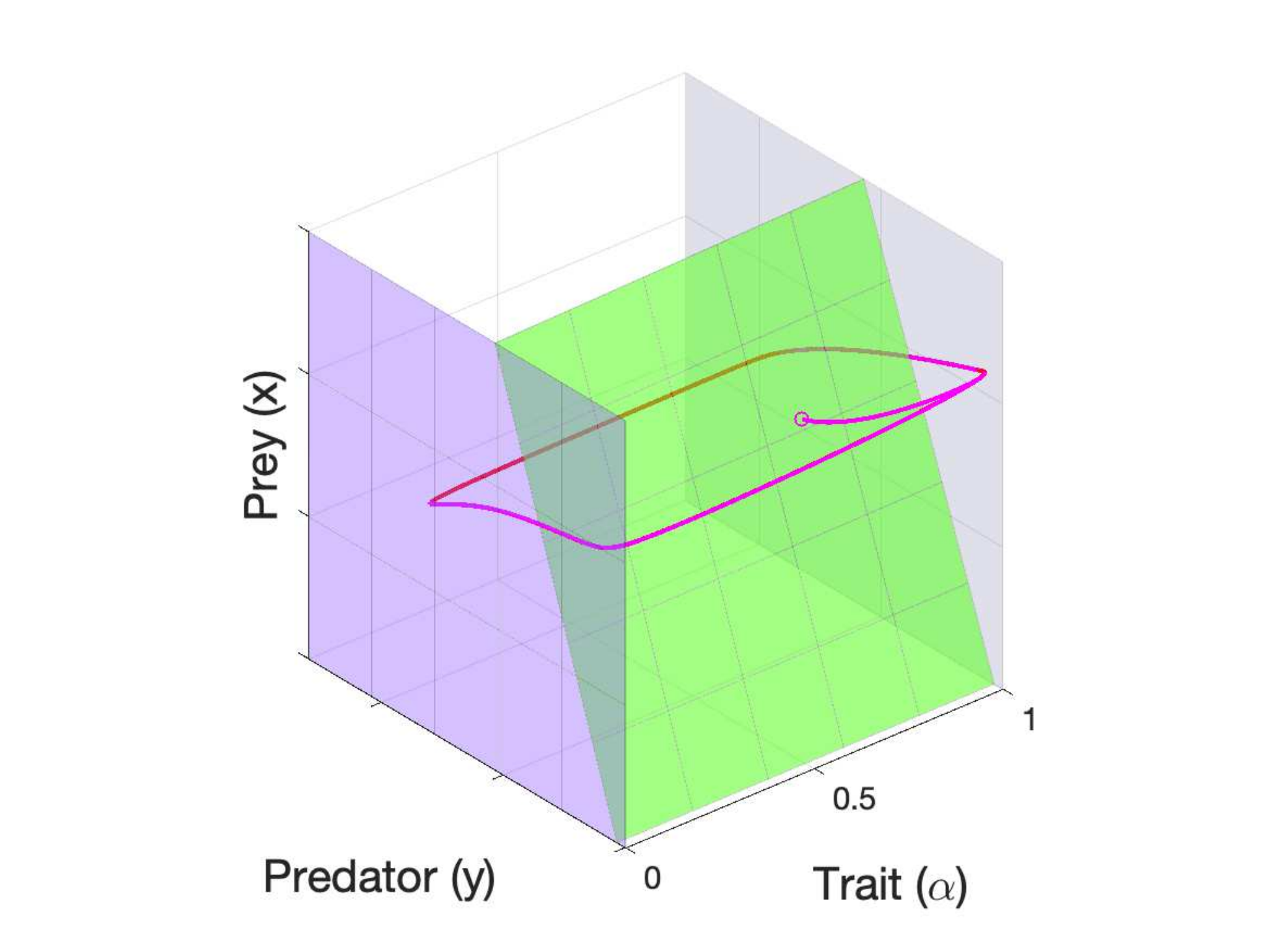}}\;
{\,\includegraphics[trim =  8em 1em 7em 1em, clip, width=.48\textwidth]%
{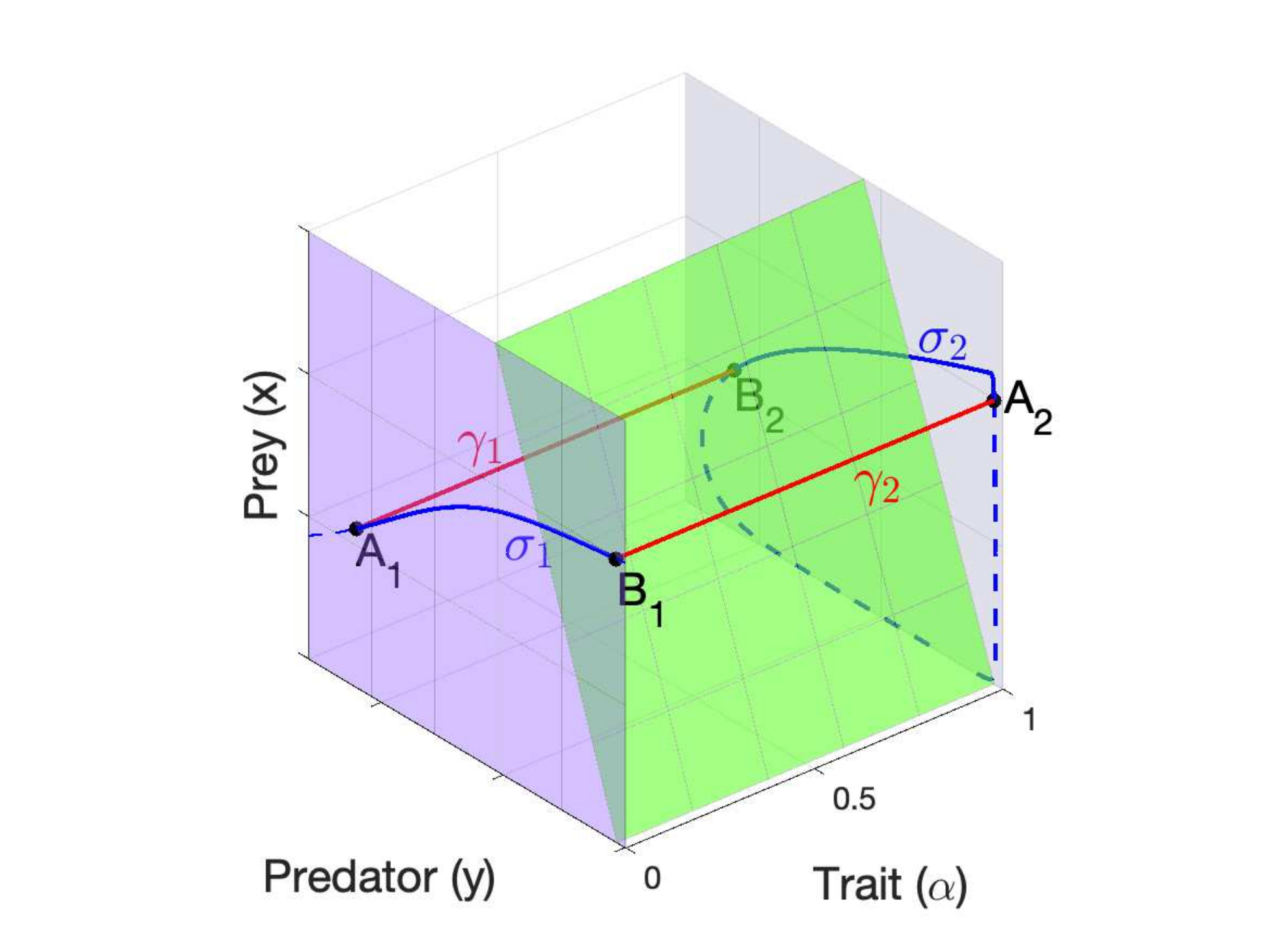}}
\\
(a)
\hspace{17em}
(b)
\end{center}
\caption{
(a)
The trajectory of \eqref{deq_tradeoff} with $\epsilon=0.1$
and initial data $(x,y,\alpha) = (10,0.5,0.5)$
converges to a periodic orbit.
(b) A singular configuration consisting of trajectories
of limiting subsystems,
and is locally uniquely determined by \eqref{int_E}.
}
\label{fig_tradeoff}
\end{figure}

The third example is
a 1-predator-2-prey system
with rapid prey evolution
proposed by Piltz et al.~\cite{Piltz:2017}:
\begin{equation}\label{deq_switching}\begin{aligned}
  &p_1'= r_1p_1- qf_1(p_1)z,
  \\
  &p_2'= r_2p_2- (1-q)f_2(p_2)z,
  \\
  &z'= c_1qf_1(p_1)z+ c_2(1-q)f_2(p_2)z- mz,
  \\
  &\epsilon\, q'
  = q(1-q)\big(c_1f_1(p_1)-c_2f_2(p_2)\big).
\end{aligned}\end{equation}
where $p_1$ and $p_2$ are population densities of two prey species,
$z$ is the population density of predators,
and $q$ is the mean trait value of predators.
The equation of $q'$ is
analogous to the equation of $\alpha'$ in \eqref{deq_CoEvol}.

A two-parameter family of closed singular configurations
formed by trajectories
of limiting slow and fast systems of \eqref{deq_switching}
has been derived in Piltz et al.~\cite{Piltz:2017}.
In Section~\ref{sec_switching},
using our criterion (Theorem~\ref{thm_main1})
we prove that there is a locally unique closed singular configuration
that admits periodic orbits (see Figure~\ref{fig_switching}(a)).
Moreover,
with parameters adapted from that paper,
by computing the linearization of the singular transition maps
we prove that
the periodic orbits are orbitally unstable (see Figure~\ref{fig_switching}(b))
for all small $\epsilon>0$.

\begin{figure}[htbp]
\begin{center}
{\,\includegraphics[trim =  1.5cm 0cm 1cm 0cm, clip, width=.48\textwidth]%
{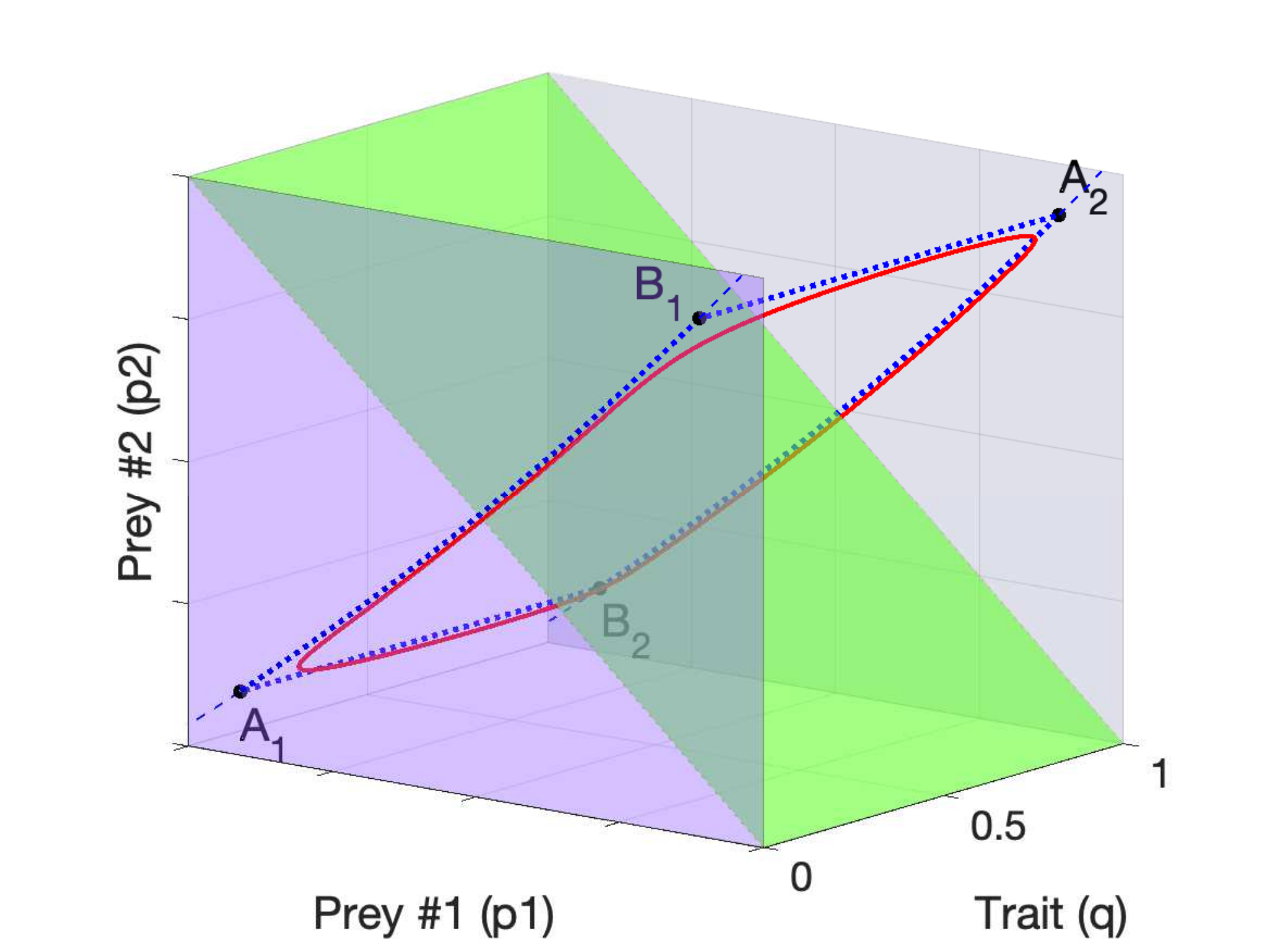}}\,
{\,\includegraphics[trim =  1.5cm 0cm 1cm 0cm, clip, width=.48\textwidth]%
{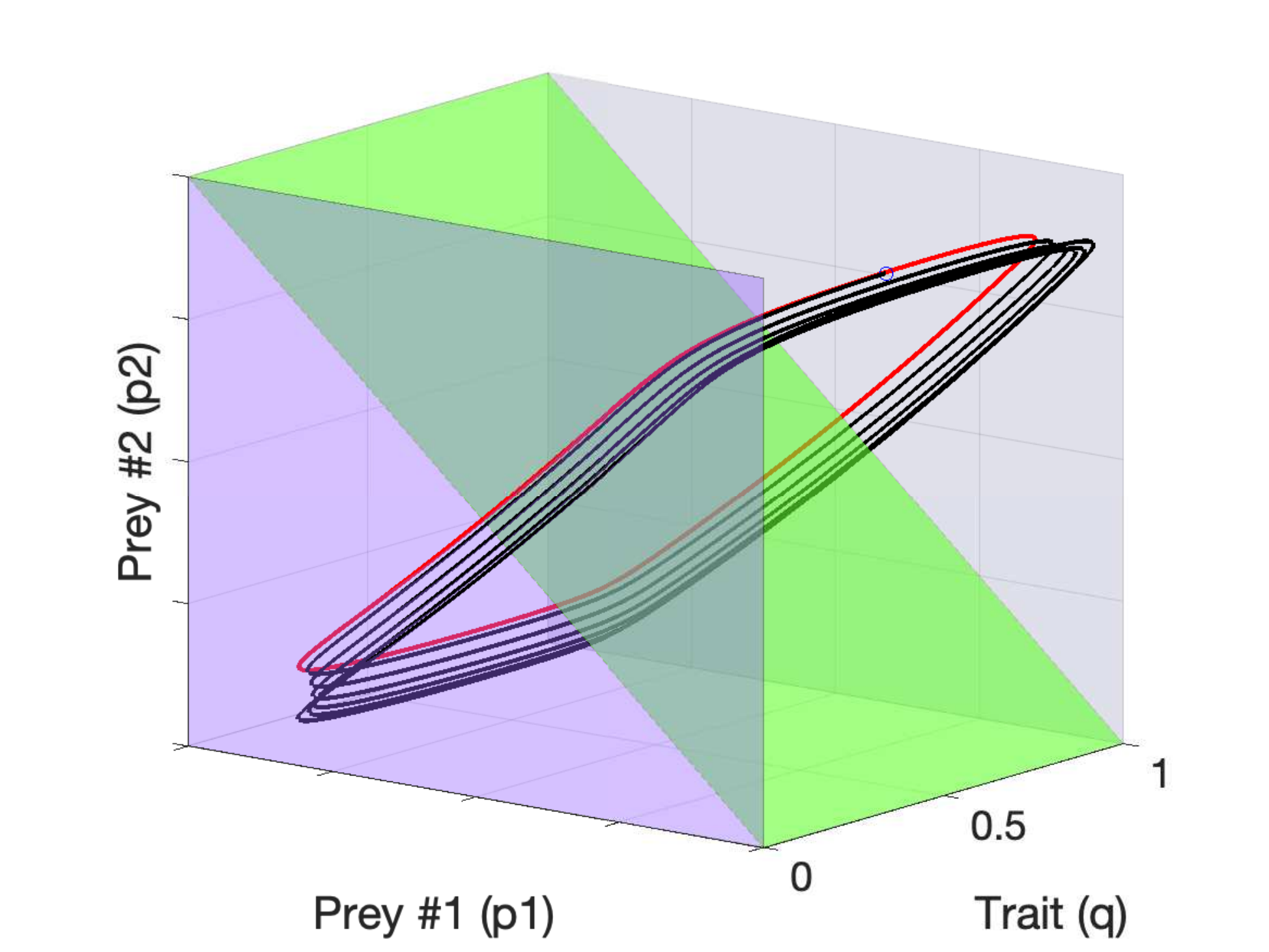}}
\\
(a)
\hspace{17em}
(b)
\end{center}
\caption{(a)
A periodic orbit for \eqref{deq_switching}
(red solid curve)
with $\epsilon=0.01$
is close to the singular configuration
(blue dotted curve)
with vertices $A_i$ and $B_i$.
(b) 
A trajectory for \eqref{deq_switching}
with $\epsilon=0.01$
and initial value (black open circle) close to the periodic orbit
leaves the vicinity of the periodic orbit
as time evolves,
which suggests that the periodic orbit is unstable.
}
\label{fig_switching}
\end{figure}

In Section~\ref{sec_planar}, we consider the planar system
studied by Hsu and Wolkowicz~\cite{Hsu:2019criterion}:
\begin{equation}\label{deq_ab}\begin{aligned}
  &\frac{d}{dt}a= \epsilon F(a,b,\epsilon)+ b\,H(a,b,\epsilon),
  \;\;
  &\frac{d}{dt}b= b\, G(a,b,\epsilon).
\end{aligned}\end{equation}
The $a$-axis is a critical manifold 
for the limiting fast system of \eqref{deq_ab}.
For singular closed orbits for this system,
a criterion of the existence and stability
of corresponding relaxation oscillations
was derived in Hsu and Wolkowicz~\cite{Hsu:2019criterion},
which generalizes the criterion in Hsu~\cite{Hsu:2019siads}.
Using our results (Theorem~\ref{thm_main3}),
we provide an alternative proof of that result.
The derivations in those papers
were based on the asymptotic expansion of Floquet exponents
for system \eqref{deq_ab}
with $\epsilon>0$.
In the present paper,
we analyze the transition maps
for the limiting slow and fast systems with $\epsilon=0$ directly,
which provides a better understanding
of the slow-fast feature in the system.

The rapid evolution model,
i.e., system \eqref{deq_CoEvol}
with $0<\epsilon\ll 1$,
has been studied by
Cortez \cite{Cortez:2011,Cortez:2015,Cortez:2016AmerNatur,Cortez:2018},
Cortez and Ellner \cite{Cortez:2010},
Cortez and Patel \cite{Cortez:2017},
Cortez and Weitz \cite{Cortez:2014-PNAS},
and Haney and Siepielski \cite{Haney:2018}.
System \eqref{deq_CoEvol}
with slow evolution,
i.e.\ $\epsilon\gg 1$,
has been studied by
Khibnik and Kondrashov \cite{Khibnik:1997},
Shen, Hsu, and Yang \cite{Shen:2019}.
Transient behaviors,
which are related to regime shifts in ecological systems,
have been studied by
Hastings \cite{Hastings:2004},
Wysham and Hastings \cite{Wysham:2008},
and Hastings et al.\ \cite{Hastings:2018-Science}.
Model \eqref{deq_switching}
is a continuous version of
the piecewise-smooth model in
Piltz, Porter and Maini~\cite{Piltz:2014}.
A comparison of the numerical solutions
of \eqref{deq_switching}
with real data was given in 
Piltz, Veerman and Maini~\cite{Piltz:2018}.

Relaxation oscillations for systems with turning points
have been studied by
Szmolyan and Wechselberger \cite{Szmolyan:2004},
Liu, Xiao and Yi \cite{Liu:2003}.
Our work is complementary to those results
since our singular orbit 
is away from fold points
(i.e.\ singular points of the slow flow).
Our result is a generalization
of the criterion of relaxation oscillations
given by
Li et al.~\cite{Li:2016},
Hsu \cite{Hsu:2019siads},
and Hsu and Wolkowicz \cite{Hsu:2019criterion}.
Relaxation oscillations
in predator-prey systems
have been studied by
various researchers,
including
Ghazaryan, Manukian and Schecter~\cite{Ghazaryan:2015},
Hsu and Shi~\cite{Hsu:2009},
Huzak~\cite{Huzak:2018},
Li and Zhu~\cite{Li:2013}
Rinaldi and Muratori~\cite{Rinaldi:1992},
and Shen, Hsu and Yang~\cite{Shen:2019}.
Relaxation oscillations in multi-dimensional slow-fast systems 
without turning points
have been studied by Soto-Trevi\~{n}o \cite{Soto:2001}.
Boundary value problems for slow-fast systems
have been studied by
Lin \cite{Lin:1989}
and Tin, Kopell and Jones \cite{Tin:1994}.

The entry-exit function can be
traced back to Benoit~\cite{Benoit:1981},
and is called the {\em way-in way-out function}
in Diener~\cite{Diener:1984}.
This phenomenon
that the landing and jumping points satisfy
the entry-exit function
has been called
{\em bifurcation delay} in Beno\^{i}t~\cite{Benoit:1991},
{\em Pontryagin delay} in Mishchenko et al.~\cite{Mishchenko:1994},
and {\em delay of instability} in Liu~\cite{Liu:2000}.

The proof of our criterion
is a generalization of the method in Hsu \cite{Hsu:2016,Hsu:2017},
which is a variation of the classical blow-up method.
The blow-up method
was developed by 
Dumortier and Roussarie
\cite{Dumortier:1996}
and Krupa and Szmolyan \cite{Krupa:2001-SIMA,Krupa:2001-JDE},
and has been applied extensively
to study various problems, 
including
Gasser, Szmolyan and W\"{a}chtler~\cite{Gasser:2016},
Iuorio, Popovi\'{c} and Szmolyan~\cite{Iuorio:2019},
Kosiuk and Szmolyan~\cite{Kosiuk:2011},
Manukian and Schecter~\cite{Manukian:2009},
Schecter~\cite{Schecter:2004-JDE},
and Schecter ans Szmolyan~\cite{Schecter:2004-JDDE}.

This paper is organized as follows.
In Section~\ref{sec_main},
we state our criteria 
for the existence and stability
of relaxation oscillations
and provide some computable formulas for the criteria.
Proofs of the criteria are given in Section~\ref{sec_proofs}.
In Section~\ref{sec_ex}
we apply our criteria to models described in Section~\ref{sec_intro}.

\section{Main Theorems}
\label{sec_main}
Assumptions needed for our main results
are stated in Section~\ref{sec_hyp}.
The criteria
for the existence of relaxation oscillations
are split into Sections~\ref{sec_main1}--\ref{sec_main3},
from single to multiple dimensional fast variables.
Formulas for computing quantities in the criteria
are given in Section~\ref{sec_formulas}.

\subsection{The Assumptions}
\label{sec_hyp}
Let $N$ be a fixed positive integer.
Throughout this paper
we adopt the notion that $A_i=A_{i+N}$
for any integer $i$
and any object $A$.
For any vector $z$ in $\mathbb R^m$,
we denote $z^{(j)}$ the $j$-th component of $z$.
We denote $\{{\sf e}_1,{\sf e}_2,\dots,{\sf e}_m\}$
the standard basis of $\mathbb R^m$.

\begin{assumption}
\label{hyp_cmin}
For each $j=1,2,\dots,m$,
there exist $-\infty\le \zmin^{(j)}< \zmax^{(j)}\le \infty$
such that for all sufficiently small $\epsilon\ge 0$,
\begin{equation*}
  h(p,z,\epsilon)=0
  \quad\text{and}\quad
  g^{(j)}(p,z,\epsilon)=0
\end{equation*}
whenever
$z^{(j)}=\zmin^{(j)}$ or $z=\zmax^{(j)}$.
\end{assumption}

\begin{assumption}
\label{hyp_fast}
For each $i=1,2,\dots,N$,
where $N$ is a positive integer,
there exist
$A_i,B_i\in \mathbb R^n$,
$J_i\in \{1,2,\dots,m\}$, \begin{equation*}
  z_i\in \{\zmin^{(1)},\zmax^{(1)}\}
  \times \{\zmin^{(2)},\zmax^{(2)}\}
  \times \cdots \times \{\zmin^{(m)},\zmax^{(m)}\}
  \quad
  \text{with}\;\; |z_i|<\infty,
\end{equation*}
and smooth functions
$\theta_i:\mathbb R\to \mathbb R^n$
and $\rho_i:\mathbb R\to \mathbb R$
such that
$\rho_i$ is non-constant
and the curve \begin{equation*}
  \gamma_i(t)
  =\big(\theta_i(t),z_i+\rho_i(t){\sf e}_{J_i}\big),\;\;
  -\infty<t<\infty,
\end{equation*}
is a heteroclinic orbit of \eqref{fast_pz}
that connects $(B_{i-1},z_{i-1})$ and $(A_{i},z_{i})$.
In additional,
for each $j=1,2,\dots,m$,
there exists $i\in \{1,2,\dots,N\}$ such that $J_i=j$.
\end{assumption}

The expression of the heteroclinic orbit 
in Assumption~\ref{hyp_fast}
implies that
$z_i$ differs from $z_{i+1}$
at no more than one component.
Note that we do not exclude the possibility that $z_i=z_{i+1}$.

The assumption of the existence of $i$
such that $J_i=j$
means that each component $z^{(j)}$ of $(p,z)$
must be non-constant along at least one $\gamma_i$.
If it is not the case,
then we can treat $z^{(j)}$ as a constant
and replace the equation of $\dot{z}^{(j)}$ in \eqref{deq_pz}
by $\dot{z}^{(j)}=0$
because the space 
$\{(p,z): z^{(j)}=\zmin^{(j)} \text{ or } \zmax^{(j)}\}$
is invariant under \eqref{deq_pz} by Assumption~\ref{hyp_cmin}.

We define ${M}_i= \{(p,z): p\in \mathbb R^n,z=z_i\}$
for $i=1,2,\dots,N$.
Then Assumption~\ref{hyp_cmin}
implies that ${M}_i$ is invariant under \eqref{deq_pz}
for all sufficiently small $\epsilon>0$.
The restriction of \eqref{deq_pz} on ${M}_i$ is \eqref{slow_pz}.
We denote the solution operator of \eqref{slow_pz}
by $\Phi_{i}$.

\begin{assumption}
\label{hyp_slow}
For each $i=1,2,\dots,n$,
$f_i(A_i,z_i,0)\ne 0$
and 
there exists
$\tau_i>0$
such that $\Phi_{i}(\tau_i,A_i)=B_i$.
\end{assumption}

Denote $\sigma_i=\Phi_{i}([0,\tau_i],A_i)\times \{z_i\}$.
Then by Assumptions~\ref{hyp_fast}--\ref{hyp_slow}
the configuration \eqref{config_N} forms a closed orbit.
The idea of GSPT
is that solutions of the full system
can potentially be obtained
by joining some trajectories of
its limiting systems.
The limiting systems \eqref{fast_pz} and \eqref{slow_pz}
provide a family of uncountably many loops.
Our goal is to establish a criterion
for the existence of a locally unique periodic orbit
near this closed singular orbit.

We impose the following non-degeneracy condition.
\begin{assumption}
\label{hyp_nondegen}
For $i=1,2,\dots,N$, \begin{equation*}
  \frac{\partial g^{(J_i)}}{\partial z^{(J_i)}}(A_i,z_i,0)< 0
  \quad\text{and}\quad
  \frac{\partial g^{(J_i)}}{\partial z^{(J_i)}}(B_i,z_i,0)> 0.
\end{equation*}
\end{assumption}

\begin{remark}
By Assumption~\ref{hyp_cmin},
the linearization 
of \eqref{fast_pz}
at any point $(p,z_i)$ in ${M}_i$
has the Jacobian matrix
\begin{equation*}
  \begin{pmatrix}
  0_{n\times n}& *
  \\[.5em]
  0_{m\times n}&
  \mathrm{diag}\left(
  \frac{\partial g^{(1)}}{\partial z^{(1)}},
  \dots,
  \frac{\partial g^{(m)}}{\partial z^{(m)}}
  \right)
  \end{pmatrix}
\end{equation*}
where
the partial derivatives are evaluated at $(p,z_i,0)$.
In the case that $m=1$,
the inequalities in Assumption~\ref{hyp_nondegen}
imply that ${M}_i$
is normally hyperbolic at $(A_i,z_i)$ and $(B_i,z_i)$
and that there is a turning point on $M_i$
between these two points.
\end{remark}

In the case that $m=1$,
where $z$ and $g$ are scalar,
the classical entry-exit relation for \eqref{deq_pz}
between $A_i$ and $B_i$
can be expressed by \begin{equation}\label{int_ABm1}
  \int_0^{s}\frac{\partial g}{\partial z}\big(
  \Phi_{i}(\tau,A_i),z_{i},0
  \big)\;d\tau
  \;\begin{cases}
    =0,&\text{if $s=\tau_i$},
    \\
    < 0,&\text{if $0<s<\tau_i$}.
  \end{cases}
\end{equation}
Assuming \eqref{int_ABm1} and Assumption~\ref{hyp_nondegen},
on some neighborhood $\A_i$ of $A_i$ in $\mathbb R^n$
we can implicitly define $T_i:\A_i\to (0,\infty)$ by $T_i(A_i)=\tau_i$
and \begin{equation}\label{def_Tim1}
  \int_0^{T_i(p)}\frac{\partial g}{\partial z}\big(
  \Phi_{i}(\tau,p),z_{i},0
  \big)\;d\tau=0.
\end{equation}
The entry-exit function is then defined by \begin{equation}\label{def_Qim1}
  Q_i(p)=\Phi_i(T_i(p),p).
\end{equation}
Each pair of points $(p,z_i)$ and $(Q_i(p),z_i)$, where $p\in \A_i$,
is a pair of landing and jumping points on $M_i$.

For the general case that $m\ge 1$,
we first introduce some notations.
Let $J_i$, where $i=1,2,\dots,N$,
be the numbers defined in Assumption~\ref{hyp_cmin}.
For each $j=1,2,\dots,m$,
let \begin{equation*}
  I_j=\min\{i\in \{1,2,\dots,N\}: J_i=j\}.
\end{equation*}
This means that $I_j$ is the smallest positive $i$
for which the value of $z^{(j)}$
changes
along the trajectory 
$\gamma_i$.
By Assumption~\ref{hyp_fast},
$I_j$ is well-defined and is finite.
We define \begin{equation*}
  \zeta_0^{(j)}
  = -\sum_{k=1}^{I_j}\left(
  \int_0^{\tau_k}
  \frac{\partial g^{(j)}}{\partial z^{(j)}}\big(
  \Phi_{i}(\tau,A_k),z_{k},0
  \big)
  \;d\tau
  \right)
\end{equation*}
and \begin{equation*}
  \zeta_i^{(j)}
  = \zeta_0^{(j)}
  + \sum_{k=1}^{i}\left(
  \int_0^{\tau_k}
  \frac{\partial g^{(j)}}{\partial z^{(j)}}\big(
  \Phi_{i}(\tau,A_k),z_{k},0
  \big)
  \;d\tau
  \right)
\end{equation*}
for $i=1,2,\dots,N$ and $j=1,2,\dots,m$.
Also we denote $\zeta_i=(\zeta_1^{(1)},\dots,\zeta_1^{(m)})$.
The following assumption is a generalization of \eqref{int_ABm1}.

\begin{assumption}
\label{hyp_int}
For each $i\in \{1,2,\dots,N\}$, $j\in \{1,2,\dots,m\}$ and $s\in (0,\tau_i]$,
\begin{equation*}
  \zeta_i^{(j)}
  + \int_0^s
  \frac{\partial g^{(j)}}{\partial z^{(j)}}\big(
    \Phi_{i}(\tau,A_i),z_{i},0
  \big)
  \;d\tau
  \;\begin{cases}
    =0,&\text{if\; $j=J_i$ and $s=\tau_i$},
    \\
    \ne 0,&\text{otherwise}.
  \end{cases}
\end{equation*}
\end{assumption}

For each $i=1,2,\dots,N$,
we consider the system \begin{equation}\label{slow_pzeta}\begin{aligned}
  &\frac{d}{d\tau}p
  = f(p,z_i,0),
  \\
  &\frac{d}{d\tau}\zeta^{(j)}
  =\frac{\partial g^{(j)}}{\partial z^{(j)}}(p,z_i,0),
  \quad j=1,2,\dots,m.
\end{aligned}\end{equation}
Let \begin{equation}\label{def_Lambda}
  \Lambda_i
  =\left\{
    \zeta\in \mathbb R^m:
    \big|\zeta-\zeta_i\big|< \delta,\;
    \zeta^{(J_i)}= \zeta_i^{(J_i)}
  \right\},
\end{equation} where $\delta>0$ .
Let $\widehat{\Phi}_i$ be the solution operator for \eqref{slow_pzeta}.
From Assumption~\ref{hyp_nondegen},
by shrinking $\A_i$ and $\delta$ if necessary,
we can define $\widehat{T}_i(p,\zeta)$
on $\A_i\times \Lambda_i$ 
implicitly by $\widehat{T}_i(A_i,\zeta_i)=0$
and \begin{equation}\label{def_TiHat}
  \zeta^{(J_i)}
  + \int_0^{\widehat{T}_i(p,\zeta)}
  \frac{\partial g^{(J_i)}}{\partial z^{(J_i)}}
  (\Phi_i(\tau,p),z_i,0)\;d\tau
  =0.
\end{equation}
Finally, we define 
the generalized entry-exit function
$\widehat{Q}_i(p,\zeta)$ on $\A_i\times \Lambda_i$
by \begin{equation}\label{def_QiHat}
  \widehat{Q}_i(p,\zeta)
  = \widehat{\Phi}_i((p,\zeta),\widehat{T}_i(p,\zeta)\big).
\end{equation} 
Note that $\widehat{T}_i(p,\zeta_i)=T_i(p)$
and therefore $\widehat{Q}_i(p,\zeta_i)= (Q_i(p),\zeta_{i+1})$
for all $p\in \A_i$.
In particular, $\widehat{Q}_i(A_i,\zeta_i)= (B_i,\zeta_{i+1})$.

\begin{remark}
In the case that $m=1$,
we have
$\zeta_i^{(j)}=0$ for all $i$ and $j$,
so Assumption~\ref{hyp_int}
is reduced to the classical entry-exit relation \eqref{int_ABm1},
and $\widehat{Q}_i$ defined by \eqref{def_TiHat}--\eqref{def_QiHat}
coincides with $Q_i$ defined by \eqref{def_Tim1}--\eqref{def_Qim1}.
\end{remark}

\subsection{
Systems with a Single and Simple Fast Variable}
\label{sec_main1}
First we state our results for system \eqref{deq_pz}
with $m=1$ and $h=0$,
which can be applied to models \eqref{deq_tradeoff} and \eqref{deq_switching}.
These restrictions mean that
the system has a single variable
and that
the slow variable is steady in the fast system \eqref{fast_pz}.

Since the slow variable is steady in the fast system \eqref{fast_pz}
in the case that $h=0$,
the function $\theta_i$ in Assumption~\ref{hyp_fast}
is constant for each $i=1,2,\dots,N$.
Hence $B_{i}=A_{i+1}$ for each $i$.
Since $Q_i(A_i)=B_i$,
it follows that $Q_i(A_i)=A_{i+1}$.
Let \begin{equation}\label{def_P}
  P=
  Q_{N}\circ \cdots\circ Q_2\circ Q_1.
\end{equation}
Then $P(A_1)=A_1$
and $P$ maps a neighborhood
of $A_1$ in $\A_1$
into $\A_1$.

\begin{theorem}\label{thm_main1}
Suppose that Assumptions~\ref{hyp_cmin}--\ref{hyp_int}
hold for system \eqref{deq_pz} with $m=1$ and $h=0$.
Let $P$ be defined by \eqref{def_P}.
If \begin{equation}\notag
  \det(DP(A_1)-\mathrm{I}_n)\ne 0,
\end{equation}
where $\mathrm{I}_n$ is the identify matrix of rank $n$,
then the configuration \eqref{config_N}
admits a relaxation oscillation.
Furthermore,
the corresponding periodic orbits
are orbitally asymptotically stable
if the spectrum radius of $DP(A_1)$ is less than one
and orbitally unstable
if the spectrum radius of $DP(A_1)$ is greater than one.
\end{theorem}

The proof of the theorem
is shown in Section~\ref{sec_proof_main1}.

\subsection{Systems with Simple Fast Dynamics}
\label{sec_main2}
System \eqref{deq_pz}
with $m\ge 1$ and $h=0$
can be applied to \eqref{deq_CoEvol}.
For this case,
we introduce the following definitions.

Under the assumption that $h=0$,
we have $B_{i}=A_{i+1}$.
Since $\widehat{Q}_i(A_i,\zeta_i)= (B_i,\zeta_{i+1})$,
it follows that $\widehat{Q}_i(A_i,\zeta_i)=(A_{i+1},\zeta_{i+1})$.
Let \begin{equation}\label{def_Phat}
  \widehat{P}
  =\widehat{Q}_{N}
  \circ \cdots 
  \circ \widehat{Q}_2
  \circ \widehat{Q}_1.
\end{equation}
Then $\widehat{P}(A_1,\zeta_1)=(A_1,\zeta_1)$
and $\widehat{P}$
maps a neighborhood 
of $(A_1,\zeta_1)$ in $\A_1\times \Lambda_1$
into $\A_1\times \Lambda_1$.

\begin{theorem}\label{thm_main2}
Suppose that Assumptions~\ref{hyp_cmin}--\ref{hyp_int} hold
for system \eqref{deq_pz} with $h=0$.
Let $\widehat{P}$ be defined by \eqref{def_Phat}.
If \begin{equation}\notag
  \det(D\widehat{P}(A_1,\zeta_1)-\mathrm{I}_{n+m-1})\ne 0,
\end{equation}
where $D\widehat{P}$
is the Jacobian matrix with respect to the standard coordinate
of $\A_1\times \Lambda_1$,
then the configuration \eqref{config_N}
admits a relaxation oscillation.
Furthermore,
the corresponding periodic orbits
are orbitally asymptotically stable
if the spectrum radius of $D\widehat{P}(A_1,\zeta_1)$ is less than one
and orbitally unstable
if the spectrum radius of $D\widehat{P}(A_1,\zeta_1)$ is greater than one.
\end{theorem}

Theorem~\ref{thm_main2} 
is resulted from a more general theorem,
Theorem~\ref{thm_main3},
stated below.

\subsection{Systems with Multiple Slow and Fast Variables}
\label{sec_main3}
Now we consider 
system \eqref{deq_pz}
with general $h$
for treating system \eqref{deq_ab}.

For $i=1,2,\dots,N$ and $j=1,2,\dots,m$,
let \begin{equation}\label{def_omega}\begin{aligned}
  \omega_i^{(j)}
  =\begin{cases}
    1,
    &\text{if }\; z_{i}^{(j)}= \zmin^{(j)},
    \\[.5em]
    -1,
    &\text{if }\; z_{i}^{(j)}= \zmax^{(j)}.
  \end{cases}
\end{aligned}\end{equation}
Let \begin{equation*}\begin{aligned}
  \phi_{i}(q)
  = \begin{cases}
    \displaystyle
    \frac{\omega_i^{(J_i)}}{q-z_i^{(J_i)}},
    &\text{if } z_{i-1}^{(J_i)}= z_{i-1}^{(J_i)},
    \\[1.2em]
    \displaystyle
    \frac{\omega_i^{(J_i)}}{q-z_i^{(J_i)}}\;
    \frac{\omega_{i-1}^{(J_{i})}}{q-z_{i-1}^{(J_{i})}},
    &\text{if } z_i^{(J_i)}\ne z_{i-1}^{(J_i)}.
  \end{cases}
\end{aligned}\end{equation*}
Note that $\phi_i(z^{(J_i)})>0$
for all $(p,z)$ on $\gamma_i$.

Define functions $g_i$ and $h_i$
of $(p,q)\in \mathbb R^N\times \mathbb R$
by
\begin{equation}\label{def_gi}
  (g_i,h_i)(p,q)
  =\phi_i(q)\,
  (g^{(J_i)},h)(p,z_{i-1}+q{\sf e}_{J_i},0)
  \quad
  \text{for }q\ne z_i^{(J_i)},z_{i-1}^{(J_{i})}.
\end{equation}
By Assumptions~\ref{hyp_cmin},
$(g_i,h_i)$ can be continuously extended
at those singularities.
We identify $(g_i,h_i)$ 
with its continuous extension.
Thus $g_i(B_{i-1},z_{i-1}^{(J_i)})$
and $g_i(A_i,z_i^{(J_i)})$
are multiples of 
$\frac{\partial g^{(J_i)}}{\partial z^{(J_i)}}(B_{i-1},z_{i-1},0)$
and
$\frac{\partial g^{(J_i)}}{\partial z^{(J_i)}}(A_i,z_i,0)$,
respectively,
by nonzero constants.
By Assumption~\ref{hyp_nondegen},
it follows that
$g_i(B_{i-1},z_{i-1}^{(J_{i})})\ne 0$
and $g_i(A_i,z_i^{(J_i)})\ne 0$.

Note that the functions $\theta_i$ and $\rho_i$
in Assumption~\ref{hyp_fast}
satisfy that $\{(\theta_i,\rho_i)(t)): t\in \mathbb R\}$
is a trajectory of the system \begin{equation}\label{fast_pq}
  \dot{p}= h_i(p,q),
  \quad
  \dot{q}= g_i(p,q),
\end{equation}
that connects $(B_{i-1},z_{i-1}^{(J_{i})})$ and $(A_i,z_i^{(J_i)})$.
Since $g_i(B_{i-1},z_{i-1}^{(J_{i})})\ne 0$
and $g_i(A_i,z_i^{(J_i)})\ne 0$,
there exists a neighborhood
$\B_{i-1}$ of $B_{i-1}$
such that we can define $\pi_i:\B_{i-1}\to \A_i$
implicitly by that \begin{equation}\label{def_pi}
  \text{$\big(p,z_{i-1}^{(J_{i})}\big)$
  and
  $\big(\pi_i(p),z_i^{(J_{i})}\big)$
  are connected by a trajectory of \eqref{fast_pq}.
  }
\end{equation}

Let $\pi_i\times \mathrm{id}$
be the map from
$\B_{i-1}\times \Lambda_{i}$
to
$\A_i\times \Lambda_{i}$
given by $(\pi_i\times \mathrm{id})(p,\zeta)=(\pi_i(p),\zeta)$.
Define \begin{equation}\label{def_Ptilde}
  \widetilde{P}
  =
  (\pi_N\times \mathrm{id})
  \circ \widehat{Q}_{N}
  \circ (\pi_N\times \mathrm{id})
  \circ \cdots 
  \circ \widehat{Q}_2
  \circ (\pi_2\times \mathrm{id})
  \circ \widehat{Q}_1.
\end{equation}

\begin{theorem}\label{thm_main3}
Suppose that Assumptions~\ref{hyp_cmin}--\ref{hyp_int} hold
for system \eqref{deq_pz}.
Let $\widetilde{P}$ be defined by \eqref{def_Ptilde}.
If \begin{equation}\notag
  \det(D\widetilde{P}(A_1,\zeta_1)-\mathrm{I}_{n+m-1})\ne 0,
\end{equation}
where $D\widetilde{P}$
is the Jacobian matrix with respect to the standard coordinate
of $\A_1\times \Lambda_1$,
then the configuration \eqref{config_N}
admits a relaxation oscillation.
Furthermore,
the corresponding periodic orbits
are orbitally asymptotically stable
if the spectrum radius of $D\widetilde{P}(A_1,\zeta_1)$ is less than one
and orbitally unstable
if the spectrum radius of $D\widetilde{P}(A_1,\zeta_1)$ is greater than one.
\end{theorem}

The proof of the theorem
is shown in Section~\ref{sec_proof_main3}.

\subsection{Some Computable Formulas}
\label{sec_formulas}
For each $i=1,2,\dots,N$,
for convenience we define $f_i(p)=f(p,z_i,0)$
and $p_i(\tau)=\Phi_{i}(\tau,A_i)$.
Let $L_i(\tau)$ be the fundamental matrix
for the variational equations of \eqref{slow_pz}
along $\sigma_i$.
This means that for any $v\in \mathbb R^n$,
$w(\tau)=L_i(\tau)v$ is the solution of
\begin{equation}\label{def_Li}
  \frac{d}{d\tau}w
  =\big[Df_i(p_i(\tau))\big]w,
  \quad
  w(0)=v_0,
  \quad\text{for }\;
  0\le \tau\le \tau_i.
\end{equation}
It can be shown that, for $v\in \mathbb R^n$ and $0\le \tau\le \tau_i$, \begin{equation}\label{Li_DPhi}
  L_i(\tau)v=D\Phi(\tau,A_i)v
\end{equation} and \begin{equation}\label{Li_int}
  L_i(\tau)v
  = v
  + \int_0^{\tau} \big[Df_i(p_i(s))\big]L_i(s)v\; ds.
\end{equation}
We define the linear functional $\mu_i$ on $\mathbb R^n$
by \begin{equation}\label{def_mu}
  \mu_i(v)
  =\int_0^{\tau_i}
  \left\langle
  L_i(\tau)v,
  D\frac{\partial g^{(J_i)}}{\partial z^{(J_i)}}{(p_i(\tau),z_i,0)}
  \right\rangle
  d\tau
  \quad
  \text{for }v\in \mathbb R^n,
\end{equation}
where $D$ denotes the derivative with respect to $p$.

\begin{proposition}
\label{prop_DQi}
Let ${Q}_i$ be defined by \eqref{def_Qim1}.
Then
\begin{equation}\label{DQi_Li}
  D{Q}_i(A_i,\zeta_i)v
  = L_i(\tau_i)v
  -\frac{\mu_i(v)}{\displaystyle
    \frac{\partial g^{(J_i)}}{\partial z^{(J_i)}
  }(B_i,z_i,0)}
  f(B_i,z_i,0)
  \quad
  \forall v\in \mathbb R^n.
\end{equation}
In particular, \begin{equation}\label{DQi_fBi}
  D{Q}_i(A_i,\zeta_i)f(A_i,z_i,0)
  = \frac{\displaystyle
    \frac{\partial g^{(J_i)}}{\partial z^{(J_i)}}(A_i,z_i,0)
  }{\displaystyle
    \frac{\partial g^{(J_i)}}{\partial z^{(J_i)}}(B_i,z_i,0)
  }f(B_i,z_i,0).
\end{equation}
\end{proposition}

\begin{proof}
By differentiating \eqref{def_Tim1} with respect to $p$
we obtain \begin{equation*}\begin{aligned}
  &\langle DT_i(p),v\rangle
  \frac{\partial g^{(J_i)}}{\partial z^{(J_i)}}
  \big(
    \Phi_{i}(\tau_i,A),z_i,0
  \big)
  \\
  &\quad
  +\int_0^{T_i(A)}
  \left\langle
  D\frac{\partial g^{(J_i)}}{\partial z^{(J_i)}}
  \big(
    \Phi_{i}(\tau,A),z_i,0
  \big),
  D\Phi_{i}(\tau,A)v
  \right\rangle
  \;d\tau
  =0
  \quad
  \forall v\in  \mathbb R^n.
\end{aligned}\end{equation*}
Evaluating this equation at $A=A_i$
yields \begin{equation*}
  \langle DT_i(p),v\rangle
  \frac{\partial g^{(J_i)}}{\partial z^{(J_i)}}(B_i,z_i,0)
  =-\int_0^{\tau_i}
  \left\langle
  D\frac{\partial g^{(J_i)}}{\partial z^{(J_i)}}
  \big(
    p_i(\tau),z_i,0
  \big),
  L_i(\tau)v
  \right\rangle
  \;d\tau.
\end{equation*} By \eqref{def_mu} it follows that
\begin{equation}\label{DTi_mu}
  \langle DT_i(p),v\rangle
  =\frac{-\mu_i(v)}{\displaystyle
    \frac{\partial g^{(J_i)}}{\partial z^{(J_i)}}(B_i,z_i,0)
  }.
\end{equation}

On the other hand,
since $\Phi_i$ is the solution operator for \eqref{slow_pz},
the definition of $Q_i$ in \eqref{def_Qim1} means that \begin{equation*}
  Q_i(p)=p+\int_0^{T_i(p)} f_i(\Phi_i(\tau,p))\;d\tau.
\end{equation*}
Differentiating both sides of the equation with respect to $p$ gives
\begin{equation*}\begin{aligned}
  DQ_i(p)v
  &=v
  + \left\langle DT_i(p),v \right\rangle
  f_i(\Phi_i(T_i(p),p))
  \\
  &\qquad
  +\int_0^{T_i(p)} Df_i(\Phi_i(\tau,p))D\Phi_i(\tau,p)v\;d\tau
  \quad\forall v\in \mathbb R^n.
\end{aligned}\end{equation*}
Evaluating the equation at $p=A_i$
and using \eqref{Li_DPhi}
we have \begin{equation*}\begin{aligned}
  DQ_i(A_i)v
  &= v
  + \left\langle DT_i(A_i),v \right\rangle f_i(B_i)
  + \int_0^{\tau_i} Df_i(p_i(\tau))L_i(\tau)v\;d\tau.
\end{aligned}\end{equation*}
By \eqref{Li_int} it follows that \begin{equation}\label{DQi_DTi}\begin{aligned}
  DQ_i(A_i)v
  &= L_i(\tau_i)v
  + \left\langle DT_i(A_i),v \right\rangle f_i(B_i).
\end{aligned}\end{equation}
Substituting \eqref{DTi_mu} into \eqref{DQi_DTi},
we then obtain \eqref{DQi_Li}.

Since $f_i(p_i(\tau))$ is a solution of \eqref{def_Li}
with $v_0=f_i(A_i)$, \begin{equation}\label{Li_fAi}
  L_i(\tau)f_i(A_i)=f_i(p_i(\tau))
  \quad 
  \text{for }\;
  0\le \tau\le \tau_i.
\end{equation}
Using $\frac{d}{d\tau}p_i(\tau)=f_i(p_i(\tau))$
and \eqref{Li_fAi},
evaluating \eqref{def_mu} at $v=f_i(p)$ gives
\begin{equation}\label{mu_fAi}
  \mu(f_i(A_i))
  =\left.
  \frac{\partial g^{(J_i)}}{\partial z^{(J_i)}}{(p_i(\tau),z_i,0)}
  \right|_{\tau=0}^{\tau_i}
  =
  \frac{\partial g^{(J_i)}}{\partial z^{(J_i)}}{(B_i,z_i,0)}
  - \frac{\partial g^{(J_i)}}{\partial z^{(J_i)}}{(A_i,z_i,0)}.
\end{equation}
Substituting \eqref{mu_fAi} into \eqref{DQi_Li}
we obtain \eqref{DQi_fBi}.
\end{proof}

\begin{remark}
Numerical approximations of
$L_i$ and $\mu_i$ 
can be computed
by extending system \eqref{fast_pz}
of $p$
to a system of $(p,w,\mu)$
by appending
equations \eqref{def_mu} and \begin{equation*}
  \frac{d}{d\tau}
  \mu_i
  =   \left\langle
  L_i(\tau)v,
  D\frac{\partial g^{(J_i)}}{\partial z^{(J_i)}}{(p_i(\tau),z_i,0)}
  \right\rangle.
\end{equation*}
\label{rmk_computing}
\end{remark}

\begin{proposition}
\label{prop_DQiHat}
Let $\widehat{Q}_i$ be defined by \eqref{def_QiHat}.
Then \begin{equation}\label{DpQiHat}\begin{aligned}
  &D\widehat{Q}_i(A_i,\zeta_i)(v,0)
  \\
  &\hspace{2em}
  =\left(
    DQ_i(A_i)v,\;
    \frac{-\nu_i(v)}{\frac{\partial g^{(J_i)}}{\partial z^{(J_i)}}(B_i,z_i,0)}
    \sum_{j\ne J_i}
      \frac{\partial g^{(j)}}{\partial z^{(j)}}\big(B_i,z_i,0)
    \,{\sf e}_j
  \right)
  \quad\forall v\in \mathbb R^n,
\end{aligned}\end{equation} where $\nu_i(v)$ is defined by \eqref{def_mu},
and \begin{equation}\label{DzQiHat}\begin{aligned}
  &D\widehat{Q}_i(A_i,\zeta_i)(0,{\sf e}_j)
  \\
  &\quad
  =\begin{cases}
    (0,{\sf e}_j),
    &\text{if }j\ne J_i
    \\[.5em]
    \frac{1}{\frac{\partial g^{(J_i)}}{\partial z^{(J_i)}}(B_i,z_i,0)}
    \left(
      f(B_i,z_i,0),\;
      \sum_{k\ne J_i}
      \frac{\partial g^{(k)}}{\partial z^{(k)}}\big(B_i,z_i,0)
      \,{\sf e_k}
    \right),
    &\text{if }j= J_i.
  \end{cases}
\end{aligned}\end{equation}
\end{proposition}

\begin{proof}
We identify vectors $v\in \mathbb{R}^n$ with 
their images $(v,0_m)\in \mathbb R^n\times \mathbb R^m$,
and identify
the vector
${\sf e}_j$, $j\in \{1,2,\dots,m\}$,
in the standard basis of $\mathbb R^m$,
with the vector $(0_n,{\sf e}_j)$ in $\mathbb R^n\times \mathbb R^m$.
The function $\widehat{Q}_i(p,\zeta)$ defined by \eqref{def_QiHat}
can be written as
\begin{equation}\label{QiHat_ek}\begin{aligned}
  &\widehat{Q}_i(p,\zeta)
  \\
  &\quad
  =\left(
  \Phi\big(p,\widehat{T}_i(p,\zeta)\big),
  \sum_{k\ne J_i}\left[
    \zeta^{(k)}
    +\int_0^{\widehat{T}(p,\zeta^{(J_i)})}
    \frac{\partial g^{(k)}}{\partial z^{(k)}}\big(\Phi(p,\tau),z_i,0)\;d\tau
  \right]
  {\sf e}_k
  \right).
\end{aligned}\end{equation}
Since $\widehat{T}_i(p,\zeta^{(J_i)})=T_i(p)$
and $\Phi\big(p,T_i(p)\big)=Q_i(p)$ for all $p\in \A_i$, \begin{equation*}
  \widehat{Q}_i(p,\zeta_i)
  = \left(
  Q_i(p),
  \sum_{k\ne J_i}\left[
    \zeta^{(k)}
    +\int_0^{T(p)}
    \frac{\partial g^{(k)}}{\partial z^{(k)}}\big(\Phi(p,\tau),z_i,0)\;d\tau
  \right]
  {\sf e}_k
  \right).
\end{equation*}
Hence \begin{equation*}\begin{aligned}
  &D\widehat{Q}_i(p,\zeta_i)(v,0)
  \\
  &\quad
  =\left(
    DQ_i(p)v,
    \left\langle
      DT(p),v
    \right\rangle
    \sum_{j\ne J_i}
      \frac{\partial g^{(j)}}{\partial z^{(j)}}\big(\Phi(p,\tau),z_i,0)
    \,{\sf e}_j
  \right)
  \quad\forall v\in \mathbb R^n.
\end{aligned}\end{equation*}
Evaluating this equation at $p=A_i$,
by \eqref{DTi_mu}
we then obtain \eqref{DpQiHat}.

For each $j\in \{1,2,\dots,m\}\setminus\{J_i\}$,
differentiating \eqref{QiHat_ek} with respect to $\zeta^{(j)}$
gives
$\frac{\partial}{\partial \zeta^{(j)}}\widehat{Q}_i(p,\zeta)={\sf e}_j$
for all $(p,\zeta)$.
On the other hand, 
by differentiating \eqref{QiHat_ek} with respect to $\zeta^{(J_i)}$,
from the relation $\frac{\partial}{\partial \tau}\Phi(p,\tau)=f(\Phi(p,\tau))$ we obtain
\begin{equation}\label{DzQiHat_Ti}\begin{aligned}
  &\frac{\partial}{\partial \zeta^{(J_i)}}
  \widehat{Q}_i(p,\zeta)
  \\
  &\quad
  =\frac{\partial \widehat{T}(p,\zeta^{(J_i)})}{\partial \zeta^{(J_i)}}
  \left(
    f\big(\Phi(p,\widehat{T}_i(p,\zeta^{(J_i)})),z_i,0\big),
    \sum_{k\ne J_i}
    \frac{\partial g^{(k)}}{\partial z^{(k)}}\big(B_i,z_i,0)
    \,{\sf e_k}
  \right),
\end{aligned}\end{equation}
Note that differentiating \eqref{def_TiHat} 
with respect to $\zeta^{(J_i)}$
gives
\begin{equation}\label{DTiHat_Bi}
  \frac{\partial \widehat{T}_i(A_i,\zeta^{(J_i)})}{\partial \zeta^{(J_i)}}
  =\frac{-1}{\frac{\partial g^{(J_i)}}{\partial z^{(J_i)}}\big(B_i,z_i,0)}.
\end{equation}
By \eqref{DzQiHat_Ti} and \eqref{DTiHat_Bi} it follows that \begin{equation*}
  \frac{\partial}{\partial \zeta^{(J_i)}}
  \widehat{Q}_i(A_i,\zeta)
  =\frac{-1}{\frac{\partial g^{(J_i)}}{\partial z^{(J_i)}}\big(B_i,z_i,0)}
  \left(
    f(B_i,z_i,0),
    \sum_{k\ne J_i}
    \frac{\partial g^{(k)}}{\partial z^{(k)}}\big(B_i,z_i,0)
    \,{\sf e_k}
  \right).
\end{equation*}
This means that \eqref{DzQiHat} holds.
\end{proof}

Let $\Psi_i$ be the solution operator for \eqref{fast_pq}.
Let $t_i$ be the positive number such that \begin{equation*}
  \Psi_i\big(t_i,(B_{i-1},z_{i-1}^{(J_i)})\big)
  = (A_{i-1},z_{i}^{(J_i)}).
\end{equation*}
Let \begin{equation*}
  \bar{\gamma}_i(t)= \Psi_i\big(t,(B_{i-1},z_{i-1}^{(J_i)})\big),
  \quad
  0\le t\le t_i.
\end{equation*}
Thus $\bar{\gamma}$ has the same trajectory
as the curve $\gamma$ given in Assumption~\ref{hyp_fast}.

We define $R_i(t):\mathbb R^n\to \mathbb R^n$
and $\nu_i(t):\mathbb R^n\to \mathbb R$, $0\le t\le t_i$,
to be the linear operators so that
for any $v_0\in \mathbb R^n$,
$R_i(t)[v_0]$ and $\nu_i((t)[v_0]$
are the $v$- and $w$-components, respectively,
of the variational equations of \eqref{fast_pq}
along $\bar{\gamma}_i(t)$
with initial data $(v_0,0)$.
This means that
for any $(v_0,w_0)\in \mathbb R^{n}\times \mathbb R$,
$(v,w)=\big(R_i(t)[v_0],\nu_i((t)[v_0]\big)$
is the solution of \begin{equation}\label{def_R}
  \frac{d}{dt}\begin{pmatrix} v\\ w\end{pmatrix}
  =
  \begin{pmatrix} 
    D_ph_i& D_qh_i
    \\
    D_pg_i& D_qg_i
  \end{pmatrix}_{\bar{\gamma}_i(t)}
  \begin{pmatrix} v\\ w\end{pmatrix},
  \quad
  \begin{pmatrix} v\\ w\end{pmatrix}(0)
  =\begin{pmatrix} v_0\\ 0\end{pmatrix},
\end{equation}
where $g_i$ and $h_i$ are defined by \eqref{def_gi}.

\begin{proposition}
Let $\pi_i$ be defined by \eqref{def_pi}.
Then
\begin{equation}\label{Dpi_nu}
  D\pi_i(B_{i-1})[v]
  = R_i(t_i)[v]
  - \nu_i(t_i)[v]
  \frac{h_i(A_i,z_i)}{g_i(A_i,z_i)}
  \qquad\forall v\in \mathbb R^n.
\end{equation}
Moreover, if $n=1$,
then \begin{equation}\label{Dpi1}
  D\pi_i(B_{i-1})
  =\frac{g_i(B_{i-1},z_{i-1})}{g_i(A_i,z_i)}
  \exp\left(
    \int_0^{t_i} (D_ph_i+ D_qg_i)(\widetilde{\gamma}_i(t))\;dt
  \right).
\end{equation}
\label{prop_Dpi}
\end{proposition}

\begin{proof}
The first part of the proof is similar to that of Proposition~\ref{prop_DQi}.
Define $S_i:\B_{i-1}\to (0,\infty)$
implicitly by $S_i(p)=t_i$ and \begin{equation}\label{Si_ci}
  z_{i-1}^{(J_{i})}
  +\int_0^{S_i(p)}
    g_i\big(
      \Psi_i\big(t,(p,z_{i-1}^{(J_{i})}\big)
    \big)
  \;dt
  =z_{i}^{(J_{i})}.
\end{equation}
Then \begin{equation}\label{pi_Psi}
    (\pi_i(p),z_i^{(J_i)})
  = \Psi_i\big(S_i(p),(p,z_{i-1}^{(J_{i})})\big).
\end{equation}
Differentiating \eqref{Si_ci} gives
(similar to the derivation of \eqref{DTi_mu})
\begin{equation}\label{DSi_nu}
  \left\langle DS_i(p),v \right\rangle g_i(A_i,z_i^{(J_i)})
  =\nu_i(t_i)[v].
\end{equation}
Differentiating \eqref{pi_Psi} gives 
(similar to the derivation of \eqref{DQi_DTi})
\begin{equation}\label{Dpi_Ri}
  D\pi_i(p)[v]
  = R_i(t_i)[v]
  - \left\langle DS_i(p),v \right\rangle h_i(A_i,z_i^{(J_i)})
\end{equation}
By \eqref{DSi_nu} and \eqref{Dpi_Ri} we obtain \eqref{Dpi_nu}.

Now we assume $n=1$.
Then \eqref{Dpi_Ri} gives \begin{align}
  D\pi_i(B_{i-1})
  &=\frac{R_i(t_i)g_i(A_i,z_i)-\nu_i(t_i)h_i(A_i,z_i)}{g_i(A_i,z_i)},
  \notag
  \\
  &
  =\frac{1}{g_i(A_i,z_i)}
  \det\begin{pmatrix}
    R_i(t)& h_i(\widetilde{\gamma}_i(t))
    \\[.2em]
    \nu_i(t)& g_i(\widetilde{\gamma}_i(t))
  \end{pmatrix}_{t=t_i}.
  \label{Dpi_det}
\end{align}
On the other hand,
when $n=1$,
$(R_i,\nu_i)(t)$ is the solution of \eqref{def_R} with $v_0=1$.
Note that $(h_i,g_i)(\widetilde{\gamma}_i(t))$
also satisfies the differential equations in \eqref{def_R}.
Hence \begin{equation*}
  \frac{d}{dt}
  \begin{pmatrix}
    R_i(t)& h_i(\widetilde{\gamma}_i(t))
    \\[.2em]
    \nu_i(t)& g_i(\widetilde{\gamma}_i(t))
  \end{pmatrix}
  = \begin{pmatrix}
    D_pg& D_qg
    \\[.5em]
    D_ph& D_qh
  \end{pmatrix}_{(p,q)=\widetilde{\gamma}_i(t)}
  \begin{pmatrix}
    R_i(t)& h_i(\widetilde{\gamma}_i(t))
    \\[.2em]
    \nu_i(t)& g_i(\widetilde{\gamma}_i(t))
  \end{pmatrix}
\end{equation*}
and \begin{equation*}
  \begin{pmatrix}
    R_i(t)& h_i(\widetilde{\gamma}_i(t))
    \\[.2em]
    \nu_i(t)& g_i(\widetilde{\gamma}_i(t))
  \end{pmatrix}_{t=0}
  =  \begin{pmatrix}
    1& h_i(B_{i-1},z_{i-1})
    \\[.2em]
    0& g_i(B_{i-1},z_{i-1})
  \end{pmatrix}.
\end{equation*}
By Abel's formula, it follows that \begin{align}
  &\det\begin{pmatrix}
    R_i(t)& h_i(\widetilde{\gamma}_i(t))
    \\[.2em]
    \nu_i(t)& g_i(\widetilde{\gamma}_i(t))
  \end{pmatrix}_{t=t_i}
  \notag
  \\
  &\quad
  =
  \det\begin{pmatrix}
    R_i(t)& h_i(\widetilde{\gamma}_i(t))
    \\[.2em]
    \nu_i(t)& g_i(\widetilde{\gamma}_i(t))
  \end{pmatrix}_{t=0}
  \exp\left(
   \int_0^{t_i}
    \mathrm{tr}
    \begin{pmatrix}
    D_pg_i& D_qg_i
    \\[.5em]
    D_ph_i& D_qh_i
  \end{pmatrix}_{(p,q)=\widetilde{\gamma}_i(t)}
  dt
  \right)
  \notag
  \\
  &\quad
  =
  \det\begin{pmatrix}
    1& h_i(B_{i-1},z_{i-1})
    \\[.2em]
    0& g_i(B_{i-1},z_{i-1})
  \end{pmatrix}
  \exp\left(
   \int_0^{t_i}(D_pg_i+D_qh_i)(\widetilde{\gamma}_i(t)) \;dt
  \right),
  \notag
  \\
  &\quad
  =g_i(B_{i-1},z_{i-1})
  \exp\left(
   \int_0^{t_i}(D_pg_i+D_qh_i)(\widetilde{\gamma}_i(t)) \;dt
  \right).
  \label{det_exp}
\end{align}
By \eqref{Dpi_det} and \eqref{det_exp}, we then obtain \eqref{Dpi1}.
\end{proof}

\section{Proofs of the Criteria}
\label{sec_proofs}
Note that Theorem~\ref{thm_main3}
is a generalization of Theorems~\ref{thm_main2} and \ref{thm_main1}.
While Theorem~\ref{thm_main3}
can be proved
without relying on the results of the other theorems,
for clarity we prove Theorem~\ref{thm_main1} first
in Section~\ref{sec_proof_main1},
and then prove the general Theorem~\ref{thm_main3}
in Section~\ref{sec_proof_main3}.

\subsection{Proof of Theorem \ref{thm_main1}}
\label{sec_proof_main1}
In this section we assume $m=1$
for system \eqref{deq_pz},
namely $(p,z)\in \mathbb R^n\times \mathbb R$.
For each $i=1,2,\dots,N$,
on curve $\gamma_i=\{(\theta_i(t),\rho_i(t)\}$
from Assumption~\ref{hyp_fast}
the function $\rho_i$ is non-constant,
so we can choose a point $(p_{0i},q_{0i})\in \gamma_i$
at which $\dot{\rho}_i\ne 0$.
Let $\Gamma_i$ be a cross section
of $\gamma_i$ at a point $(p_{0i},z_{0i})$ of the form
\begin{equation}\notag
  \Gamma_i=\{(p,z): |p-p_{0i}|<\delta_0,\; z=q_{0i}\},
\end{equation}
where $\delta_0>0$ is to be determined.
Our strategy is to track trajectories
that evolve from $\Gamma_i$ along
the flow \eqref{deq_pz}
and reach $\Gamma_{i+1}$
near the configuration $\gamma_i\cup \sigma_i\cup \gamma_{i+1}$.
We set a cross section $\Sigma_i$
of $\sigma_i$ and analyze
the dynamics between $\Gamma_i$ and $\Sigma_i$.
By symmetry,
the dynamics between $\Sigma_i$ and $\Gamma_{i+1}$
can also be treated.
We will choose two cross sections,
$\A_i^{\IN}$ and $\A_i^{\OUT}$,
near $A_i$
to analyze the transition map from $\Gamma_i$ to $\Sigma_i$.
A list a symbols in this proof
is given in Table~\ref{table_symbols1}.
Note that we use the notation $\kappa_{\epsilon i}^{(jk)}$
for several $\epsilon$-dependent charts.
We denote $\kappa_{\epsilon i}^{(kj)}$
the inverse of $\kappa_{\epsilon i}^{(jk)}$,
and denote $\kappa_{\epsilon i}^{(jl)}=
\kappa_{\epsilon i}^{(jk)}\circ\kappa_{\epsilon i}^{(kl)}$,
whenever they are defined.

\begin{table}[htbp]
\caption{Notations in Section~\ref{sec_proof_main1}.}
\centering
\begin{tabular}{|l|l|l|}
\hline
Variables
& Charts
& Objects
\\
\hline
$(p,z)\in \Omega$
&
$\kappa_{\epsilon i}^{(12)}(p,z,\zeta)=(p,z)$
&
$\Omega$, $\Gamma_i$
\\
\quad $=\mathbb R^n\times (\zmin,\zmax)$
&
$\kappa_{\epsilon i}^{(13)}(p,\zeta)=(p,z)$
&
\\
\hline
$p\in \mathbb R^m$
&
&
$\A_i$, $\B_i$
\\
\hline
$(p,z,\zeta)\in \Omega\times \mathbb R_+$
&
$\kappa_{\epsilon i}^{(21)}(p,z)=(p,z,\zeta)$
&
$\Atilde_i$, 
$\Atilde_i^{\IN}$, 
$\Atilde_i^{\OUT}$
\\
&
$\kappa_{\epsilon i}^{(23)}(p,\zeta)=(p,z,\zeta)$
&
\\
\hline
$(p,\zeta)\in \mathbb R^n\times \mathbb R_+$
&
$\kappa_{\epsilon i}^{(31)}(p,z)=(p,\zeta)$
&
$\Ahat_i^{\OUT}$, 
$\widehat{\Sigma}_i$
\\
&
$\kappa_{\epsilon i}^{(32)}(p,z,\zeta)=(p,\zeta)$
&
\\
\hline
\end{tabular}
\label{table_symbols1}
\end{table}

Let $\omega_i$, $1\le i\le N$, be the numbers
defined in \eqref{def_omega} for $m=1$,
which means $\omega_i=\omega_i^{(1)}$.
By Assumption~\ref{hyp_nondegen},
in a neighborhood of $(A_i,z_i)$,
for $\delta_1>0$ sufficiently small,
there is a unique point $(p_i^{\IN},z_i+\omega_i\delta_1)$
that lies on the curve $\gamma_i$.
Here $\mathbb{B}(p,r)$ is the open ball centered at $p$ with radius $r$.
Let \begin{equation}\label{def_Ain1}
  \A_i^{\IN}
  =\{(p,z): p\in \mathbb{B}(p_i^{\IN},\delta_2), z=z_i+\omega_i\delta_1\},
\end{equation}
where $\delta_1$ and $\delta_2$ are positive constants to be determined.

\begin{proposition}
Let $\Gamma_i$ and $\A_i^{\IN}$ be defined
as in the preceding paragraphs.
For fixed $\delta_1>0$ and $\delta_2>0$,
if $\delta_0>0$ is sufficiently small,
then the transition map
$\Pi_{\epsilon\Gamma_i}^{\A_i^{\IN}}$
from $\Gamma_i$ to $\A_i^{\IN}$
for system \eqref{deq_pz}
is well-defined
for all small $\epsilon\ge 0$.
Moreover, \[
  \left\|
    \Pi_{\epsilon\Gamma_i}^{\A_i^{\IN}}- \Pi_{0\Gamma_i}^{\A_i^{\IN}}
  \right\|_{C^1(\Gamma_i)}
  =O(\epsilon)
  \quad\text{as $\epsilon\to 0$},
\]
that is, 
$\Pi_{\epsilon\Gamma_i}^{\A_i^{\IN}}$
is $O(\epsilon)$-close
to $\Pi_{0\Gamma_i}^{\A_i^{\IN}}$
in the $C^1(\Gamma_i)$-norm
as $\epsilon\to 0$.
\label{prop_K1m1}
\end{proposition}

\begin{proof}
Since \eqref{deq_pz} is a regular perturbation of \eqref{fast_pz},
the results follow directly from regular perturbation theory.
\end{proof}

Next we investigate the dynamics near $\sigma_i$.
Let $\Omega=\mathbb R^n\times (\zmin,\zmax)$.
We define
an $\epsilon$-dependent
chart $\kappa_{\epsilon i}^{(31)}$ on $\Omega$
by \begin{equation}\notag
  \kappa_{\epsilon i}^{(31)}(p,z)= (p,\zeta)
  \quad\text{with}\quad
  \zeta=\epsilon\ln\left(\frac{\omega_i}{z-z_i}\right).
\end{equation}
In this chart
system \eqref{deq_pz} 
is converted to
\begin{equation}\label{deq_k3m1}\begin{aligned}
  &{p}'= f(p,z,\epsilon)+ h(p,z,\epsilon)/\epsilon,
  \\
  &{\zeta}'= -\omega_i\,\frac{g(p,z,\epsilon)}{z-z_i},
  \\
  &\text{where }
  z=z_i+\epsilon\,\omega_i\exp(-\zeta_i/\epsilon).
\end{aligned}\end{equation}
Formally,
the limit of \eqref{deq_k3m1}
as $\epsilon\to 0$ with $z=z_i+o(\epsilon)$ is
\begin{equation}\label{slow_k3m1}\begin{aligned}
  &p'= f(p,z_i,0),
  \\
  &\zeta'= -\omega_i\,\frac{\partial g}{\partial z}(p,z_i,0).
\end{aligned}\end{equation}
Let $\widehat{\Phi}_i$ to be
the solution operator of \eqref{slow_k3m1}.
Let \begin{equation}\label{def_Ai1}
  \A_i
  =\mathbb{B}(A_i,\delta_4)
\end{equation}
and \begin{equation}\label{def_Ahat1}
  \Ahat_i^{\OUT}
  =\widehat{\Phi}_i(\A_i\times \{0\},\delta_3),
\end{equation}
where $\delta_3>0$ and $\delta_4>0$ are constants to be determined.
Let
$\widehat{\sigma}_i(\tau)=\widehat{\Phi}_i((A_i,\zeta_i),\tau)$,
$0\le \tau\le T_i$.
Let $\widehat{\Sigma}_i$
be a cross section of the curve $\widehat{\sigma}_i$
at $\widehat{\sigma}_i(T_i/2)$
in $\mathbb R^n\times \mathbb R$.
We denote $\Pi_{0\,\Ahat_i^{\OUT}}^{\widehat{\Sigma}}$
the transition map from $\Ahat_i^{\OUT}$ to $\widehat{\Sigma}_i$
following the flow of \eqref{slow_k3m1}.

\begin{proposition}
Let $\A_i$ and $\Ahat_i^{\OUT}$
be defined as in the preceding paragraphs.
For fixed $\delta_3>0$,
if $\delta_4>0$ is sufficiently small,
then the transition map
$\Pi_{\epsilon\,\Ahat_i^{\OUT}}^{\widehat{\Sigma}}$
from $\Ahat_i^{\OUT}$ to $\widehat{\Sigma}_i$
for system \eqref{deq_k3m1}
is well-defined
for all small $\epsilon> 0$.
Moreover, $\Pi_{\epsilon\,\Ahat_i^{\OUT}}^{\widehat{\Sigma}}$
is $O(\epsilon)$-close
to $\Pi_{0\,\Ahat_i^{\OUT}}^{\widehat{\Sigma}}$
in the $C^1(\Ahat_i^{\OUT})$-norm
as $\epsilon\to 0$.
\label{prop_K3m1}
\end{proposition}

\begin{proof}
Let $\Sigma$
be the image of $\widehat{\Sigma}$ via the projection $(p,\zeta)\mapsto p$.
Since the trajectory $\sigma_i$ of \eqref{slow_pz} 
connects $A_i$ and $\Sigma_i$,
we can choose $\Delta>0$
such that 
the transition map from $\A_i$
to $\Sigma_i$
whenever
$\delta_4>0$ is sufficiently small.

Note that the $p$-component of $\widehat{\Phi}_i(A_i,\tau)$
equals $\sigma_i(\tau)=\Phi_i(A_i,\tau)$ in Assumption~\ref{hyp_slow}.
Also note that Assumption~\ref{hyp_int} gives \begin{equation*}
  \inf \left\{
    \zeta: (p,\zeta)\in \widehat{\Phi}_i((A_i,0),\tau),\;
    \tau\in [\delta_3,\tau_i-\delta_3]
  \right\}
  >0.
\end{equation*}
Therefore,
by decreasing $\Delta$
if necessary,
for $\A_i$ defined by \eqref{def_Ai1} with $\delta_3\in (0,\Delta)$,
\begin{equation}\label{est_zeta1}
  \inf \left\{
    \zeta: (p,\zeta)\in \widehat{\Phi}_i((p_0,0),\tau),\;
    p_0\in \A_i,\;
    \tau\in [\delta_3,\tau_i-\delta_3]
  \right\}>C
\end{equation} for some $C>0$.
Substituting \eqref{est_zeta1} into \eqref{deq_k3m1}, we have
\begin{equation}\label{deq_k30m1}\begin{aligned}
  &{p}'= f(p,z_i,0)+ O\big(\epsilon+ e^{-C/\epsilon}/\epsilon\big),
  \\
  &{\zeta}'= -\omega_i\frac{\partial g}{\partial z}(p,z_i,0)
  + O(\epsilon).
\end{aligned}\end{equation}
Hence
\eqref{deq_k3m1}
is a regular perturbation 
of \eqref{slow_k3m1}
in a neighborhood of the set
\begin{equation*}
  \{\widehat{\Phi}(x,\tau):
  x\in \Ahat_i^{\OUT},\tau\in [0,\tau_i-2\delta_3]\}.
\end{equation*}
Therefore, 
by regular perturbation theory,
$\Pi_{\epsilon\,\Ahat_i^{\OUT}}^{\widehat{\Sigma}}$
is well-defined for small $\epsilon>0$
and 
is $O(\epsilon)$-close
to $\Pi_{0\,\Ahat_i^{\OUT}}^{\widehat{\Sigma}}$
in the $C^1(\Ahat_i^{\OUT})$-norm
as $\epsilon\to 0$.
\end{proof}

Finally we investigate the dynamics
near the joint of $\gamma_1$ and $\sigma_i$.
We define
\begin{equation}\notag
  \kappa_{\epsilon i}^{(21)}(p,z)= (p,z,\zeta)
  \quad\text{with}\quad
  \zeta=\epsilon\ln\left(\frac{\omega_i}{z-z_i}\right)
  \quad\text{for}\quad
  (p,z)\in \Omega,\;
  \epsilon\ge 0.
\end{equation}
Note that $\kappa_{\epsilon i}^{(21)}(p,z)=(p,z,\zeta)$
can be obtained by appending $z$ to $\kappa_{\epsilon i}^{(31)}(p,z)=(p,\zeta)$.
The transformation $\kappa_{\epsilon i}^{(21)}$
converts system \eqref{deq_pz} to
\begin{equation}\label{deq_k2m1}\begin{aligned}
  &\dot{p}= \epsilon f(p,z,\epsilon)+ h(p,z,\epsilon),
  \\
  &\dot{z}= g(p,z,\epsilon),
  \\
  &\dot{\zeta}= -\epsilon\,\omega_i\,\frac{g(p,z,\epsilon)}{z-z_i}.
\end{aligned}\end{equation}
We define \begin{equation}\label{def_AtildeIN1}
  \Atilde_{\epsilon i}^{\IN}
  =\kappa_{\epsilon i}^{(12)}(\A_i^{\IN})
  \quad\text{for $\epsilon\ge 0$},
\end{equation}
which means \[
  \Atilde_{\epsilon i}^{\IN}
  =\left\{
    (p,z,\zeta): p\in \mathbb B(p^{\IN}_{0i},\delta_2),\;
    z=z_i+\omega_i\delta_1,\;
    \zeta=\epsilon\ln \delta_1
  \right\}.
\]
Note that $\kappa_{0 i}^{(21)}(p,z)=(p,z,0)$ for all $(p,z)\in \A_i^{\IN}$.

Taking $\epsilon\to 0$ in \eqref{deq_k2m1}
leads to the system \eqref{fast_pz}
companioned with $\dot{\zeta}=0$.
By Assumptions~\ref{hyp_fast} and \ref{hyp_nondegen},
the projection \begin{equation*}
  \Pi_{0\,\A_i^{\IN}}^{\A_i}
  : \A_i^{\IN}\to {\A_i}\times \{z_i\}
\end{equation*}
following the flow of \eqref{fast_pz} is well-defined
and is a local homeomorphism.
We define
$\Pi_{0\,\Atilde_{0i}^{\IN}}^{\Atilde_{0i}}
=\Pi_{0\,\A_i^{\IN}}^{\A_i}\times \mathrm{id}$,
which means \begin{equation*}
  \Pi_{0\,\Atilde_{0i}^{\IN}}^{\Atilde_{0i}}(p,z,\zeta_i)
  = \big(\Pi_{0\,\A_i}^{\A_i}(p,z),\zeta_i\big).
\end{equation*}

In the slow time variable $\tau=\epsilon t$,
taking $\epsilon\to 0$
in \eqref{deq_k2m1} 
with $z=z_i+o(\epsilon)$
leads to \eqref{slow_k3m1}
appended by the equation $z=z_i$.
We define $\widetilde{\Phi}_i\big((p,z_i,\zeta),\tau\big)$
on $\Atilde_{0i}\times [0,\tau_i]$
to be the image of $\widehat{\Phi}\big((p,\zeta),\tau\big)$
in the space $\{(p,z,\zeta): z=z_i\}$.
Also we define
\begin{equation}\notag
  \Atilde_{\epsilon i}^{\OUT}
  =\kappa_{\epsilon i}^{(23)}\big(\Ahat_i^{\OUT}\big)
  \quad\text{for}\;
  \epsilon> 0.
\end{equation}
Note that \begin{equation*}
  \Pi_{0\,\Atilde_{0i}}^{\Atilde_{0i}^{\OUT}}
  =\widetilde{\Phi}_{i}(\cdot,\delta_4).
\end{equation*}

\begin{proposition}
There exists $\Delta>0$ such that
the following assertions hold.
Let $\Atilde_i^{\IN}$ and $\Atilde_i^{\OUT}$
be defined as in the preceding paragraphs
with $\delta_j<\Delta$, $j=1,2,3$,
then for all sufficiently small $\delta_4>0$,
the transition map 
$\Pi_{\epsilon\,\Atilde_{\epsilon i}^{\IN}}^{\Atilde_{\epsilon i}^{\OUT}}$
from $\Atilde_{\epsilon i}^{\IN}$ to $\Atilde_{\epsilon i}^{\OUT}$
following the flow of \eqref{deq_k2m1} is well-defined
for all small $\epsilon>0$.
Moreover,
\begin{equation}\label{trans_K2m1}
  \left\|
    \Pi_{\epsilon\,\Atilde_{\epsilon i}^{\IN}}^{\Atilde_{\epsilon i}^{\OUT}}
    \circ
    \kappa_{\epsilon i}^{(21)}
    -
    \Pi_{0\,\Atilde_{0i}}^{\Atilde_{0i}^{\OUT}}
    \circ 
    \Pi_{0\,\Atilde_{0 i}^{\IN}}^{\Atilde_{0i}}
    \circ \kappa_{0i}^{(21)}
  \right\|_{C^1(\A_i^{\IN})}
  =O(\epsilon)
\end{equation}
as $\epsilon\to 0$.
\label{prop_K2m1}
\end{proposition}

A schematic diagram 
representing Proposition~\ref{prop_K2m1}
is shown in Figure~\ref{fig_K2m1}.
The significance
in estimate \eqref{trans_K2m1} is that
the transition map
$\Pi_{\epsilon\,\A_{\epsilon i}^{\IN}}^{\A_{\epsilon i}^{\OUT}}$
can be approximated
by the composition function of
$\Pi_{0\,\Atilde_{0i}}^{\Atilde_{0i}^{\OUT}}$
and $\Pi_{0\,\Atilde_{0 i}^{\IN}}^{\Atilde_{0i}}$,
which are determined only by the limiting systems.
To prove Proposition~\ref{prop_K2m1},
we need the following lemma,
which is a variation of 
the Exchange Lemma
in Jones and Tin \cite{Jones:2009}
and Schecter~\cite{Schecter:2008b}.

\begin{figure}[htbp]
\centering
\begin{tikzcd}
  \A_i^{\IN}
  \arrow[r,hook,"\kappa_{\epsilon i}^{(21)}"]
  \arrow[d,equal]
  & \Atilde_{\epsilon i}^{\IN}
  \arrow[d,rightsquigarrow]
  \arrow[rr,"\displaystyle%
  \Pi_{\epsilon\,\Atilde_{\epsilon i}^{\IN}}^{\Atilde_{\epsilon i}^{\OUT}}"]
  &[2em] 
  \mbox{}
  &[2em] 
  \Atilde_{\epsilon i}^{\OUT} \arrow[d,equal]
  \\
  \A_i^{\IN}
  \arrow[r,hook,"\kappa_{0 i}^{(21)}"]
  & \Atilde_{0i}^{\IN}
  \arrow[r,"\displaystyle\Pi_{0\,\Atilde_{0i}^{\IN}}^{\Atilde_{0i}}"] 
  & \Atilde_{0i}
  \arrow[r,"\displaystyle
  \Pi_{0\,\Atilde_{0i}}^{\Atilde_{0i}^{\OUT}}"]
  & \Atilde_{\epsilon i}^{\OUT}
\end{tikzcd}
\caption{A schematic diagram representing Proposition~\ref{prop_K2m1}.
Here $\hookrightarrow$ indicates injection
and $\rightsquigarrow$ indicates
the limit as $\epsilon\to 0$.
The transition map
from $\widetilde{\mathscr{A}}_{\epsilon i}^{\IN}$
to $\widetilde{\mathscr{A}}_{\epsilon i}^{\OUT}$
along \eqref{deq_k2m1}
is approximated
by the composition function of
the transition maps for
the limiting systems.
}
\label{fig_K2m1}
\end{figure}
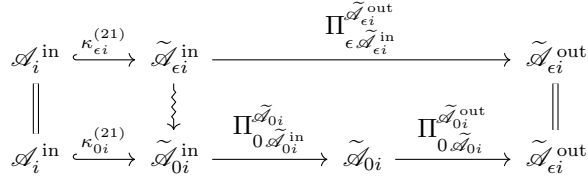

\begin{lemma}\label{lem_EL}
Consider a system for
$(a,b)\in \mathbb R^n\times \mathbb R$, $N\ge 1$,
of the form
\begin{equation}\label{deq_bc}\begin{aligned}
  &\dot{a}= \epsilon f(a,b,\epsilon)+b\,h(a,b,\epsilon),
  \\
  &\dot{b}= b\,g(a,b,\epsilon),
\end{aligned}\end{equation}
where $\cdot$ denotes $\frac{d}{dt}$,
and $f$, $g$ and $h$ are smooth functions.
Assume
\begin{equation}\label{cond_rho}
  \sup g(a,b,\epsilon)<0.
\end{equation}
Assume that for some $\bar{a}\in \mathbb R^N$
the point $(\bar{a},0)$
is the omega limit point of
a trajectory $\gamma$
of the system
\begin{equation}\label{fast_bc}\begin{aligned}
  &\dot a= b\, h(b,c,0),
  \\
  &\dot b= b\,g(b,c,0),
\end{aligned}\end{equation}
Then there exists $\Delta>0$ such that the following assertions hold.

Let $\{\A_\epsilon^{\IN}\}_{\epsilon\in [0,\epsilon_0]}$
be a smooth family of $\ell$-dimensional manifolds, $0\le \ell\le N$,
that intersects $\gamma$
at a point in $\mathbb{B}((a_0,0),\Delta)$.
Let $\Lambda\subset \mathbb R^n$
be the projection of $\A_0^{\IN}$
along the flow of system \eqref{fast_bc}.
Let $\Phi$ the solution operator
for the system \begin{equation}\label{slow_bc}
  \frac{d}{d\tau}a= f(a,0,0).
\end{equation}
Assume the following conditions hold.
\begin{enumerate}
\item[$\mathrm{(i)}$]
$\A_0^{\IN}$ is non-tangential to the flow of \eqref{fast_bc};
\item[$\mathrm{(ii)}$]
$\bar{a}\in \Lambda$ and
$\Lambda$ is compact
and is non-tangential to the flow of \eqref{slow_bc};
\item[$\mathrm{(iii)}$]
The trajectory
$\sigma=\Phi([0,\tau_1],\bar{a})$,
where $\tau_1>0$,
lies in $\mathbb{B}(a_0,\Delta)$
and is rectifiable and not self-intersecting.
\end{enumerate}
Let $\iota_\epsilon:K\to \A_\epsilon^{\IN}$
be a smooth parameterization of $\A_\epsilon^{\IN}$
for $\epsilon\in [0,\epsilon_0]$,
where $K$ is an $\ell$-dimensional manifold.
Let $\bar{x}\in \A_0\cap \gamma$
be the pre-image of $\bar{a}$ along \eqref{fast_bc}
and $\bar{k}\in K$ be the pre-image of $\bar{x}$ by $\iota_0$.

If $\A^{\OUT}$ is an $n$-dimensional manifold
that intersects transversally at an interior point of $\sigma$,
then there is an open neighborhood $V$ of $\bar{k}$ in $K$
such that the transition map
$\Pi_{\epsilon\,\A_\epsilon^{\IN}}^{\A^{\OUT}}$
from $\iota_\epsilon(V)\subset \A_\epsilon^{\IN}$ 
to $\A^{\OUT}$
following the flow of \eqref{deq_bc}
is well defined
for all sufficiently small $\epsilon>0$.
Moreover,
\begin{equation}\label{C1_EL}
  \|
    \Pi_{\epsilon\,\A_\epsilon}^{\A^{\OUT}}
    \circ
    \iota_\epsilon
    - 
    \Pi_{0\Lambda}^{\A^{\OUT}}
    \circ
    \Pi_{0\,\A_0}^{\Lambda}
    \circ
    \iota_{0}
  \|_{C^1(V)}
  = O(\epsilon)
\end{equation} as $\epsilon\to 0$,
where $\Pi_{0\,\A_0}^{\Lambda}$ is
the transition map 
from $\A_0$ to $\Lambda$
along the flow of \eqref{fast_bc},
and $\Pi_{0\Lambda}^{\A^{\OUT}}$ is
the transition map 
from $\Lambda$ to $\A^{\OUT}\cap \{b=0\}$
along the flow of \eqref{slow_bc}.
\end{lemma}

\begin{proof}[Proof of Lemma~\ref{lem_EL}]
Using a Fenichel type coordinate (see Jones~\cite{Jones:1995}),
in the open ball $\mathbb B(0,2\Delta)$ in the $(a,b)$-space,
for sufficiently small $\Delta>0$
we can choose an $\epsilon$-dependent
change of variable $(a,b)\mapsto (\tilde{a},\tilde{b})$
with \begin{equation*}
  (\tilde{a},\tilde{b})\big|_{b=0}
  =(a,0)
\end{equation*}
such that,
after dropping the tilde symbol, 
system \eqref{deq_bc} is converted to
\begin{equation}\label{deq_bc_fenichel}\begin{aligned}
  &\dot{a}= \epsilon f(a,\epsilon),
  \\
  &\dot{b}= b\,g(a,b,\epsilon).
\end{aligned}\end{equation}
We write \begin{equation*}
  \A_\epsilon^{\IN}
  =\{({a},{b}): {a}\in \Lambda, b=\beta_\epsilon(a)\}.
\end{equation*}
Since $\A^{\OUT}$ intersects $\sigma$ transversally,
for some neighborhood $U$ of $\bar{a}$ in $\mathbb R^n$,
we can write \begin{equation*}
  \Pi_{0\Lambda}^{\A^{\OUT}}(a)
  =\Phi(a,T_0(a))
  \quad\forall\;a\in \Lambda\cap U,
\end{equation*}
where $T_0$ is a smooth function
with $\tau_-<T_0<\tau_+$
for some $\tau_-,\tau_+\in (0,\tau_1)$.
To prove \eqref{C1_EL},
it suffices to show that
\begin{equation}\label{est_bc_Pi}
  \left\|
  \Pi_{\epsilon \A_{\epsilon}^{\IN}}^{\A^{\OUT}}(a,\beta_\epsilon(a))
  -\big(
    \Phi(a,T_0(a)),0
  \big)
  \right\|_{C^1(\Lambda\cap U)}
  =O(\epsilon)
\end{equation} as $\epsilon\to 0$.
Let $(a_{\epsilon},b_{\epsilon})(t;a_0)$
be the solution of \eqref{deq_bc_fenichel}
at time $t$ with initial data $(a_0,\beta_\epsilon(a_0))$.
Define \begin{equation}\label{def_bc1}
  (a_{\epsilon1},b_{\epsilon1})(a_0,\tau)
  =(a_\epsilon,b_\epsilon)
  (\tau/\epsilon; a_0)
  \quad\text{for }\;
  a_0\in \Lambda_1,
  \tau\in [\tau_{-},\tau_{+}].
\end{equation}
By the General Exchange Lemma
(see Schecter~\cite{Schecter:2008b}),
\begin{equation}\label{est_bc_Phi}
  \left\|
    (a_{\epsilon1},b_{\epsilon1})(a_0,\tau)
    - \big(\Phi(a_0,\tau),0\big)
  \right\|_{C^1(\Lambda_1\times [\tau_{-},\tau_{+}])}
  = O(\epsilon)
\end{equation} as $\epsilon\to 0$.
Since the graph of $\big(\Phi(a_0,\tau),0\big)$
is transversal to $\A^{\OUT}$,
it follows from the Implicit Function Theorem
that there exists a function $T_\epsilon(a_0)$
defined for all small $\epsilon>0$
such that \begin{equation}\label{est_bc_T}
  \|T_\epsilon-T_0\|_{C^1(\Lambda\cap U)}
  =O(\epsilon)
\end{equation} and \begin{equation*}
  (a_{\epsilon1},b_{\epsilon1})(a_0,T_\epsilon(a_0))
  \in \A^{\OUT}
  \quad\forall\; a_0\in\Lambda\cap U.
\end{equation*}
Note that the last relation
means \begin{equation}\label{trans_bc_Pi}
  \Pi_{\epsilon \A^{\IN}}^{\A^{\OUT}}(a_0,\delta)
  =(a_{\epsilon1},b_{\epsilon1})(a_0,T_\epsilon(a_0)).
\end{equation}
From \eqref{est_bc_Phi}, \eqref{est_bc_T} and \eqref{trans_bc_Pi}
we then obtain \eqref{est_bc_Pi}.
\end{proof}

\begin{proof}[Proof of Proposition~\ref{prop_K2m1}]
Note that \eqref{deq_bc}
can be written as \begin{equation*}\begin{aligned}
  &\dot{a}= \epsilon f(a,b,0)+b\,h(a,b,0)+ O(|(a,b)|^2),
  \\
  &\dot{b}= b\,g(a,b,0)+ O(|(a,b)|^2).
\end{aligned}\end{equation*}
For system \eqref{deq_k2m1},
setting $s=z-z_i$ yields \begin{equation*}\begin{aligned}
  &\dot{p}
  =\epsilon f(p,z_i+s,0)
  + O(|(\epsilon,s)|^2),
  \\[.2em]
  &\dot{s}
  = s \frac{\partial g}{\partial z}(p,z_i,0)
  + O(|(\epsilon,s)|^2),
  \\[.2em]
  &\dot{\zeta}
  =\epsilon
  \frac{\partial g}{\partial z}(p,z_i,0)
  + O(|(\epsilon,s)|^2)
\end{aligned}\end{equation*}
as $(\epsilon,s)\to 0$.
Since $\frac{\partial g}{\partial z}(A_i,z_i,0)<0$
by Assumption~\ref{hyp_int},
applying Lemma~\ref{lem_EL}
with $b=z$ and $a=(p,s)$
we obtain \eqref{trans_K2m1}.
\end{proof}

Also we denote
$\Pi_{0\Gamma_i}^{\A_i}$
the transition map from $\Gamma_i$ to $\A_i\times \{z_i\}$
along the flow of \eqref{fast_pz}
and $\Pi_{0\,\Ahat_i}^{\widehat{\Sigma}_i}$
the transition map from ${0\,\Ahat_i}$ to ${\widehat{\Sigma}_i}$
along the flow of \eqref{slow_k3m1}.

\begin{proposition}
There exist $\delta_j>0$, $0\le j\le 4$, such that
if $\Gamma_i$, $\A_i$, $\Sigma_i$
are defined in the preceding paragraphs,
then the transition map $\Pi_{\epsilon\Gamma_i}^{{\Sigma}_i}$
from $\Gamma_i$ to ${\Sigma}_i$
following the flow of \eqref{deq_pz} is well-defined
for all small $\epsilon>0$,
and
\begin{equation}\label{trans_Am1}
  \left\|
    \kappa_{\epsilon i}^{(31)}
    \circ
    \Pi_{\epsilon\Gamma_i}^{\Sigma_i}
    -
    \Pi_{0\,\Ahat_i}^{\widehat{\Sigma}_i}
    \circ \kappa_{0 i}^{(31)}
    \circ \Pi_{0\Gamma_i}^{\A_i}
  \right\|_{C^1(\Gamma_i)}
  =O(\epsilon)
\end{equation}
as $\epsilon\to 0$.
\label{prop_trans_Am1}
\end{proposition}

\begin{proof}[Proof of Proposition \ref{prop_trans_Am1}]
First we fix constants
$\delta_1, \delta_2$ and $\delta_3$
in $(0,\Delta)$,
where $\Delta$ is the numbers
in Propositions~\ref{prop_K2m1}.
Then we choose positive constants $\delta_0$ and $\delta_4$,
such that the results in Propositions~\ref{prop_K1m1} and \ref{prop_K3m1} hold.
Then
\begin{equation*}\begin{aligned}
  \Pi_{\epsilon\Gamma_i}^{\Sigma_i}
  &=\Pi_{\epsilon\,\A_i^{\OUT}}^{\Sigma_i}
  \circ
  \Pi_{\epsilon\,\A_i^{\IN}}^{\A_i^{\OUT}}
  \circ
  \Pi_{\epsilon\Gamma_i}^{\A_i^{\IN}}
  \\
  &=
  \big(
    \kappa_{\epsilon i}^{(13)}
    \circ
    \Pi_{\epsilon\,\Ahat_i^{\OUT}}^{\widehat{\Sigma}_i}
    \circ
    \kappa_{\epsilon i}^{(31)}
  \big)
  \circ
  \big(
    \kappa_{\epsilon i}^{(12)}
    \circ
    \Pi_{\epsilon\,\Atilde_{\epsilon i}^{\IN}}^{\Atilde_{\epsilon i}^{\OUT}}
    \circ
    \kappa_{\epsilon i}^{(21)}
  \big)
  \circ
  \Pi_{\epsilon\Gamma_i}^{\A_i^{\IN}}.
\end{aligned}\end{equation*}
From Propositions~\ref{prop_K1m1},~\ref{prop_K3m1} and~\ref{prop_K2m1},
it follows that
\begin{equation*}\begin{aligned}
  \Pi_{\epsilon\Gamma_i}^{\Sigma_i}
  &=
  \left(
    \kappa_{0 i}^{(13)}
    \circ
    \Pi_{0\,\Ahat_i^{\OUT}}^{\widehat{\Sigma}_i}
    \circ
    \kappa_{0 i}^{(31)}
  \right)
  \circ
  \left(
    \kappa_{0i}^{(12)}
    \circ
    \Pi_{0\,\Atilde_i}^{\Atilde_i^{\OUT}}
    \circ
    \Pi_{0\,\Atilde_{0 i}^{\IN}}^{\Atilde_{0i}}
    \circ \kappa_{0i}^{(21)}
  \right)
  +O(\epsilon)
  \\
  &=
  \kappa_{0 i}^{(13)}
  \circ
  \left(
    \Pi_{0\,\Ahat_i^{\OUT}}^{\widehat{\Sigma}_i}
    \circ
    \kappa_{0 i}^{(32)}
    \circ
    \Pi_{0\,\Atilde_i}^{\Atilde_i^{\OUT}}
  \right)
  \circ
  \left( 
    \Pi_{0\,\Atilde_{0 i}^{\IN}}^{\Atilde_{0i}}
    \circ \kappa_{0i}^{(21)}
    \circ
    \Pi_{0\Gamma_i}^{\A_i^{\IN}}
  \right)
  +O(\epsilon).
\end{aligned}\end{equation*}
Since \begin{equation*}
  \Pi_{0\,\Ahat_i^{\OUT}}^{\widehat{\Sigma}_i}
  \circ
  \kappa_{0 i}^{(32)}
  \circ
  \Pi_{0\,\Atilde_i}^{\Atilde_i^{\OUT}}
  =
  \Pi_{0\,\Ahat_i}^{\widehat{\Sigma}_i}
  \circ
  \kappa_{0i}^{(23)}
\end{equation*}
and \begin{equation*}
  \Pi_{0\,\Atilde_{0 i}^{\IN}}^{\Atilde_{0i}}
  \circ \kappa_{0i}^{(21)}
  \circ
  \Pi_{0\Gamma_i}^{\A_i^{\IN}}
  =
  \kappa_{0i}^{(21)}
  \circ
  \Pi_{0\Gamma_i}^{\A_i},
\end{equation*}
it follows that \begin{equation*}\begin{aligned}
  \Pi_{\epsilon\Gamma_i}^{\Sigma_i}
  &=
  \kappa_{0 i}^{(13)}
  \circ
  \left(
    \Pi_{0\,\Ahat_i}^{\widehat{\Sigma}_i}
    \circ
    \kappa_{0i}^{(32)}
  \right)
  \circ
  \left(
    \kappa_{0i}^{(21)}
    \circ
    \Pi_{0\Gamma_i}^{\A_i}
  \right)
  + O(\epsilon)
  \\
  &=
  \kappa_{0 i}^{(13)}
  \circ
  \Pi_{0\,\Ahat_i}^{\widehat{\Sigma}_i}
  \circ
  \kappa_{0i}^{(31)}
  \circ
  \Pi_{0\Gamma_i}^{\A_i}
  + O(\epsilon).
\end{aligned}\end{equation*}
Applying both sides of equation by $\kappa_{0i}^{(31)}$
yields \eqref{trans_Am1}.
\end{proof}

\begin{proof}[Proof of Theorem~\ref{thm_main1}]
By a reversal of the time variable,
applying Proposition \ref{prop_trans_Am1} we obtain
\begin{equation*}
  \left\|
    \kappa_{\epsilon i}^{(31)}
    \circ
    \Pi_{\epsilon \Gamma_{i+1}}^{\Sigma_i}
    -
    \Pi_{0\Bhat_i}^{\widehat{\Sigma}_i}
    \circ 
    \kappa_{0i}^{(31)}
    \circ \Pi_{0\Gamma_{i+1}}^{\B_i}
  \right\|_{C^1(\Gamma_{i+1})}
  =O(\epsilon).
\end{equation*}
Taking the inverse of the mappings we obtain \begin{equation}\label{trans_Bm1}
  \left\|
    \Pi_{\epsilon\Gamma_i}^{\Sigma_i}
    \circ
    \kappa_{\epsilon i}^{(13)}
    -
    \Pi_{0\B_i}^{\Gamma_{i+1}}
    \circ 
    \kappa_{0i}^{(13)}
    \circ
    \Pi_{0{\widehat{\Sigma}_i}}^{\Bhat_i}
  \right\|_{C^1(\widehat{\Sigma}_i)}
  =O(\epsilon).
\end{equation}
By \eqref{trans_Am1} and \eqref{trans_Bm1}, 
it follows that
\begin{equation}\label{Pi_Gamma_ii1}\begin{aligned}
  \Pi_{\epsilon\Gamma_i}^{\Gamma_{i+1}}
  &=\left(\Pi_{\epsilon\Gamma_i}^{\Sigma_i}\circ \kappa_{i\epsilon}^{(13)}\right)
  \circ
  \left(\kappa_{i\epsilon}^{(31)}\circ \Pi_{\epsilon\Gamma_i}^{\Sigma_i}\right)
  \\
  &=\left(
    \Pi_{0\B_i}^{\Gamma_{i+1}}
    \circ 
    \kappa_{0i}^{(13)}
    \circ
    \Pi_{0{\widehat{\Sigma}_i}}^{\Bhat_i}
  \right)
  \circ
  \left(
    \Pi_{0\,\Ahat_i}^{\widehat{\Sigma}_i}
    \circ \kappa_{0 i}^{(31)}
    \circ \Pi_{0\Gamma_i}^{\A_i}
  \right)
  +O(\epsilon)
  \\
  &= 
  \Pi_{0\B_i}^{\Gamma_{i+1}}
  \circ 
  \kappa_{0i}^{(13)}
  \circ
  \Pi_{0\,\Ahat_i}^{\Bhat_i}
  \circ \kappa_{0 i}^{(31)}
  \circ \Pi_{0\Gamma_i}^{\A_i}
  +O(\epsilon).
\end{aligned}\end{equation}
Define $\varrho(p,z)=p$.
Since we assumed $h=0$ in \eqref{fast_pz}, \begin{equation*}
  \varrho\circ
  \Pi_{\epsilon\B_i}^{\A_{i+1}}(p,z)
  =p
  \quad\forall\;(p,z)\in \B_i.
\end{equation*}
Hence \eqref{Pi_Gamma_ii1} implies that \begin{equation*}
  \varrho\circ \Pi_{\epsilon\Gamma_i}^{\Gamma_{i+1}}
  =
  Q_i
  +O(\epsilon).
\end{equation*}
Let \begin{equation*}
  P_\epsilon
  =\Pi_{\epsilon\Gamma_{N}}^{\Gamma_1}
  \circ
  \cdots
  \circ
  \Pi_{\epsilon\Gamma_2}^{\Gamma_3}
  \circ
  \Pi_{\epsilon\Gamma_1}^{\Gamma_2}.
\end{equation*}
Then \begin{equation*}
  \varrho
  \circ
  P_\epsilon
  = Q_N\circ \cdots \circ Q_2\circ Q_1
  + O(\epsilon)
  = P+ O(\epsilon),
\end{equation*}
where $P$ is defined by \eqref{def_P}.
Since the $z$-component on $\Gamma_1$ is a constant,
we conclude that \begin{equation*}
  \det\left(
  P_\epsilon
  -\mathrm{id}
  \right)
  =
  \det\left(
    DP
    -\mathrm{id}
  \right)
  + O(\epsilon).
\end{equation*}
Hence the linearization of
the return map at $p_{01}\in \Gamma_1$
does not have a singular value equal to $1$
for all small $\epsilon>0$
if $\det(DP-\mathrm{id})\ne 0$.
Consequently,
for all small $\epsilon>0$
there exists a locally unique 
fixed point $(p_{\epsilon 1},z_{\epsilon 1})\in \Gamma_i$
of $P_\epsilon$.
The trajectory passing through $(p_{\epsilon 1},z_{\epsilon 1})$
is a periodic orbit of system \eqref{deq_pz}.
If the spectrum radius of $DP(p_{01},\zeta_1)$
is smaller (resp.\ greater) than $1$,
then $P_\epsilon$ is a contraction (resp.\ expansion),
hence the periodic orbit is orbitally asymptotically stable (resp.\ unstable).
This proves the theorem.
\end{proof}

\subsection{Proof of Theorem~\ref{thm_main3}}
\label{sec_proof_main3}
The approach in this section is to generalize
the proof of Theorem~\ref{thm_main1}.
Some notations to be used are listed in Table~\ref{table_symbols2}.

\begin{table}[htbp]
\caption{Notations in Section~\ref{sec_proof_main3}.}
\centering
\begin{tabular}{|l|l|l|}
\hline
Variables
& Charts
& Objects
\\
\hline
$(p,z)\in \Omega\subset \mathbb R^n\times \mathbb R^m$
&
$\kappa_{\epsilon i}^{(01)}(p,q,\widehat{\zeta})=(p,z)$
&
$\Omega$,
$\overline{\Gamma}_i$
\\
\quad\text{with }$z^{(j)}\in (\zmin^{(j)},\zmax^{(j)})$
&
$\kappa_{\epsilon i}^{(03)}(p,\zeta)=(p,z)$
&
\\
\hline
$p\in \mathbb R^m$
&
&
$\A_i$, $\B_i$
\\
\hline
$(p,q,\widehat{\zeta})$
&
$\kappa_{\epsilon i}^{(10)}(p,z)=(p,q,\widehat{\zeta})$
&
$\Gamma_i$,
\\
$\quad\in
\mathbb R^n\times \big(\zmin^{(j)},\zmax^{(j)}\big)\times \mathbb R^{m-1}_+$
&
$\kappa_{\epsilon i}^{(12)}(p,q,\zeta)=(p,z,\widehat{\zeta})$
&
$\A_i^{\IN}$, 
$\A_i^{\OUT}$
\\
\hline
$(p,q,\zeta)$
&
$\kappa_{\epsilon i}^{(21)}(p,q,\widehat{\zeta})=(p,q,\zeta)$
&
$\Atilde_i$, 
$\Atilde_i^{\IN}$, 
$\Atilde_i^{\OUT}$
\\
$\quad\in
\mathbb R^n\times \big(\zmin^{(j)},\zmax^{(j)}\big)\times \mathbb R^{m}_+$
&
$\kappa_{\epsilon i}^{(23)}(p,\zeta)=(p,q,\zeta)$
&
\\
\hline
$(p,\zeta)\in \mathbb R^n\times \mathbb R^m_+$
&
$\kappa_{\epsilon i}^{(30)}(p,z)=(p,\zeta)$
&
$\Ahat_i^{\OUT}$, 
$\widehat{\Sigma}_i$
\\
&
$\kappa_{\epsilon i}^{(32)}(p,q,\zeta)=(p,\zeta)$
&
\\
\hline
\end{tabular}
\label{table_symbols2}
\end{table}

Let \begin{equation*}
  \Omega=
  \mathbb R^n\times \left(\zmin^{(1)},\zmax^{(1)}\right)
  \times \cdots
  \times \left(\zmin^{(N)},\zmax^{(N)}\right)
  \subset \mathbb R^n\times \mathbb R^m.
\end{equation*}
We define the $\epsilon$-dependent chart on $\Omega$ by \begin{equation*}
  \kappa_{\epsilon i}^{(10)}(p,z)
  =(p,z^{(J_i)},\widehat{\zeta})
  \quad\text{with}\quad
  \widehat{\zeta}^{(j)}= \begin{cases}
    \zeta_i^{(J_i)},&\text{if }j=J_i,
    \\
    \epsilon \ln\frac{\omega_i^{(j)}}{z^{(j)}-z_{i}^{(j)}},
    &\text{if }j\ne J_i.
  \end{cases}
\end{equation*}
On the curve $(p_i(t),q_i(t))\subset \mathbb R^m\times \mathbb R$
in Assumption~\ref{hyp_fast},
since $q_i(t)$ is non-constant,
we can choose a point $(p_{0i},q_{0i})$
at which $q_i'(t)\ne 0$.
Let \begin{equation}\notag
  \Gamma_i
  =\left\{
    (p,q,\widehat{\zeta})\in 
    \mathbb R^n\times \mathbb R
    \times \Lambda_i:
    |p-p_{0i}|<\delta_0,\;
    q=q_{0i},\;
    |\widehat{\zeta}-\zeta_{i}|<\delta_0
  \right\},
\end{equation}
where $\delta_0>0$ is to be determined.
Let $\overline{\Gamma}_i=\kappa_{\epsilon 1}^{(01)}(\Gamma_i)$.
Our strategy is to track the transition map from
$\Gamma_i$
to $\Gamma_{i+1}$
in the $(p,q,\widehat{\zeta})$-space
to find a fixed point of a composition map
from $\Gamma_1$ to $\Gamma_1$
and then covert it back via $\kappa_{\epsilon 1}^{(01)}$
to obtain a periodic orbit
passing through $\overline{\Gamma}_1$
in the $(p,z)$-space.

Let \begin{equation}\notag
  \A_i^{\IN}
  =\{(p,q,\widehat{\zeta}):
  p\in \mathbb{B}(p_i^{\IN},\delta_2),\;
  q=z_i^{(J_i)}+\omega_i\delta_1,\;
  |\widehat{\zeta}-\widehat{\zeta}_i|<\delta_2
  \}.
\end{equation}
where $\delta_1$ and $\delta_2$ are positive constants to be determined.

\begin{proposition}
Let $\Gamma_i$ and $\A_i^{\IN}$ be defined
as in the preceding paragraphs.
For fixed $\delta_1>0$ and $\delta_2>0$,
if $\delta_0>0$ is sufficiently small,
then the transition map
$\Pi_{\epsilon \Gamma_i}^{\A_i^{\IN}}$
from $\Gamma_i$
to $\A_i^{\IN}$
following the flow of \eqref{deq_pz}
is well-defined
for all small $\epsilon\ge 0$
and is $O(\epsilon)$-close
to $\Pi_{0 \Gamma_i}^{\A_i^{\IN}}$
in the $C^1(\Gamma_i)$-norm
as $\epsilon\to 0$.
\label{prop_K1}
\end{proposition}

\begin{proof}
Chart $\kappa_{\epsilon i}^{(1)}$
converts system \eqref{deq_pz} to
\begin{equation}\label{deq_K1}\begin{aligned}
  &\dot{p}= \epsilon f(p,z,\epsilon)+ h(p,z,\epsilon),
  \\
  &\dot{q}= g^{(J_i)}(p,z,\epsilon),
  \\
  &\dot{\widehat{\zeta}}^{(j)}= -\epsilon\;
  \frac{g^{(j)}(p,z,\epsilon)}{z^{(j)}-z_i^{(j)}},
  \quad
  j\in \{1,2,\dots,m\}\setminus\{J_i\},
  \\
  &
  \text{with}\quad
  z^{(J_i)}=q
  \quad\text{and}\quad
  z^{(j)}=z_i^{(j)}+\omega_i^{(j)}\exp(-\widehat{\zeta}^{(j)}/\epsilon)
  \;\text{for $j\ne J_i$}.
\end{aligned}\end{equation}
By Assumption \ref{hyp_int}, 
all components of $\widehat{\zeta}_i\in \Lambda_i$
are bounded away from zero.
Therefore, 
for each $j\in \{1,2,\dots,m\}\setminus\{J_i\}$,
\begin{equation*}
  z_i^{(j)}+\omega_i^{(j)}\exp(-\widehat{\zeta}^{(j)}/\epsilon)
  \to z_i^{(j)}
  \quad\text{as $\epsilon\to 0$},
\end{equation*}
which implies
\begin{equation*}
  \frac{g^{(j)}(p,z,0)}{z^{(j)}-z_i^{(j)}}
  \to
  \frac{\partial g^{(j)}}{\partial z^{(j)}}(p,z_{i-1}+q\,{\sf e}_{J_i},0)
  \quad\text{as $\epsilon\to 0$}.
\end{equation*}
Hence the expression of
$\dot{\widehat{\zeta}}^{(j)}$
in \eqref{deq_K1}
tends to zero as $\epsilon\to 0$.
Consequently, \eqref{deq_K1} is a regular perturbation
of the system \begin{equation}\label{fast_K1}\begin{aligned}
  &\dot{p}= h(p,z_{i-1}+q\,{\sf e}_{J_{i-1}},0),
  \\
  &\dot{q}= g^{(J_i)}(p,z_{i-1}+q\,{\sf e}_{J_{i-1}},0),
  \\
  &\dot{\widehat{\zeta}}^{(j)}= 0,
  \quad
  j\in \{1,2,\dots,m\}\setminus\{J_i\}.
\end{aligned}\end{equation}
Hence $\Pi_{\epsilon \Gamma_i}^{\A_i^{\IN}}$ is well-defined
and is $O(\epsilon)$ $C^1$-close
to $\Pi_{0 \Gamma_i}^{\A_i^{\IN}}$
as $\epsilon\to 0$.
\end{proof}

We define charts $\kappa_{\epsilon i}^{(30)}$ 
for $(p,z)\in \Omega$
by \begin{equation}\notag
\begin{aligned}
  &\kappa_{\epsilon i}^{(30)}(p,z)
  =\big(p,\zeta)
  \\
  &
  \text{with}\quad
  \zeta^{(j)}=\epsilon\ln \frac{\omega_i^{(j)}}{z^{(j)}-z_i^{(j)}}
  \text{\; for }j=1,2\dots,m.
\end{aligned}\end{equation}
In this chart
system \eqref{deq_pz} 
is converted to
\begin{equation}\label{deq_K3}\begin{aligned}
  &\frac{d}{d\tau}{p}= f(p,z,\epsilon)+ h(p,z,\epsilon)/\epsilon,
  \\
  &\frac{d}{d\tau}{{\zeta}}^{(j)}= 
  \frac{-g^{(j)}(p,z,\epsilon)}{z^{(j)}-z_i^{(j)}},
  \quad
  j=1,2,\dots,m,
  \\
  &
  \quad\text{with}\quad
  z^{(j)}=z_i^{(j)}+\omega_i^{(j)}\exp(-{\zeta}^{(j)}/\epsilon)
  \;\;\text{for }j=1,2,\dots,m.
\end{aligned}\end{equation}
Let $\widehat{\Phi}_i$ be the solution operator of
\begin{equation}\label{slow_K3}\begin{aligned}
  &\frac{d}{d\tau}{p}= f(p,z_i,0),
  \\
  &\frac{d}{d\tau}{\zeta}^{(j)}=
  \frac{-\partial g^{(j)}}{\partial z^{(j)}}(p,z_i,0)
  \quad\text{for }j=1,2,\dots,m.
\end{aligned}\end{equation}
Let $\A_i$ and $\A_i^{\IN}$
be defined by \eqref{def_Ai1} and \eqref{def_Ain1}.
We define \begin{equation}\label{def_Ahat}
  \Ahat_i= \A_i\times \Lambda_i
  \quad\text{and}\quad
  \Ahat_i^{\OUT}
  =\widehat{\Phi}_i(\Ahat_i,\delta_3),
\end{equation}
where $\delta_3>0$ is a constant to be determined.
Let
$\widehat{\sigma}_i(\tau)=\widehat{\Phi}_i((A_i,\zeta_i),\tau)$,
$0\le \tau\le T_i$.
Let $\widehat{\Sigma}_i$
be a cross section of the curve $\widehat{\sigma}_i$
at $\widehat{\sigma}_i(\tau_i/2)$.
We denote $\Pi_{0\,\Ahat_i^{\OUT}}^{\widehat{\Sigma}}$
the transition map from $\Ahat_i^{\OUT}$ to $\widehat{\Sigma}_i$
following the flow of \eqref{slow_k3m1}.

\begin{proposition}
Let $\A_i$ and $\Ahat_i^{\OUT}$
be defined as in the preceding paragraphs.
For fixed $\delta_3>0$,
if $\delta_4>0$ is sufficiently small,
then the transition map
$\Pi_{\epsilon\,\Ahat_i^{\OUT}}^{\widehat{\Sigma}}$
from $\Ahat_i^{\OUT}$ to $\widehat{\Sigma}_i$
for system \eqref{deq_K3}
is well-defined
for all small $\epsilon> 0$.
Moreover, $\Pi_{\epsilon\,\Ahat_i^{\OUT}}^{\widehat{\Sigma}}$
is $O(\epsilon)$-close
to $\Pi_{0\,\Ahat_i^{\OUT}}^{\widehat{\Sigma}}$
in the $C^1(\Ahat_i^{\OUT})$-norm
as $\epsilon\to 0$.
\label{prop_K3}
\end{proposition}

\begin{proof}
By Assumption \ref{hyp_int}, 
\begin{equation*}
  \inf\left\{
    \zeta^{(j)}:
    (p,\zeta)
    =\widehat{\sigma}_i(\tau),
    \tau\in [\delta_3,\tau_i-\delta_3],
    j=1,2,\dots,m
  \right\}
  >C
\end{equation*}
for some $C>0$.
Therefore, 
similar to the proof of Proposition~\ref{prop_K1},
system \eqref{deq_K3}
is a regular perturbation of \eqref{slow_K3},
and the desired result follows.
\end{proof}

Define chart $\kappa_{\epsilon i}^{(20)}$ 
for $(p,z)\in \Omega$
by \begin{equation}\notag
\begin{aligned}
  &\kappa_{\epsilon i}^{(20)}
  \big(p,z)
  = (p,q,\zeta)
  \\
  &\text{with}\;\;
  q=z^{(J_i)}
  \;\;\text{and}\;\;
  z^{(j)}=z_{i}^{(j)}+ \omega_i^{(j)}\exp(-\widehat{\zeta}^{(j)}/\epsilon)
  \text{\; for }j=1,2\dots,m.
\end{aligned}\end{equation}
Chart $\kappa_{\epsilon i}^{(20)}$
converts system \eqref{deq_pz} to
\begin{equation}\label{deq_K2}\begin{aligned}
  &\dot{p}= \epsilon f(p,z,\epsilon)+ h(p,z,\epsilon),
  \\
  &\dot{q}= g^{(J_i)}(p,z,\epsilon),
  \\
  &{{\zeta}}^{(j)}= 
  \epsilon\; \frac{-g^{(j)}(p,z,\epsilon)}{z^{(j)}-z_i^{(j)}},
  \quad
  j=1,2,\dots,m,
  \\
  &
  \text{with}\quad
  z^{(j)}=z_i^{(j)}+\omega_i^{(j)}\exp(-\widehat{\zeta}^{(j)}/\epsilon).
\end{aligned}\end{equation}
Here we temporarily ignore the relation
$z^{(J_{i-1})}=z_{i-1}^{(J_{i-1})}+q$.
Formally the limiting slow system of \eqref{deq_K2}
at $z=z_i$ is \begin{equation}\label{slow_K2}\begin{aligned}
  &\frac{d}{d\tau}{p}= f(p,z_i0),
  \\
  &\frac{d}{d\tau}{q}= 0,
  \\
  &\frac{d}{d\tau}{\zeta}^{(j)}= 
  \frac{-\partial g^{(j)}}{\partial z^{(j)}}(p,z_i,0),
  \quad
  j=1,2,\dots,m.
\end{aligned}\end{equation}
Denote $\widetilde{\Phi}_{ i}$
the solution operator for \eqref{slow_K2}.
Let $\A_i^{\IN}$ and $\Ahat_i^{\OUT}$
be the sets defined by \eqref{def_Ain1} and \eqref{def_Ahat}.
We define \begin{equation*}
  \Atilde_{\epsilon i}^{\IN}
  =\kappa_{\epsilon i}^{(21)}\big(\A_i^{\IN}\big),\quad
  \Atilde_{\epsilon i}^{\OUT}
  =\kappa_{\epsilon i}^{(23)}\big(\Ahat_i^{\OUT}\big)
  \quad\text{for}\;
  \epsilon\ge 0.
\end{equation*}
Note that \begin{equation*}
  \Pi_{0\,\Atilde_{0i}}^{\Atilde_{0i}^{\OUT}}
  =\widetilde{\Phi}_{ i}(\cdot,\delta_3).
\end{equation*}

\begin{proposition}
There exists $\Delta>0$ such that
the following assertions hold.
Let $\Atilde_i^{\IN}$ and $\Atilde_i^{\OUT}$
be defined as in the preceding paragraphs
with $\delta_j<\Delta$, $j=1,2,3$,
then for all sufficiently small $\delta_4>0$,
the transition map 
$\Pi_{\epsilon\,\Atilde_{\epsilon i}^{\IN}}^{\Atilde_{\epsilon i}^{\OUT}}$
from $\Atilde_{\epsilon i}^{\IN}$ to $\Atilde_{\epsilon i}^{\OUT}$
following the flow of \eqref{deq_K2} is well-defined
for all small $\epsilon>0$.
Moreover,
\begin{equation}\label{trans_K2}
  \left\|
    \Pi_{\epsilon\,\Atilde_{\epsilon i}^{\IN}}^{\Atilde_{\epsilon i}^{\OUT}}
    \circ
    \kappa_{\epsilon i}^{(21)}
    -
    \Pi_{0\,\Atilde_{0i}}^{\Atilde_{0i}^{\OUT}}
    \circ 
    \Pi_{0\,\Atilde_{0 i}^{\IN}}^{\Atilde_{0i}}
    \circ \kappa_{0i}^{(21)}
  \right\|_{C^1(\A_i^{\IN})}
  =O(\epsilon).
\end{equation}
\label{prop_K2}
\end{proposition}

\begin{proof}
Note that we have
$z^{(J_{i-1})}=z_{i-1}^{(J_{i-1})}+q$
when converting \eqref{deq_pz} to \eqref{deq_K2}.
Let $s=q-z_{i}^{(J-1)}+z_{i-1}^{(J_{i-1})}$.
By Assumption~\ref{hyp_cmin}, \eqref{deq_K2} can be written as
\begin{equation}\notag\begin{aligned}
  &\dot{p}= \epsilon f(p,z_i,0)
  + h(p,z_i+s\,{\sf e}_{J_{i-1}},0)
  + O(|(\epsilon,s)|^2),
  \\
  &\dot{s}= g(p,z_{i}+s\,{\sf e}_{J_{i-1}},0)+ O(|(\epsilon,s)|^2),
  \\
  &\dot{\zeta}^{(j)}=
    -\epsilon\, 
    \frac{\partial g^{(j)}}{\partial z^{(j)}}(p,z_i,0)
    + O(|(\epsilon,s)|^2)
\end{aligned}\end{equation} as $(\epsilon,s)\to 0$.
Since $\frac{\partial g^{(J_i)}}{\partial z^{(J_i)}}(A_i,z_i,0)<0$
by Assumption~\ref{hyp_int},
applying Lemma~\ref{lem_EL}
with $b=z$ and $a=(p,s)$
we obtain \eqref{trans_K2}.
\end{proof}

We denote
$\Pi_{0\Gamma_i}^{\A_i}$
the transition map from $\Gamma_i$
to $\A_i\times \{z_i^{(J_i)}\}\times \{\widehat{\zeta}_i\}$
along the flow of \eqref{fast_K1}
and $\Pi_{0\,\Ahat_i}^{\widehat{\Sigma}_i}$
the transition map from ${0\,\Ahat_i}$ to ${\widehat{\Sigma}_i}$
along the flow of \eqref{slow_K3}.

\begin{proposition}
There exist $\delta_j>0$, $0\le j\le 4$, such that
if $\Gamma_i$, $\A_i$, $\Sigma_i$
are defined in the preceding paragraphs,
then the transition map $\Pi_{\epsilon\Gamma_i}^{{\Sigma}_i}$
from $\Gamma_i$ to ${\Sigma}_i$
following the flow of \eqref{deq_pz} is well-defined
for all small $\epsilon>0$,
and
\begin{equation}\label{trans_A}
  \left\|
    \kappa_{\epsilon i}^{(31)}
    \circ
    \Pi_{\epsilon\Gamma_i}^{\Sigma_i}
    -
    \Pi_{0\,\Ahat_i}^{\widehat{\Sigma}_i}
    \circ \kappa_{0 i}^{(31)}
    \circ \Pi_{0\Gamma_i}^{\A_i}
  \right\|_{C^1(\Gamma_i)}
  =O(\epsilon)
\end{equation}
as $\epsilon\to 0$.
\label{prop_trans_A}
\end{proposition}

\begin{proof}
Analogous to the proof of Proposition~\ref{prop_trans_Am1},
the assertions
can be derived from
Propositions~\ref{prop_K1},~\ref{prop_K3} and~\ref{prop_K2}.
We skip it here.
\end{proof}

\begin{proof}[Proof of Theorem~\ref{thm_main3}]
By a reversal of the time variable,
applying Proposition \ref{prop_trans_A} we have
\begin{equation*}
  \left\|
    \kappa_{\epsilon i}^{(31)}
    \circ
    \Pi_{\epsilon \Gamma_{i+1}}^{\Sigma_i}
    -
    \Pi_{0\Bhat_i}^{\widehat{\Sigma}_i}
    \circ 
    \kappa_{0i}^{(31)}
    \circ \Pi_{0\Gamma_{i+1}}^{\B_i}
  \right\|_{C^1(\Gamma_{i+1})}
  =O(\epsilon).
\end{equation*}
Taking the inverse of the mappings we obtain \begin{equation}\label{trans_B}
  \left\|
    \Pi_{\epsilon\Sigma_i}^{\Gamma_{i+1}}
    \circ
    \kappa_{\epsilon i}^{(13)}
    -
    \Pi_{0\B_i}^{\Gamma_{i+1}}
    \circ 
    \kappa_{0i}^{(13)}
    \circ
    \Pi_{0{\widehat{\Sigma}_i}}^{\Bhat_i}
  \right\|_{C^1(\widehat{\Sigma}_i)}
  =O(\epsilon).
\end{equation}
By \eqref{trans_A} and \eqref{trans_B}, \begin{equation*}\begin{aligned}
  \Pi_{\epsilon\Gamma_i}^{\Gamma_{i+1}}
  &=\left(\Pi_{\epsilon\Sigma_i}^{\Gamma_{i+1}}\circ \kappa_{i\epsilon}^{(13)}\right)
  \circ
  \left(\kappa_{i\epsilon}^{(31)}\circ \Pi_{\epsilon\Gamma_i}^{\Sigma_i}\right)
  \\
  &=\left(
    \Pi_{0\B_i}^{\Gamma_{i+1}}
    \circ 
    \kappa_{0i}^{(13)}
    \circ
    \Pi_{0{\widehat{\Sigma}_i}}^{\Bhat_i}
  \right)
  \circ
  \left(
    \Pi_{0\,\Ahat_i}^{\widehat{\Sigma}_i}
    \circ \kappa_{0 i}^{(31)}
    \circ \Pi_{0\Gamma_i}^{\A_i}
  \right)
  +O(\epsilon)
  \\
  &= 
  \Pi_{0\B_i}^{\Gamma_{i+1}}
  \circ 
  \kappa_{0i}^{(13)}
  \circ
  \Pi_{0\,\Ahat_i}^{\Bhat_i}
  \circ \kappa_{0 i}^{(31)}
  \circ \Pi_{0\Gamma_i}^{\A_i}
  +O(\epsilon)
  \\
  &=
  \Pi_{0\B_i}^{\Gamma_{i+1}}
  \circ 
  \widehat{Q}_i
  \circ \Pi_{0\Gamma_i}^{\A_i}
  +O(\epsilon)
\end{aligned}\end{equation*}
Therefore, \begin{equation}\label{Pi_Gamma_ii}\begin{aligned}
  &\Pi_{\epsilon\Gamma_{i+1}}^{\Gamma_{i+2}}
  \circ
  \Pi_{\epsilon\Gamma_i}^{\Gamma_{i+1}}
  \\
  &\quad
  = 
  \left(
    \Pi_{0\B_{i+1}}^{\Gamma_{i+2}}
    \circ 
    \widehat{Q}_{i+1}
    \circ \Pi_{0\Gamma_{i+1}}^{\A_{i+1}}
  \right)
  \circ
  \left(
   \Pi_{0\B_i}^{\Gamma_{i+1}}
    \circ 
    \widehat{Q}_i
    \circ \Pi_{0\Gamma_i}^{\A_i}
  \right)
  +O(\epsilon).
\end{aligned}\end{equation}
We denote \begin{equation*}
  P_\epsilon
  = \Pi_{\epsilon\Gamma_{N}}^{\Gamma_1}
  \circ
  \cdots
  \circ
  \Pi_{\epsilon\Gamma_2}^{\Gamma_3}
  \circ
  \Pi_{\epsilon\Gamma_1}^{\Gamma_2}.
\end{equation*}
By \eqref{Pi_Gamma_ii}
and the relation that $\Pi_{0\B_{i-1}}^{\A_{i}}=\pi_{i}\times \mathrm{id}$,
we have \begin{equation*}\begin{aligned}
  P_\epsilon
  &=
  \Pi_{0\B_N}^{\Gamma_1}
  \circ
  \widehat{Q}_N
  \circ (\pi_N\times \mathrm{id})
  \circ
  \cdots \circ \widehat{Q}_2
  \circ (\pi_2\times \mathrm{id})
  \circ \widehat{Q}_1
  \circ
  \Pi_{0\Gamma_1}^{\A_1}
  + O(\epsilon).
\end{aligned}\end{equation*} 
Writing $\Pi_{0\Gamma_1}^{\A_1}
=\Pi_{0\B_N}^{\A_1}\circ \Pi_{0\Gamma_1}^{\B_N}
=(\pi_1\times \mathrm{id})\circ \big(\Pi_{\B_N}^{\Gamma_1}\big)^{-1}$,
it follows that 
\begin{equation*}\begin{aligned}
  P_\epsilon
  &=
  \Pi_{0\B_N}^{\Gamma_1}
  \circ
  \widetilde{P}
  \circ
  \left(\Pi_{0\B_N}^{\Gamma_1}\right)^{-1}
  +O(\epsilon),
\end{aligned}\end{equation*}
where $\widetilde{P}$ is defined by \eqref{def_Ptilde}.
This implies that \begin{equation*}\begin{aligned}
  &\det\left(
  DP_\epsilon
  -\mathrm{id}
  \right)
  \\
  &\quad=
  \det\left(
    D\Pi_{0\B_N}^{\Gamma_1}
    \circ
    DP
    \circ
    \left(D\Pi_{0\B_N}^{\Gamma_1}\right)^{-1}
    -\mathrm{id}
  \right)
  + O(\epsilon)
  \\
  &\quad=
  \det\left(
    DP
    -\mathrm{id}
  \right)
  + O(\epsilon).
\end{aligned}\end{equation*}
Hence, the linearization of
the return map $P_\epsilon$
at $(p_{01},q_{01},\widehat{\zeta}_{1})\in \Gamma_1$
does not have a singular value equal to $1$
for all small $\epsilon>0$
if $\det(DP-\mathrm{id})\ne 0$.
Consequently,
for all small $\epsilon>0$
there exists a locally unique 
fixed point $(p_{\epsilon 1},q_{\epsilon1},\widehat{\zeta}_{\epsilon 1})\in \Gamma_i$
of $P_\epsilon$.
Let $(p_{\epsilon 1},z_{\epsilon 1})
=\kappa_{\epsilon 1}^{(01)}
(p_{\epsilon 1},q_{\epsilon1},\widehat{\zeta}_{\epsilon 1})$.
Then the trajectory passing through $(p_{\epsilon 1},z_{\epsilon 1})$
is a periodic orbit of system \eqref{deq_pz}.
If the spectrum radius of $DP(p_{01},q_{01},\widehat{\zeta}_1)$
is smaller (resp.\ greater) than $1$,
then $P_\epsilon$ is a contraction (resp.\ expansion),
hence the periodic orbit is orbitally asymptotically stable (resp.\ unstable).
\end{proof}

\section{Examples}
\label{sec_ex}
In this section we apply the main results
to study
the examples \eqref{deq_CoEvol}, \eqref{deq_tradeoff}, \eqref{deq_switching}
and the planar system \eqref{deq_ab}
mentioned in Section~\ref{sec_intro}.

\subsection{Trade-off between Encounter and Growth Rates}
\label{sec_tradeoff}
Consider system \eqref{deq_tradeoff},
which takes the form
\begin{equation}\notag
\begin{aligned}
  &x'=F(x,\alpha)-G(x,y,\alpha),
  \\
  &y'=H(x,y,\alpha)- D(y),
  \\
  &\epsilon \alpha'
  =\alpha(1-\alpha)E(x,y,\alpha),
\end{aligned}\end{equation}
with \begin{equation*}\begin{aligned}
  &F(x,\alpha)=x(\alpha+r-kx),
  \\
  &G(x,y,\alpha)
  =H(x,y,\alpha)
  =\frac{xy(a\alpha^2+b\alpha+c)}{1+x},
  \\
  &D(y,\beta)=dy,
\end{aligned}\end{equation*}
and \begin{equation*}
  E(x,y,\alpha)=\frac{\partial}{\partial \alpha}\left(\frac{x'}{x}\right)
  = 1-\frac{y(2a\alpha+b)}{1+x}.
\end{equation*}
The limiting fast system is \begin{equation}\notag
  \dot{x}=0,
  \quad
  \dot{y}=0,
  \quad
  \dot{\alpha}=\alpha(1-\alpha)E(x,y,\alpha).
\end{equation}
The critical manifolds are \begin{equation*}
  {M}_1=\{(x,y,\alpha): \alpha=0\}
  \quad\text{and}\quad
  {M}_2=\{(x,y,\alpha): \alpha=1\}.
\end{equation*}
On the critical manifolds ${M}_i$,
the limiting slow system is \begin{equation}\label{slow_tradeoff}\begin{aligned}
  &x'=F(x,\bar\alpha)-G(x,y,\bar\alpha),
  \\
  &y'=H(x,y,\bar\alpha)- D(y),
\end{aligned}\end{equation}
where $\bar\alpha=0,1$.
Let $\Phi_1$ and $\Phi_2$
be the solution operators for \eqref{slow_tradeoff}
with $\alpha=0$ and $\alpha=1$, respectively.
The transition maps $Q_1$ and $Q_2$
in Theorem~\ref{thm_main1}
are determined by \[\begin{aligned}
  Q_1(A_1)
  =
  \Phi_1(A_1,\tau_1)
  \;\;\text{with}\;\;
  \int_0^{\tau_1} 
  \left.\left(1-\frac{by}{1+x}\right)\right|_{(x,y)
  =\Phi_1(A_1,\tau)}\;d\tau
  =0
\end{aligned}\]
and \[\begin{aligned}
  Q_2(A_2)
  =
  \Phi_2(A_2,\tau_2)
  \;\;\text{with}\;\;
  \int_0^{\tau_2} 
  \left.\left(1-\frac{y(2a+b)}{1+x}\right)\right|_{(x,y)=\Phi_2(A_2,\tau)}\;d\tau
  =0.
\end{aligned}\]

Following \cite{Cortez:2010}, we set
$a = -0.1$,
$b = 3$,
$c = 1$,
$d = 2.8$,
$k = 1$,
and $r = 10$.
By implementing Newton's iteration,
we find points $A_1=B_2\approx (5.57,11.03)$
and $B_1=A_2\approx (9.96,0.36)$ satisfying \[
  A_2=B_1=Q_1(A_1)
  \quad\text{and}\quad
  A_1=B_2=Q_2(A_2).
\]
This means that $A_i$ and $B_i$
satisfy the following conditions (see Figure~\ref{fig_tradeoff}(b)):
\begin{itemize}
\item[(i)] $A_1$ and $B_1$
are connected by a trajectory $\sigma_1$
of \eqref{slow_tradeoff}
with $\bar{\alpha}=0$;
\item[(ii)] $A_2$ and $B_2$
are connected by a trajectory $\sigma_2$
of \eqref{slow_tradeoff}
with $\bar{\alpha}=1$;
\item[(iii)]
$\int_{\sigma_1}E(x,y,0)\;d\tau=0$
and
$\int_{\sigma_2}E(x,y,1)\;d\tau=0$.
\end{itemize}
Using the formulas
in Proposition~\ref{prop_DQi}
and Remark~\ref{rmk_computing},
we obtain
\begin{equation*}
  DQ_1(A_1)
  \approx \begin{pmatrix}
    -0.0001& -0.0029
    \\
    0.0009& 0.0258
  \end{pmatrix}
  \quad\text{and}\quad
  DQ_2(A_2)
  \approx \begin{pmatrix}
    0.02& 18.91
    \\
    -0.02& -16.95
  \end{pmatrix}.
\end{equation*}
Hence,
the eigenvalues of 
$DP(A_1)=DQ_2(A_2)\,DQ_1(A_1)$
are $\lambda_1\approx 2.86\cdot 10^{-14}$
and $\lambda_2=-0.42$,
which are both of magnitude less than one.
Therefore,
by Theorem~\ref{thm_main1},
the configuration \begin{equation*}
  \gamma_1\cup \sigma_1\cup \gamma_2\cup \sigma_2
\end{equation*}
corresponds to a relaxation oscillation
formed by
orbitally locally asymptotically stable periodic orbits.

For system \eqref{deq_tradeoff} with $\epsilon=0.1$,
taking initial data
$(x,y,\alpha) = (10,0.5,0.5)$
we find that the trajectory
converges to a periodic orbit
(see Figure~\ref{fig_tradeoff}(a))
near the singular configuration.

\subsection{Prey Switching}
\label{sec_switching}
Assuming that
the response functions $f_i(p_i)$
in \eqref{deq_switching} are linear,
after rescaling, the system is converted to
\begin{equation}\label{deq_switching1}\begin{aligned}
  &p_1'= (1-qz)p_1,
  \\
  &p_2'= (r-(1-q)z)p_2,
  \\
  &z'= \big(qp_1+(1-q)p_2-1\big)z,
  \\
  &\epsilon q'
  = q(1-q)(p_1-p_2).
\end{aligned}\end{equation}

The critical manifolds for \eqref{deq_switching1} are \begin{equation*}
  {M}_1=\{(p_1,p_2,z,q): q=0\}
  \quad\text{and}\quad
  {M}_2=\{(p_1,p_2,z,q): q=1\}.
\end{equation*}
On ${M}_1$, the restriction of \eqref{deq_switching} is
\begin{equation}\label{slow_Switching_M1}
\begin{aligned}
  &p_1'= p_1,
  \\
  &p_2'= (r-z)p_2,
  \\
  &z'= (p_2-1)z.
\end{aligned}\end{equation}
which means that the predators
hunt exclusively only the first prey population.
On ${M}_2$, the restriction of \eqref{deq_switching} is
\begin{equation}\label{slow_Switching_M2}
\begin{aligned}
  &p_1'= (1-z)p_1,
  \\
  &p_2'= rzp_2,
  \\
  &z'= (p_1-1)z,
\end{aligned}\end{equation}
which means that the predators
hunt exclusively only the second prey population.

Let $\Phi_1$ and $\Phi_2$
be the transition maps for
\eqref{slow_Switching_M1} and \eqref{slow_Switching_M2}, respectively.
The transition maps $Q_1$ and $Q_2$
in Theorem~\ref{thm_main1}
are determined by \[\begin{aligned}
  Q_1(A_1)
  =
  \Phi_1(A_1,\tau_1)
  \;\;\text{with}\;\;
  \int_0^{\tau_1} 
  \big(p_1-p_2\big)\Big|_{(p_1,p_2,z)=\Phi_1(A_1,\tau)}\;d\tau
  =0
\end{aligned}\]
and \[\begin{aligned}
  Q_2(A_2)
  =
  \Phi_2(A_2,\tau_2)
  \;\;\text{with}\;\;
  \int_0^{\tau_2} 
  \big(p_1-p_2\big)\Big|_{(p_1,p_2,z)=\Phi_2(A_2,\tau)}\;d\tau
  =0.
\end{aligned}\]

With the parameters given in \cite{Piltz:2017},
$r = 0.5$ and $m = 0.4$,
we find
$A_1=B_2\approx (0.92,1.08,1.50)$
and
$A_2=B_1\approx (1.08,0.92,1.50)$
such that the transition maps $Q_i$
in Theorem~\ref{thm_main1}
satisfy
$Q_1(A_1)=B_1$ and $Q_2(A_2)=B_2$
(see Figure~\ref{fig_switching}(b)).
Using the formulas
in Proposition~\ref{prop_DQi}
and Remark~\ref{rmk_computing},
we obtain \begin{equation*}
  DQ_1(A_1)
  \approx \begin{pmatrix}
    -6.78& 5.74& -1.00
    \\
    6.77& -4.03& 0.70
    \\
    0.34& -0.16& 1.04
  \end{pmatrix},\;\;
  DQ_2(A_2)
  \approx \begin{pmatrix}
    -1.56& 3.38& 0.55
    \\
    2.80& -2.80& -0.99
    \\
    -0.07& 0.34& 1.06
  \end{pmatrix}.
\end{equation*}
Hence, the eigenvalues of $DP(A_1)=DQ_2(A_2)\,DQ_1(A_1)$
are $\lambda_1\approx 60.55$ and $\lambda_{2,3}\approx 0.97\pm 0.26\sqrt{-1}$.
Since $\lambda_1$ is greater than $1$,
by Theorem~\ref{thm_main1},
the configuration 
connecting $A_i$ and $B_i$
corresponds to a relaxation oscillation
formed by orbitally unstable periodic orbits
(see Figure~\ref{fig_switching}(b)).

\subsection{Coevolution}
\label{sec_coevolution}
System \eqref{deq_CoEvol} 
has critical manifolds ${M}_i$, $1\le i\le 4$,
corresponding to
$(\alpha,\beta)=(\alpha_i,\beta_i)$
with $(\alpha_i,\beta_i)$, $i=1,2,3,4$,
equal to $(0,0)$, $(0,1)$, $(1,1)$ and $(1,0)$, respectively.
The limiting slow system on each ${M}_i$ is
\begin{equation}\label{slow_CoEvol}
\begin{aligned}
  &\frac{d}{d\tau}x=F(x,\alpha_i)-G(x,y,\alpha_i,\beta_i),
  \\
  &\frac{d}{d\tau}y=H(x,y,\alpha_i,\beta_i)- D(y,\beta_i).
\end{aligned}\end{equation}
The numbers $\omega_i=(\omega_i^{(1)},\omega_i^{(2)})$
defined by \eqref{def_omega}
are $\omega_1=(1,1)$,
$\omega_2=(1,-1)$,
$\omega_3=(-1,-1)$,
and $\omega_4=(-1,1)$.
Equations for $\zeta=(\zeta^{(1)},\zeta^{(2)})$ in \eqref{slow_pzeta}
on $M_i$ are
\begin{equation}\label{slow_CoEvol_zeta}\begin{aligned}
  &\frac{d}{d\tau}\zeta_i^{(1)}=
  \omega_i^{(1)}
  E_1(x,y,\alpha_i,\beta_i),
  \\
  &\frac{d}{d\tau}\zeta_i^{(2)}=
  \omega_i^{(2)}
  E_2(x,y,\alpha_i,\beta_i),
\end{aligned}\end{equation}
where \[
  E_1(x,y,\alpha,\beta)
  =
  \frac{\partial}{\partial \alpha}
  \left(
    \frac{F(x,\alpha)-G(x,y,\alpha,\beta)}{x}
  \right)
\]
and \[
  E_2(x,y,\alpha,\beta)
  =
  \frac{\partial}{\partial \beta}
  \left(
    \frac{H(x,\alpha)-D(x,y,\alpha,\beta)}{y}
  \right).
\]
Let $\widehat{\Phi}_i$,  $1\le i\le 4$,
be the solution operators for
system \eqref{slow_CoEvol}-\eqref{slow_CoEvol_zeta}.
Then the transition maps $\widehat{Q}_i$
in Theorem~\ref{thm_main2}
are determined by \begin{align*}
  &\widehat{Q}_1(A_1,\zeta)
  =
  \widehat{\Phi}_1((A_1,\zeta),\tau_1)
  \;\;\text{with}\;\;
  \zeta^{(2)}
  +
  \int_0^{\tau_1} 
  E_2(x,y,0,0)\Big|_{(x,y)=\Phi_1(A_1,\tau)}\;d\tau
  =0,
  \\
  &\widehat{Q}_2(A_2,\zeta)
  =
  \widehat{\Phi}_2((A_2,\zeta),\tau_2)
  \;\;\text{with}\;\;
  \zeta^{(1)}
  +\int_0^{\tau_2} 
  E_1(x,y,0,1)\Big|_{(x,y)=\Phi_2(A_2,\tau)}\;d\tau
  =0,
  \\
  &\widehat{Q}_3(A_3,\zeta)
  =
  \widehat{\Phi}_3((A_3,\zeta),\tau_3)
  \;\;\text{with}\;\;
  \zeta^{(2)}
  - \int_0^{\tau_1} 
  E_2(x,y,1,1)\Big|_{(x,y)=\Phi_1(A_1,\tau)}\;d\tau
  =0,
  \\
  &\widehat{Q}_4(A_4,\zeta)
  =
  \widehat{\Phi}_4((A_4,\zeta),\tau_4)
  \;\;\text{with}\;\;
  \zeta^{(1)}
  -\int_0^{\tau_4} 
  E_1(x,y,1,0)\Big|_{(x,y)=\Phi_4(A_4,\tau)}\;d\tau
  =0.
\end{align*}
Following Cortez and Weitz \cite[Supporting Information D]{Cortez:2014-PNAS},
we consider \eqref{deq_CoEvol}
with \begin{equation*}\begin{aligned}
  &F(x,\alpha)=
  x(s_0+s_1\alpha)\left(1-\frac{x}{k_0+k_1\alpha}\right),
  \\
  &G(x,y,\alpha,\beta)
  = \frac{(r_0+r_1\alpha+r_2\beta
  +r_3\alpha\beta+r_4\beta^2)xy}{1+hx},
  \\
  &H(x,y,\alpha,\beta)
  =c_0G(x,y,\alpha,\beta),
  \\
  &D(y,\beta)=y^{1.5}(\delta_0+\delta_1\beta),
\end{aligned}\end{equation*}
and parameters
$s_0=2.5, s_1=3.5, k_0=1, k_1=0.1, 
r_0=0.65, r_1=3, r_2=2.3,
r_3=-0.2, r_4=0.01, 
c_0=1.7, 
\delta_0=0.76, \delta_1=1.77$ and $h=1$.
Implementing Newton's iteration
for $\widehat{Q}_i(A_i,\zeta_i)=(A_{i+1},\zeta_{i+1})$, $1\le i\le 4$,
we find $B_4=A_1\approx (0.33,1.99)$,
$B_1=A_2\approx (0.92,0.56)$,
$B_2=A_3\approx (0.60,0.55)$
and $B_3=A_4\approx (0.30,0.93)$
(see Figure~\ref{fig_CoEvol_sing}(b)),
and $\zeta_1\approx(0,0.98)$,
$\zeta_2\approx(3.84,0)$,
$\zeta_3\approx(0,1.12)$
and $\zeta_4\approx(0.55,0)$.

Let $\{{\sf e}_x,{\sf e}_y,{\sf e}_\alpha,{\sf e}_\beta\}$
be the standard ordered basis of the $(x,y,\alpha,\beta)$-space.
Note that
the tangent space of $\A_1\times \Lambda_1$ at $(A_1,\zeta_1)$
is spanned by $\{{\sf e}_x,{\sf e}_y,{\sf e}_\beta\}$,
and the tangent space of $\B_1\times \Lambda_2$ at $(B_1,\zeta_2)$
is spanned by $\{{\sf e}_x,{\sf e}_y,{\sf e}_\alpha\}$.
Using formulas in Proposition~\ref{prop_DQiHat},
we obtain
\begin{equation*}
  D\widehat{Q}_1(A_1,\zeta_1)
  \approx\begin{blockarray}{*{3}{c} l}
    \begin{block}{*{3}{>{$}c<{$}} l}
      ${\sf e}_{x}$ & ${\sf e}_{x}$ & ${\sf e}_{\beta}$ & \\
    \end{block}
    \begin{block}{(*{3}{c})>{$}l<{$}}
      0.013 & 0.004 & -0.007 & ${\sf e}_{x}$ \\
      0.080 & -0.254 & 0.038 & ${\sf e}_{y}$ \\
      -3.29 & -2.42 & 0.67 & ${\sf e}_{\alpha}$ \\
    \end{block}
  \end{blockarray}.
\end{equation*} 
Similarly, \begin{equation*}
  D\widehat{Q}_2(A_2,\zeta_2)
  \approx\begin{blockarray}{*{3}{c} l}
    \begin{block}{*{3}{>{$}c<{$}} l}
      ${\sf e}_{x}$ & ${\sf e}_{x}$ & ${\sf e}_{\alpha}$ & \\
    \end{block}
    \begin{block}{(*{3}{c})>{$}l<{$}}
      -0.00040 & -0.0058 & 0.00024 & ${\sf e}_{x}$ \\
      -0.00003 & 0.00024 & 0.00030 & ${\sf e}_{y}$ \\
      0.37 & -1.44 & -0.26 & ${\sf e}_{\beta}$ \\
    \end{block}
  \end{blockarray},\quad
\end{equation*}
and the approximations of
$D\widehat{Q}_3(A_3,\zeta_3)$
and $D\widehat{Q}_4(A_4,\zeta_4)$ are, respectively, \begin{equation*}
  \begin{blockarray}{*{3}{c} l}
    \begin{block}{*{3}{>{$}c<{$}} l}
      ${\sf e}_{x}$ & ${\sf e}_{x}$ & ${\sf e}_{\beta}$ & \\
    \end{block}
    \begin{block}{(*{3}{c})>{$}l<{$}}
      0.29 & -0.04 & -0.22 & ${\sf e}_{x}$ \\
      0.26 & -0.67 & 0.49 & ${\sf e}_{y}$ \\
      2.49 & 0.13 & -0.86 & ${\sf e}_{\alpha}$ \\
    \end{block}
  \end{blockarray}
  \quad\text{and}\quad
  \begin{blockarray}{*{3}{c} l}
    \begin{block}{*{3}{>{$}c<{$}} l}
      ${\sf e}_{x}$ & ${\sf e}_{x}$ & ${\sf e}_{\alpha}$ & \\
    \end{block}
    \begin{block}{(*{3}{c})>{$}l<{$}}
      -0.10 & -0.09 & 0.03 & ${\sf e}_{x}$ \\
      0.42 & 0.38 & -0.13 & ${\sf e}_{y}$ \\
      -0.36 & -0.33 & 0.11 & ${\sf e}_{\beta}$ \\
    \end{block}
  \end{blockarray}.
\end{equation*}
Hence, the eigenvalues of
\begin{equation*}
  D\widehat{P}(A_1,\zeta_1)
  =D\widehat{Q}_4(A_4,\zeta_4)\; D\widehat{Q}_3(A_3,\zeta_3)\; 
  D\widehat{Q}_2(A_2,\zeta_2)\; D\widehat{Q}_1(A_1,\zeta_1)
\end{equation*}
are $\lambda_1\approx 0.39$, $\lambda_2\approx-6.14\cdot 10^{-5}$
and $\lambda_3\approx-5.11\cdot10^{-11}$,
which are all of magnitude less than one.
Therefore,
by Theorem~\ref{thm_main2},
this singular configuration
corresponds to a relaxation oscillation
formed by
orbitally locally asymptotically stable periodic orbits.

\subsection{A Planar System}
\label{sec_planar}
The limiting fast system of \eqref{deq_ab} is
\begin{equation}\label{fast_ab}\begin{aligned}
  &\frac{d}{dt}a= b\,H(a,b,0),
  \quad
  &\frac{d}{dt}b= b\,G(a,b,0).
\end{aligned}\end{equation}
On the critical manifold ${M}=\{(a,b): b=0\}$,
the limiting slow system is \begin{equation}\notag
  a'=F(a,0,0).
\end{equation}
We assume that (see Figure~\ref{fig_bifdelay})
\begin{figure}[htbp] 
\begin{center}
{\includegraphics[trim = 2.7cm .9cm 1cm .7cm, clip, width=.55\textwidth]%
{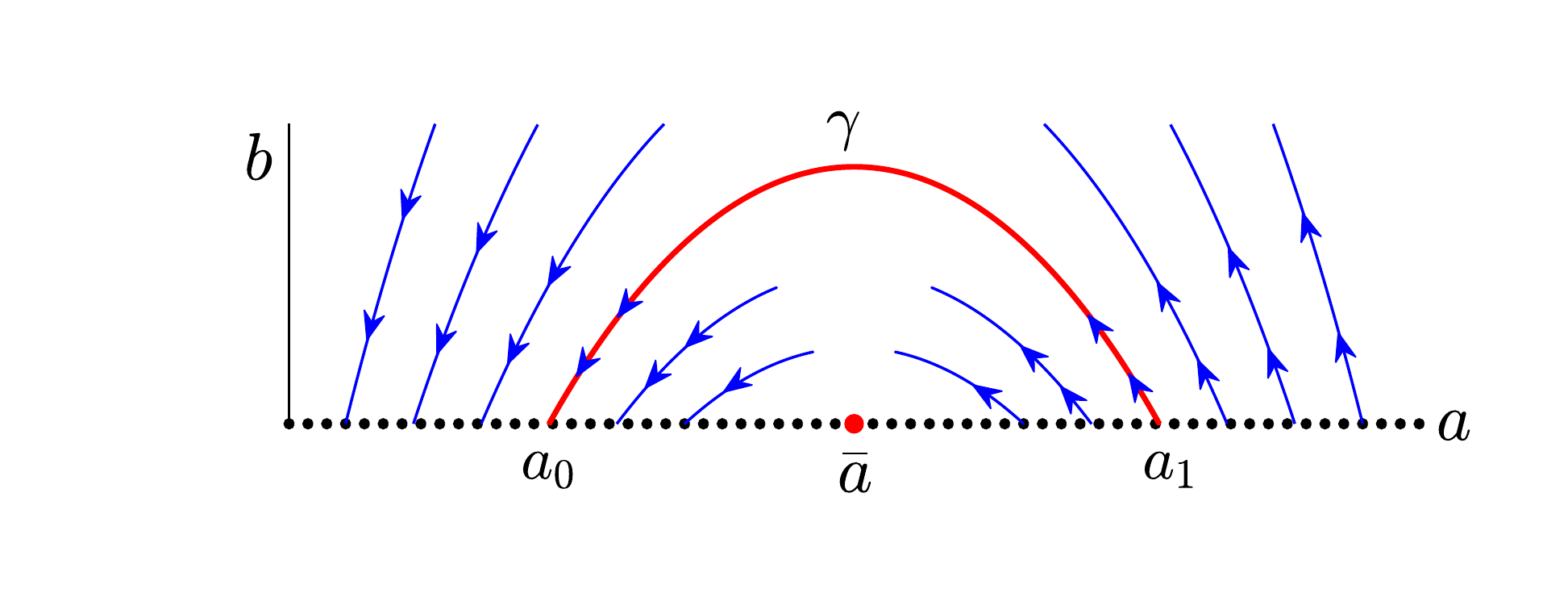}}
\end{center}
\caption{
For system \eqref{fast_ab} with $\epsilon=0$,
the $a$-axis is a line of equilibria
and
$\gamma$ is a heteroclinic orbit
connecting $(a_0,0)$ and $(a_1,0)$.
}
\label{fig_bifdelay}
\end{figure} 
\begin{enumerate}
\item[(i)]
There is a trajectory $\gamma$ of \eqref{fast_ab} satisfying
\begin{equation*}
  \lim_{t\to -\infty}\gamma(t)= (a_0,0),
  \quad
  \lim_{t\to \infty}\gamma(t)= (a_1,0);
\end{equation*}
\item[(ii)]
$F(a,0,0)>0$ for all $a\in [a_0,a_1]$;
\vspace{.5em}
\item[(iii)]
$G(a_0,0,0)<0$ and $G(a_1,0,0)>0$;
\vspace{.5em}
\item[(iv)]
$\displaystyle\int_{a_0}^{a_1}\frac{G(a,0,0)}{F(a,0,0)}\;da=0$
\;and\;
$\displaystyle\int_{a_0}^{s}\frac{G(a,0,0)}{F(a,0,0)}\;da<0$
\;$\forall s\in (a_0,a_1)$.
\end{enumerate}
We provide an alternative proof
of the following theorem from
Hsu and Wolkowicz~\cite{Hsu:2019criterion}.

\begin{theorem}
Consider system \eqref{deq_ab}.
Assume $\mathrm{(i)}$--$\mathrm{(iv)}$ and let \begin{equation*}
  \lambda=\ln \left|\frac{F(a_1,0,0)}{F(a_0,0,0)}\right|
  + \int_{\gamma}\frac{\partial_aH}{H}\;da
  + \int_{\gamma}\frac{\partial_bG}{G}\;db.
\end{equation*}
If $\lambda\ne 0$,
then $\gamma$ admits a relaxation oscillation
which is formed by
locally unique periodic orbits for small $\epsilon>0$.
Moreover,
the periodic orbit
is orbitally asymptotically stable if $\lambda< 0$
and unstable if $\lambda> 0$.
\end{theorem}

\begin{remark}
Assumptions (i) and (iv)
are weaker than the conditions assumed in \cite{Hsu:2019criterion}.
In that paper, the assumption corresponding to (i)
is that there exists a smooth family of heteroclinic orbits;
the assumption corresponding to the inequalities in (iii) and (iv)
is that $G(a,0,0)<0$ for $a<\bar{a}$ and $G(a,0,0)>0$ for $a>\bar{a}$.
However, the analysis
in that paper is also valid under these weaker assumptions.
\end{remark}

\begin{proof}
Define a function $Q$
implicitly by $Q(a_0)=a_1$ and \begin{equation}\label{def_Qab}
  \int_{a_0}^{Q(a)}
  \frac{G(r,0,0)}{F(r,0,0)}\;dr=0.
\end{equation}
By \eqref{DQi_fBi} in 
Proposition~\ref{prop_DQi}, \begin{equation}\label{DQ_ab}
  \frac{dQ(a_0)}{da}
  = \frac{F(a_1,0,0)}{G(a_1,0,0)}
  \;\frac{G(a_0,0,0)}{F(a_0,0,0)}.
\end{equation}
(Note that equation \eqref{DQ_ab}
can also be derived directly by differentiating \eqref{def_Qab}.)

Let $\pi$ be the transition map of \eqref{fast_ab}
from a neighborhood of $(a_1,0)$
to a neighborhood $(a_0,0)$
in the $a$-axis.
By \eqref{Dpi1} in Proposition~\ref{prop_Dpi},
\begin{equation}\label{Dpi_ab}
  \frac{d}{da}\pi(a_1)
  = \frac{G(a_1,0,0)}{G(a_0,0,0)}
  \exp\left(\int_{\gamma}\partial_aH+\partial_bG\;dt\right).
\end{equation}
By \eqref{Dpi_ab} and \eqref{DQ_ab},
we obtain \begin{equation*}\begin{aligned}
  \frac{d}{da}(\pi\circ Q)
  &=\frac{d\pi(a_1)}{da}\,\frac{dQ(a_0)}{da}
  \\
  &= \left(
  \frac{F(a_1,0,0)}{G(a_1,0,0)}\;\frac{G(a_0,0,0)}{F(a_0,0,0)}
  \right)
  \frac{G(a_1,0,0)}{G(a_0,0,0)}
  \exp\left(\int_0^T\partial_aH+\partial_bG\;dt\right).
\end{aligned}\end{equation*}
Using the relations $da/dt=H$ and $da/dt=G$ in \eqref{fast_ab}, 
it follows that \begin{equation*}
  \frac{d}{da}(\pi\circ Q)(a_0)
  =\frac{F(a_1,0,0)}{F(a_0,0,0)}
  \exp\left(
    \int_{\gamma}\frac{\partial_aH}{H}\;da
    + \int_{\gamma}\frac{\partial_bG}{H}\;db.
  \right).
\end{equation*}
Hence \begin{equation*}
  \ln \left|\frac{d}{da}(\pi\circ P)(a_0)\right|
  =\ln \left|\frac{F(a_1,0,0)}{F(a_0,0,0)}\right|
  + \int_{\gamma}\frac{\partial_aH}{H}\;da
  + \int_{\gamma}\frac{\partial_bG}{H}\;db.
\end{equation*}
Hence $\lambda<0$ if and only if $\left|\frac{d}{da}(\pi\circ P)(a_0)\right|<1$.
By Theorem~\ref{thm_main3},
the desired result follows.
\end{proof}


\begin{thebibliography}{10}

\bibitem{Abrams:1993}
P.~A. Abrams, H.~Matsuda, and Y.~Harada.
\newblock Evolutionarily unstable fitness maxima and stable fitness minima of
  continuous traits.
\newblock {\em Evol. Ecol.}, 7(5):465--487, Sep 1993.
\newblock \href {http://dx.doi.org/10.1007/BF01237642}
  {\path{doi:10.1007/BF01237642}}.

\bibitem{Benoit:1981}
{\'E}.~Beno{\^{\i}}t.
\newblock \'{E}quations diff{\'e}rentielles: relation entr{\'e}e--sortie.
\newblock {\em C. R. Acad. Sci. Paris S{\'e}r. I Math.}, 293(5):293--296, 1981.

\bibitem{Benoit:1991}
{\'E}.~Beno{\^{\i}}t.
\newblock Linear dynamic bifurcation with noise.
\newblock In {\em Dynamic bifurcations ({L}uminy, 1990)}, Lecture Notes in
  Math. Vol. 1493, pages 131--150. Springer, Berlin, 1991.
\newblock \href {http://dx.doi.org/10.1007/BFb0085028}
  {\path{doi:10.1007/BFb0085028}}.

\bibitem{Cortez:2011}
M.~H. Cortez.
\newblock Comparing the qualitatively different effects rapidly evolving and
  rapidly induced defences have on predator–prey interactions.
\newblock {\em Ecol. Lett.}, 14(2):202--209, 2011.
\newblock \href {http://dx.doi.org/10.1111/j.1461-0248.2010.01572.x}
  {\path{doi:10.1111/j.1461-0248.2010.01572.x}}.

\bibitem{Cortez:2015}
M.~H. Cortez.
\newblock Coevolution-driven predator-prey cycles: predicting the
  characteristics of eco-coevolutionary cycles using fast-slow dynamical
  systems theory.
\newblock {\em Theor. Ecol.}, 8(3):369--382, Aug 2015.
\newblock \href {http://dx.doi.org/10.1007/s12080-015-0256-x}
  {\path{doi:10.1007/s12080-015-0256-x}}.

\bibitem{Cortez:2016AmerNatur}
M.~H. Cortez.
\newblock How the magnitude of prey genetic variation alters predator-prey
  eco-evolutionary dynamics.
\newblock {\em Am. Nat.}, 188(3):329--341, 2016.
\newblock \href {http://dx.doi.org/10.1086/687393} {\path{doi:10.1086/687393}}.

\bibitem{Cortez:2018}
M.~H. Cortez.
\newblock Genetic variation determines which feedbacks drive and alter
  predator--prey eco-evolutionary cycles.
\newblock {\em Ecol. Monogr.}, 88(3):353--371, 2018.
\newblock \href {http://dx.doi.org/10.1002/ecm.1304}
  {\path{doi:10.1002/ecm.1304}}.

\bibitem{Cortez:2010}
M.~H. Cortez and S.~P. Ellner.
\newblock Understanding rapid evolution in predator‐prey interactions using
  the theory of fast‐slow dynamical systems.
\newblock {\em Am. Nat.}, 176(5):E109--E127, 2010.
\newblock \href {http://dx.doi.org/10.1086/656485} {\path{doi:10.1086/656485}}.

\bibitem{Cortez:2017}
M.~H. Cortez and S.~Patel.
\newblock The effects of predator evolution and genetic variation on
  predator--prey population-level dynamics.
\newblock {\em Bull. Math. Biol.}, 79(7):1510--1538, Jun 2017.
\newblock \href {http://dx.doi.org/10.1007/s11538-017-0297-y}
  {\path{doi:10.1007/s11538-017-0297-y}}.

\bibitem{Cortez:2014-PNAS}
M.~H. Cortez and J.~S. Weitz.
\newblock Coevolution can reverse predator{\textendash}prey cycles.
\newblock {\em Proc. Natl. Acad. Sci. U. S. A.}, 111(20):7486--7491, 2014.
\newblock \href {http://dx.doi.org/10.1073/pnas.1317693111}
  {\path{doi:10.1073/pnas.1317693111}}.

\bibitem{DeMaesschalck:2008JDE}
P.~De~Maesschalck.
\newblock Smoothness of transition maps in singular perturbation problems with
  one fast variable.
\newblock {\em J. Differential Equations}, 244(6):1448--1466, 2008.
\newblock \href {http://dx.doi.org/10.1016/j.jde.2007.10.023}
  {\path{doi:10.1016/j.jde.2007.10.023}}.

\bibitem{DeMaesschalck:2016}
P.~De~Maesschalck and S.~Schecter.
\newblock The entry-exit function and geometric singular perturbation theory.
\newblock {\em J. Differential Equations}, 260(8):6697--6715, 2016.
\newblock \href {http://dx.doi.org/10.1016/j.jde.2016.01.008}
  {\path{doi:10.1016/j.jde.2016.01.008}}.

\bibitem{Diener:1984}
M.~Diener.
\newblock The canard unchained or how fast/slow dynamical systems bifurcate.
\newblock {\em Math. Intelligencer}, 6(3):38--49, 1984.
\newblock \href {http://dx.doi.org/10.1007/BF03024127}
  {\path{doi:10.1007/BF03024127}}.

\bibitem{Dumortier:1996}
F.~Dumortier and R.~Roussarie.
\newblock Canard cycles and center manifolds.
\newblock {\em Mem. Amer. Math. Soc.}, 121(577):x+100, 1996.
\newblock With an appendix by Cheng Zhi Li.
\newblock \href {http://dx.doi.org/10.1090/memo/0577}
  {\path{doi:10.1090/memo/0577}}.

\bibitem{Fenichel:1979}
N.~Fenichel.
\newblock Geometric singular perturbation theory for ordinary differential
  equations.
\newblock {\em J. Differential Equations}, 31(1):53--98, 1979.
\newblock \href {http://dx.doi.org/10.1016/0022-0396(79)90152-9}
  {\path{doi:10.1016/0022-0396(79)90152-9}}.

\bibitem{Gasser:2016}
I.~Gasser, P.~Szmolyan, and J.~W\"{a}chtler.
\newblock Existence of {C}hapman-{J}ouguet detonation and deflagration waves.
\newblock {\em SIAM J. Math. Anal.}, 48(2):1400--1422, 2016.
\newblock \href {http://dx.doi.org/10.1137/140985810}
  {\path{doi:10.1137/140985810}}.

\bibitem{Ghazaryan:2015}
A.~Ghazaryan, V.~Manukian, and S.~Schecter.
\newblock Travelling waves in the {H}olling-{T}anner model with weak diffusion.
\newblock {\em Proc. R. Soc. Lond. Ser. A}, 471(2177):20150045, 16, 2015.
\newblock \href {http://dx.doi.org/10.1098/rspa.2015.0045}
  {\path{doi:10.1098/rspa.2015.0045}}.

\bibitem{Haney:2018}
S.~D. Haney and A.~M. Siepielski.
\newblock Tipping points in resource abundance drive irreversible changes in
  community structure.
\newblock {\em Am. Nat.}, 191(5):668--675, May 2018.
\newblock \href {http://dx.doi.org/10.1086/697045} {\path{doi:10.1086/697045}}.

\bibitem{Hastings:2004}
A.~Hastings.
\newblock Transients: the key to long-term ecological understanding?
\newblock {\em Trends Ecol. Evol.}, 19(1):39 -- 45, 2004.
\newblock \href {http://dx.doi.org/https://doi.org/10.1016/j.tree.2003.09.007}
  {\path{doi:https://doi.org/10.1016/j.tree.2003.09.007}}.

\bibitem{Hastings:2018-Science}
A.~Hastings, K.~C. Abbott, K.~Cuddington, T.~Francis, G.~Gellner, Y.-C. Lai,
  A.~Morozov, S.~Petrovskii, K.~Scranton, and M.~L. Zeeman.
\newblock Transient phenomena in ecology.
\newblock {\em Science}, 361(6406), 2018.
\newblock \href {http://dx.doi.org/10.1126/science.aat6412}
  {\path{doi:10.1126/science.aat6412}}.

\bibitem{Hsu:2009}
S.-B. Hsu and J.~Shi.
\newblock Relaxation oscillation profile of limit cycle in predator-prey
  system.
\newblock {\em Discrete Contin. Dyn. Syst. Ser. B}, 11(4):893--911, 2009.
\newblock \href {http://dx.doi.org/10.3934/dcdsb.2009.11.893}
  {\path{doi:10.3934/dcdsb.2009.11.893}}.

\bibitem{Hsu:2016}
T.-H. Hsu.
\newblock Viscous singular shock profiles for a system of conservation laws
  modeling two-phase flow.
\newblock {\em J. Differential Equations}, 261(4):2300--2333, 2016.
\newblock \href {http://dx.doi.org/10.1016/j.jde.2016.04.034}
  {\path{doi:10.1016/j.jde.2016.04.034}}.

\bibitem{Hsu:2017}
T.-H. Hsu.
\newblock On bifurcation delay: an alternative approach using geometric
  singular perturbation theory.
\newblock {\em J. Differential Equations}, 262(3):1617--1630, 2017.
\newblock \href {http://dx.doi.org/10.1016/j.jde.2016.10.022}
  {\path{doi:10.1016/j.jde.2016.10.022}}.

\bibitem{Hsu:2019siads}
T.-H. Hsu.
\newblock Number and stability of relaxation oscillations for predator-prey
  systems with small death rates.
\newblock {\em SIAM J. Appl. Dyn. Syst.}, 18(1):33--67, 2019.
\newblock \href {http://dx.doi.org/10.1137/18M1166705}
  {\path{doi:10.1137/18M1166705}}.

\bibitem{Hsu:2019criterion}
T.-H. Hsu and G.~S.~K. Wolkowicz.
\newblock A criterion for the existence of relaxation oscillations with
  applications to predator-prey systems and an epidemic model.
\newblock {\em Discrete Contin. Dyn. Syst. Ser. B}, 2019.
\newblock \href {http://dx.doi.org/10.3934/dcdsb.2019219}
  {\path{doi:10.3934/dcdsb.2019219}}.

\bibitem{Huzak:2018}
R.~Huzak.
\newblock Predator-prey systems with small predator's death rate.
\newblock {\em Electron. J. Qual. Theory Differ. Equ.}, 86:1--16, 2018.
\newblock \href {http://dx.doi.org/10.14232/ejqtde.2018.1.86}
  {\path{doi:10.14232/ejqtde.2018.1.86}}.

\bibitem{Iuorio:2019}
A.~Iuorio, N.~Popovi\'{c}, and P.~Szmolyan.
\newblock Singular perturbation analysis of a regularized {MEMS} model.
\newblock {\em SIAM J. Appl. Dyn. Syst.}, 18(2):661--708, 2019.
\newblock \href {http://dx.doi.org/10.1137/18M1197552}
  {\path{doi:10.1137/18M1197552}}.

\bibitem{Jones:1995}
C.~K. R.~T. Jones.
\newblock Geometric {S}ingular {P}erturbation {T}heory.
\newblock In {\em Dynamical systems ({M}ontecatini {T}erme, 1994)}, Lecture
  Notes in Math. Vol. 1609, pages 44--118. Springer, Berlin, 1995.
\newblock \href {http://dx.doi.org/10.1007/BFb0095239}
  {\path{doi:10.1007/BFb0095239}}.

\bibitem{Jones:2009}
C.~K. R.~T. Jones and S.-K. Tin.
\newblock Generalized exchange lemmas and orbits heteroclinic to invariant
  manifolds.
\newblock {\em Discrete Contin. Dyn. Syst. Ser. S}, 2(4):967--1023, 2009.
\newblock \href {http://dx.doi.org/10.3934/dcdss.2009.2.967}
  {\path{doi:10.3934/dcdss.2009.2.967}}.

\bibitem{Khibnik:1997}
A.~I. Khibnik and A.~S. Kondrashov.
\newblock Three mechanisms of red queen dynamics.
\newblock {\em Proc. R. Soc. Lond. Ser. B}, 264(1384):1049--1056, 1997.
\newblock \href {http://dx.doi.org/10.1098/rspb.1997.0145}
  {\path{doi:10.1098/rspb.1997.0145}}.

\bibitem{Kosiuk:2011}
I.~Kosiuk and P.~Szmolyan.
\newblock Scaling in singular perturbation problems: blowing up a relaxation
  oscillator.
\newblock {\em SIAM J. Appl. Dyn. Syst.}, 10(4):1307--1343, 2011.
\newblock \href {http://dx.doi.org/10.1137/100814470}
  {\path{doi:10.1137/100814470}}.

\bibitem{Krupa:2001-SIMA}
M.~Krupa and P.~Szmolyan.
\newblock Extending geometric singular perturbation theory to nonhyperbolic
  points---fold and canard points in two dimensions.
\newblock {\em SIAM J. Math. Anal.}, 33(2):286--314, 2001.
\newblock \href {http://dx.doi.org/10.1137/S0036141099360919}
  {\path{doi:10.1137/S0036141099360919}}.

\bibitem{Krupa:2001-JDE}
M.~Krupa and P.~Szmolyan.
\newblock Relaxation oscillation and canard explosion.
\newblock {\em J. Differential Equations}, 174(2):312--368, 2001.
\newblock \href {http://dx.doi.org/10.1006/jdeq.2000.3929}
  {\path{doi:10.1006/jdeq.2000.3929}}.

\bibitem{Kuehn:2016}
C.~Kuehn.
\newblock {\em Multiple {T}ime {S}cale {D}ynamics}.
\newblock Appl. Math. Sci. Vol. 191. Springer, Cham, 2015.
\newblock \href {http://dx.doi.org/10.1007/978-3-319-12316-5}
  {\path{doi:10.1007/978-3-319-12316-5}}.

\bibitem{Li:2013}
C.~Li and H.~Zhu.
\newblock Canard cycles for predator-prey systems with {H}olling types of
  functional response.
\newblock {\em J. Differential Equations}, 254(2):879--910, 2013.
\newblock \href {http://dx.doi.org/10.1016/j.jde.2012.10.003}
  {\path{doi:10.1016/j.jde.2012.10.003}}.

\bibitem{Li:2016}
M.~Y. Li, W.~Liu, C.~Shan, and Y.~Yi.
\newblock Turning points and relaxation oscillation cycles in simple epidemic
  models.
\newblock {\em SIAM J. Appl. Math.}, 76(2):663--687, 2016.
\newblock \href {http://dx.doi.org/10.1137/15M1038785}
  {\path{doi:10.1137/15M1038785}}.

\bibitem{Lin:1989}
X.-B. Lin.
\newblock Shadowing lemma and singularly perturbed boundary value problems.
\newblock {\em SIAM J. Appl. Math.}, 49(1):26--54, 1989.
\newblock \href {http://dx.doi.org/10.1137/0149002}
  {\path{doi:10.1137/0149002}}.

\bibitem{Liu:2000}
W.~Liu.
\newblock Exchange lemmas for singular perturbation problems with certain
  turning points.
\newblock {\em J. Differential Equations}, 167(1):134--180, 2000.
\newblock \href {http://dx.doi.org/10.1006/jdeq.2000.3778}
  {\path{doi:10.1006/jdeq.2000.3778}}.

\bibitem{Liu:2003}
W.~Liu, D.~Xiao, and Y.~Yi.
\newblock Relaxation oscillations in a class of predator-prey systems.
\newblock {\em J. Differential Equations}, 188(1):306--331, 2003.
\newblock \href {http://dx.doi.org/10.1016/S0022-0396(02)00076-1}
  {\path{doi:10.1016/S0022-0396(02)00076-1}}.

\bibitem{Manukian:2009}
V.~Manukian and S.~Schecter.
\newblock Travelling waves for a thin liquid film with surfactant on an
  inclined plane.
\newblock {\em Nonlinearity}, 22(1):85--122, 2009.
\newblock \href {http://dx.doi.org/10.1088/0951-7715/22/1/006}
  {\path{doi:10.1088/0951-7715/22/1/006}}.

\bibitem{Mishchenko:1994}
E.~F. Mishchenko, Y.~S. Kolesov, A.~Y. Kolesov, and N.~K. Rozov.
\newblock {\em Asymptotic methods in singularly perturbed systems}.
\newblock Monographs in Contemporary Mathematics. Consultants Bureau, New York,
  1994.
\newblock Translated from the Russian by Irene Aleksanova.
\newblock \href {http://dx.doi.org/10.1007/978-1-4615-2377-2}
  {\path{doi:10.1007/978-1-4615-2377-2}}.

\bibitem{Piltz:2018}
S.~H. Piltz, L.~Harhanen, M.~A. Porter, and P.~K. Maini.
\newblock Inferring parameters of prey switching in a 1 predator-2 prey
  plankton system with a linear preference tradeoff.
\newblock {\em J. Theoret. Biol.}, 456:108--122, 2018.
\newblock \href {http://dx.doi.org/10.1016/j.jtbi.2018.07.005}
  {\path{doi:10.1016/j.jtbi.2018.07.005}}.

\bibitem{Piltz:2014}
S.~H. Piltz, M.~A. Porter, and P.~K. Maini.
\newblock Prey switching with a linear preference trade-off.
\newblock {\em SIAM J. Appl. Dyn. Syst.}, 13(2):658--682, 2014.
\newblock \href {http://dx.doi.org/10.1137/130910920}
  {\path{doi:10.1137/130910920}}.

\bibitem{Piltz:2017}
S.~H. Piltz, F.~Veerman, P.~K. Maini, and M.~A. Porter.
\newblock A predator-2 prey fast-slow dynamical system for rapid predator
  evolution.
\newblock {\em SIAM J. Appl. Dyn. Syst.}, 16(1):54--90, 2017.
\newblock \href {http://dx.doi.org/10.1137/16M1068426}
  {\path{doi:10.1137/16M1068426}}.

\bibitem{Rinaldi:1992}
S.~Rinaldi and S.~Muratori.
\newblock Slow-fast limit cycles in predator-prey models.
\newblock {\em Ecol. Model.}, 61(3):287 -- 308, 1992.
\newblock \href {http://dx.doi.org/10.1016/0304-3800(92)90023-8}
  {\path{doi:10.1016/0304-3800(92)90023-8}}.

\bibitem{Schecter:2004-JDE}
S.~Schecter.
\newblock Existence of {D}afermos profiles for singular shocks.
\newblock {\em J. Differential Equations}, 205(1):185--210, 2004.
\newblock \href {http://dx.doi.org/10.1016/j.jde.2004.06.013}
  {\path{doi:10.1016/j.jde.2004.06.013}}.

\bibitem{Schecter:2008b}
S.~Schecter.
\newblock Exchange lemmas. {II}. {G}eneral exchange lemma.
\newblock {\em J. Differential Equations}, 245(2):411--441, 2008.
\newblock \href {http://dx.doi.org/10.1016/j.jde.2007.10.021}
  {\path{doi:10.1016/j.jde.2007.10.021}}.

\bibitem{Schecter:2004-JDDE}
S.~Schecter and P.~Szmolyan.
\newblock Composite waves in the {D}afermos regularization.
\newblock {\em J. Dynam. Differential Equations}, 16(3):847--867, 2004.
\newblock \href {http://dx.doi.org/10.1007/s10884-004-6698-2}
  {\path{doi:10.1007/s10884-004-6698-2}}.

\bibitem{Shen:2019}
J.~Shen, C.-H. Hsu, and T.-H. Yang.
\newblock Fast--slow dynamics for intraguild predation models with evolutionary
  effects.
\newblock {\em J. Dynam. Differential Equations}, pages 1--26, 2019.
\newblock \href {http://dx.doi.org/10.1007/s10884-019-09744-3}
  {\path{doi:10.1007/s10884-019-09744-3}}.

\bibitem{Soto:2001}
C.~Soto-Trevi\~{n}o.
\newblock A geometric method for periodic orbits in singularly-perturbed
  systems.
\newblock In {\em {M}ultiple-{T}ime-{S}cale {D}ynamical {S}ystems
  ({M}inneapolis, {MN}, 1997)}, IMA Math. Appl. Vol. 122, pages 141--202.
  Springer, New York, 2001.
\newblock \href {http://dx.doi.org/10.1007/978-1-4613-0117-2_6}
  {\path{doi:10.1007/978-1-4613-0117-2_6}}.

\bibitem{Szmolyan:2004}
P.~Szmolyan and M.~Wechselberger.
\newblock Relaxation oscillations in {$\Bbb R^3$}.
\newblock {\em J. Differential Equations}, 200(1):69--104, 2004.
\newblock \href {http://dx.doi.org/10.1016/j.jde.2003.09.010}
  {\path{doi:10.1016/j.jde.2003.09.010}}.

\bibitem{Tin:1994}
S.-K. Tin, N.~Kopell, and C.~Jones.
\newblock Invariant manifolds and singularly perturbed boundary value problems.
\newblock {\em SIAM J. Numer. Anal.}, 31(6):1558--1576, 1994.
\newblock \href {http://dx.doi.org/10.1137/0731081}
  {\path{doi:10.1137/0731081}}.

\bibitem{ChengWang:2018}
C.~Wang and X.~Zhang.
\newblock Stability loss delay and smoothness of the return map in slow-fast
  systems.
\newblock {\em SIAM J. Appl. Dyn. Syst.}, 17(1):788--822, 2018.
\newblock \href {http://dx.doi.org/10.1137/17M1130010}
  {\path{doi:10.1137/17M1130010}}.

\bibitem{Wysham:2008}
D.~B. Wysham and A.~Hastings.
\newblock Sudden shifts in ecological systems: Intermittency and transients in
  the coupled ricker population model.
\newblock {\em Bull. Math. Biol.}, 70(4):1013--1031, May 2008.
\newblock \href {http://dx.doi.org/10.1007/s11538-007-9288-8}
  {\path{doi:10.1007/s11538-007-9288-8}}.

\end{thebibliography}

\end{document}